\documentclass[onefignum,onetabnum]{siamart190516}

\usepackage{}
\usepackage{graphicx}
\usepackage{multirow,bigstrut}
\usepackage{amsmath,amsfonts,amssymb}
\usepackage{mathrsfs}
\usepackage{color}
\usepackage{upgreek}
\usepackage{bm}
\usepackage{verbatim}
\usepackage{mathtools}
\usepackage{calligra}
\usepackage[T1]{fontenc}
\usepackage{subfigure}
\usepackage{pgfplots}
\usepackage{amsmath,accents}

\usepackage{tikz}
  \usetikzlibrary{calc,shapes,decorations,decorations.text, mindmap,shadings,patterns,matrix,arrows,intersections,automata,backgrounds}
  
\usetikzlibrary{shapes,arrows}

\usepackage{verbatim}
\usepackage{exscale}
\usepackage{relsize}

\newtheorem{assumption}[theorem]{Assumption}
\newtheorem{rem}[theorem]{\hspace{1mm}Remark}

\DeclareMathAlphabet{\mathdutchcal}{U}{dutchcal}{m}{n}
\SetMathAlphabet{\mathdutchcal}{bold}{U}{dutchcal}{b}{n}
\DeclareMathAlphabet{\mathdutchbcal}{U}{dutchcal}{b}{n}

\usepackage[T1]{fontenc}
\DeclareFontFamily{T1}{calligra}{}
\DeclareFontShape{T1}{calligra}{m}{n}{<->s*[1.44]callig20}{}
\DeclareMathAlphabet\mathcalligra   {T1}{calligra} {m} {n}
\DeclareMathAlphabet\mathzapf       {T1}{pzc} {mb} {it}
\DeclareMathAlphabet\mathchorus     {T1}{qzc} {m} {n}
\DeclareMathAlphabet\mathrsfso      {U}{rsfso}{m}{n}

\makeatletter
\tikzset{
        hatch distance/.store in=\hatchdistance,
        hatch distance=5pt,
        hatch thickness/.store in=\hatchthickness,
        hatch thickness=5pt
        }
\pgfdeclarepatternformonly[\hatchdistance,\hatchthickness]{north east hatch}
    {\pgfqpoint{-1pt}{-1pt}}
    {\pgfqpoint{\hatchdistance}{\hatchdistance}}
    {\pgfpoint{\hatchdistance-1pt}{\hatchdistance-1pt}}%
    {
        \pgfsetcolor{\tikz@pattern@color}
        \pgfsetlinewidth{\hatchthickness}
        \pgfpathmoveto{\pgfqpoint{0pt}{0pt}}
        \pgfpathlineto{\pgfqpoint{\hatchdistance}{\hatchdistance}}
        \pgfusepath{stroke}
    }
\makeatother

\newcommand{\vertiii}[1]{{\left\vert\kern-0.25ex\left\vert\kern-0.25ex\left\vert #1 
    \right\vert\kern-0.25ex\right\vert\kern-0.25ex\right\vert}}

\newcommand{\cm}[1]{{\color{black}{#1}}}
 
\begin{document}
 

\headers{Multiscale GFEM for heterogeneous Helmholtz problems}{C. P. Ma, C. Alber, and R. Scheichl} 
 
\title{Wavenumber explicit convergence of a multiscale \cm{generalized finite element method} for heterogeneous Helmholtz problems}

\author{Chupeng Ma\thanks 
    {Institute of Scientific Research, Great Bay University, Songshan Lake International Innovation Entrepreneurship A5, Dongguan 523000, China (\email{chupeng.ma@gbu.edu.cn}).}    
 \and C. Alber\thanks{Institute for Applied Mathematics and Interdisciplinary Center for Scientific Computing, Heidelberg University, Im Neuenheimer Feld 205, Heidelberg 69120, Germany (\email{christianalber1997@gmail.com}, \email{r.scheichl@uni-heidelberg.de}).} \and R. Scheichl\footnotemark[2]
}
\maketitle

\begin{abstract}
In this paper, a generalized finite element method (GFEM) with optimal local approximation spaces for solving high-frequency heterogeneous Helmholtz problems is systematically studied. The local spaces are built from selected eigenvectors of carefully designed local eigenvalue problems defined on generalized harmonic spaces. At both continuous and discrete levels, $(i)$ wavenumber explicit and nearly exponential decay rates for local and global approximation errors are obtained without any assumption on the size of subdomains; $(ii)$ a quasi-optimal convergence of the method is established by assuming that the size of subdomains is $O(1/k)$ ($k$ is the wavenumber). A novel resonance effect between the wavenumber and the dimension of local spaces on the decay of error with respect to the oversampling size is implied by the analysis. Furthermore, for fixed dimensions of local spaces, the discrete local errors are proved to converge as $h\rightarrow 0$ ($h$ denoting the mesh size) towards the continuous local errors. The method at the continuous level extends the plane wave partition of unity method [I. Babuska and J. M. Melenk, Int.\;J.\;Numer.\;Methods Eng., 40 (1997), pp.~727--758] to the heterogeneous-coefficients case, and at the discrete level, it delivers an efficient non-iterative domain decomposition method for solving discrete Helmholtz problems resulting from standard FE discretizations. Numerical results are provided to confirm the theoretical analysis and to validate the proposed method.
\end{abstract}



\section{Introduction}\label{sec-1}
The Helmholtz equation models wave propagation and scattering phenomena in the frequency domain, and arises in a variety of science and engineering applications, including seismic imaging, medical ultrasound technologies, and underwater acoustics. It is well known that due to the so called \textit{pollution effect}, solving the Helmholtz equation with low-order finite element methods (FEMs) needs a much higher mesh resolution than typically required for a meaningful representation of the solution in the finite element spaces. Other standard numerical methods also suffer from a similar problem. In the high frequency regime, such discretizations result in very large scale and strongly indefinite linear systems of equations which are difficult to solve by classical methods. Significant research efforts have been devoted to addressing these challenging problems, which mainly focus on two directions: \cm{high-order and non-classical discretization schemes}, and efficient methods for solving the resulting linear systems from standard low-order FE approximations.

\cm{With an aim to overcome the shortcoming of low-order Galerkin FEMs, various discretization schemes have been developed for the Helmholtz equation over the past decades. High-order discretization methods \cite{ainsworth2004discrete,ainsworth2009dispersive} have been shown to be effective in alleviating 
the pollution effect. In particular, for the $hp$-FEM \cite{melenk2010convergence,melenk2011wavenumber}, quasi optimality was proved under the conditions that the polynomial degree $p$ is at least $O(\log k)$ ($k$ denoting the wave number) and that $kh/p$ is sufficiently small ($h$ denoting the mesh size). Apart from high-order discretization methods, non-classical Galerkin methods have also been extensively studied.} One class of such methods are based on variational formulations different from classical Galerkin methods; see, e.g., \cite{chen2013hybridizable,feng2009discontinuous,moiola2014helmholtz,wu2014pre}. Another important class of non-classical FEMs are Trefftz-type methods which use (local) solutions of the Helmholtz equation as basis functions; see \cite{hiptmair2016survey} for a survey. A popular choice of Trefftz-type basis functions are plane waves and associated methods include the ultraweak variational formulation \cite{cessenat1998application}, the plane wave partition of unity method \cite{babuvska1997partition}, the plane wave discontinuous Galerkin method \cite{hiptmair2011plane}, the least-squares FEM \cite{monk1999least}, and the discontinuous enrichment method \cite{farhat2003discontinuous}, to cite a few. Due to the use of operator-adapted basis functions, Trefftz-type methods usually need much fewer degrees of freedom than conventional FEMs to achieve the same accuracy. 

However, all the aforementioned \cm{high-order and non-classical FEMs} are typically developed for \cm{homogeneous media}, and it is not straightforward to extend them to Helmholtz problems in heterogeneous media. \cm{Although recently similar results as in the pioneering works on $hp$-FEMs \cite{melenk2010convergence,melenk2011wavenumber} have also been established for (piecewise) smooth coefficients \cite{chaumont2020wavenumber,lafontaine2022wavenumber}, as far as we know, no theoretical justification exists for general heterogeneous Helmholtz problems, for instance, those with multiscale features.} For Trefftz-type methods, in general, local solutions of heterogeneous Helmholtz equations are not available. \cm{Although several attempts have been made to construct local approximate solutions \cite{imbert2014generalized,tezaur2014discontinuous}, they are typically limited to piecewise smooth coefficients.} In recent years, there has been an increasing interest among the numerical homogenization community in developing numerical multiscale methods for Helmholtz problems with or without heterogeneous coefficients, e.g., (generalized) multiscale FEMs \cite{chen2021exponentially,fu2021edge,fu2017fast,fu2021wavelet}, LOD type methods \cite{brown2017multiscale,hauck2021multi,peterseim2017eliminating,peterseim2020computational}, the heterogeneous multiscale method \cite{ohlberger2018new}, the multiscale asymptotic method \cite{cao2002multiscale}, and the multiscale hybrid-mixed method \cite{chaumont2020multiscale}. These multiscale methods are usually based on solving some local problems numerically to get basis functions that capture the wave characteristics and local media information. The associated analysis, including local approximation error estimates and the quasi-optimal convergence analysis, were typically performed under \cm{the scale resolution condition}, i.e., the mesh size of the coarse grid is $O(1/k)$, e.g., in \cite{chen2021exponentially,fu2021wavelet,peterseim2017eliminating}. In fact, under this condition, the indefiniteness of the local problems largely disappears and many analysis techniques developed for positive definite problems can be applied. For multiscale methods \cm{beyond this condition}, which is of more interest for practical applications, very little analysis of the error and of the effect of the wavenumber is available in the literature.

Although using high-order or non-classical FEMs for the Helmholtz equation can yield a dramatic improvement in efficiency, due to their simplicity, low-order classical FEMs are still widely used. In this scenario, the efficient solution of the resulting large linear systems becomes a focus of research. For discrete Helmholtz problems of very large size, direct methods are in general prohibitively expensive, and classical iterative methods suffer from the problem of slow convergence \cite{ernst2012difficult}. Over the past two decades, robust preconditioning of the Helmholtz equation has been extensively studied and many novel preconditioners have been proposed; see \cite{gander2019class} for a review. Here we put a special emphasis on domain decomposition methods (DDMs), as they are a natural choice for use on parallel computers. Simple extensions of state-of-the-art techniques for symmetric positive definite problems to indefinite and non-self-adjoint problems have been shown to be inefficient \cite{ernst2012difficult}. To obtain an efficient domain decomposition preconditioner for the Helmholtz equation, two key ingredients are needed: transmission conditions and a coarse space, suitably adapted to the characteristics of the problem. We particularly highlight the DtN and GenEO spectral coarse spaces \cite{bootland2021comparison}, constructed by selected modes of local eigenvalue problems adapted to the Helmholtz equation.  Apart form practical difficulties, few domain decomposition preconditioners (and also other preconditioners) for the Helmholtz equation are amenable to rigorous analysis due to the indefinite and non-Hermitian nature of the underlying problem.

In this paper, we consider the numerical solution of high-frequency heterogeneous Helmholtz problems and deal with the two above focused topics, nonstandard FEM discretizations and efficient solvers, within the unified framework of the Multiscale Spectral Generalized Finite Element Method (MS-GFEM) \cite{babuska2011optimal,babuvska2020multiscale,ma2021novel,ma2021error}. Built on the GFEM \cite{melenk1995generalized} which constructs the trial space by gluing local approximation spaces together with a partition of unity, the MS-GFEM achieves a high efficiency by building optimal local approximation spaces from selected eigenvectors of carefully designed local eigenvalue problems. At the continuous level, the local eigenproblems are defined on generalized harmonic spaces that consist of local solutions to the governing equation with vanishing source term, and thus the method is an extension of the plane wave partition of unity method \cite{babuvska1997partition} to the heterogeneous coefficient case. Wavenumber-explicit and nearly exponential decay rates for the local approximation errors are derived without the scale resolution condition, and a nearly exponential error decay for the global approximation is obtained under some general assumptions on the stability of local Helmholtz problems. The local approximation errors are a posteriori computable from the local spectra. Our analysis implies the presence of a \textit{resonance} effect between the wavenumber and the dimension of the local approximation spaces, which affects the error decay with respect to the oversampling size; see Remark~\ref{rem:3-4}. In particular, it is shown that the (approximation) error of the method in the high-frequency regime does not always decay with increasing oversampling size, in contrast to the positive definite case \cite{ma2021novel}. A quasi-optimal global convergence rate for the method is established under the scale resolution condition. Compared to the usual Trefftz methods, a second-level discretization for solving the local problems is needed in our method. However, these local problems can be solved entirely in parallel and the local basis functions can be reused. Therefore, a significant gain in efficiency can be expected when solving problems with many different wave sources.

At the discrete level, the MS-GFEM delivers a non-iterative DDM for solving linear systems arising from standard FE discretizations of heterogeneous Helmholtz problems. The optimal local spaces for approximating the standard FE solution are constructed analogously to the continuous level by solving discrete local eigenvalue problems. Similar local and global error estimates for the discrete method are obtained under assumptions akin to those in the continuous setting. Furthermore, a proof of the convergence of the discrete eigenvalues as $h\rightarrow 0$ to those of the continuous problems is given ($h$ denoting the mesh size), indicating that the discrete local approximation errors converge towards the continuous ones as $h\rightarrow 0$ for fixed dimensions of local spaces. As in the case of classical two-level DDMs for the Helmholtz equation, the local approximation spaces and the boundary conditions of the local problems in the method are suitably adapted to the Helmholtz equation, when compared with those for positive definite problems \cite{ma2021novel,ma2021error}; see Remark~\ref{rem:3-1}. However, although both methods are based on solving some local problems and a global coarse problem, the discrete MS-GFEM can solve the problem in one shot without iteration.

Our work distinguishes itself from previous works on multiscale methods for heterogeneous Helmholtz problems in three aspects: in contrast to previous studies, a non-standard FEM discretization and an efficient method for solving discrete Helmholtz problems are unified under the same mathematical framework. Second, all the local analysis of the method, particularly the local approximation error estimates, hold without a scale resolution condition, and so do the resulting global approximation results. Finally, the effect of the wavenumber on the error of the method is systematically investigated both theoretically and numerically, especially the aforementioned error resonance phenomenon. It is worth noting that although the scale resolution condition is required for proving quasi optimality of the method, numerical results show that it is not necessary in practice, just as the condition "$k^{2}h$ is small" for the standard linear FEM for the Helmholtz equation \cite{melenk1995generalized} is not required in practice either. In fact, in our numerical experiments, we obtain excellent results for $kH\approx 13$, i.e., about two wavelengths per subdomain. Also, with the desirable local approximation spaces, it may be possible to use a least squares method \cite{monk1999least}, instead of the partition of unity method, to yield a coercive formulation, and thus get rid of this condition.

The remainder of this paper is structured as follows. In \cref{sec-2}, we introduce the model Helmholtz problem considered in this paper and describe the continuous MS-GFEM for solving the problem. In \cref{sec-3}, we prove the local and global error estimates for the continuous method. The discrete MS-GFEM for solving the discrete Helmholtz problem is presented in the first part of \cref{sec-4}, followed by some technical tools used in the subsequent analysis. We focus on the convergence analysis of the discrete MS-GFEM in \cref{sec-5}, including the local and global error estimates and the convergence of the eigenvalues of the local eigenproblems. Numerical results are reported in \cref{sec-6} to support the theoretical analysis and to validate the method.

\section{Problem formulation and the continuous MS-GFEM}\label{sec-2}
\subsection{Model Helmholtz problem}
Let $\Omega\subset \mathbb{R}^{d}$ $(d=2,3)$ be a bounded Lipschitz domain with boundary $\Gamma$. Given $k>0$, we consider the following heterogeneous Helmholtz equation with mixed boundary conditions: Find $\cm{u^{\mathdutchcal{e}}}:\Omega\rightarrow\mathbb{C}$ such that
\begin{equation}\label{eq:1-1}
\left\{
\begin{array}{lll}
{\displaystyle -{\rm div}(A\nabla \cm{u^{\mathdutchcal{e}}}) - k^{2}V^{2}\cm{u^{\mathdutchcal{e}}}= f,\,\quad {\rm in}\;\, \Omega }\\[2mm]
{\displaystyle  \;\,\qquad \quad \quad \;\;\qquad \qquad \cm{u^{\mathdutchcal{e}}} = 0,\quad \,{\rm on}\;\,\Gamma_{D}}\\[2mm]
{\displaystyle \qquad \,\;A\nabla \cm{u^{\mathdutchcal{e}}}\cdot {\bm n} - {\rm i}k\beta \cm{u^{\mathdutchcal{e}}}=g,  \quad \,{\rm on}\;\,\Gamma_{R},}
\end{array}
\right.
\end{equation}
where ${\bm n}$ denotes the outward unit normal to $\Gamma$, $\Gamma_{D}\cap\Gamma_{R} = \emptyset$, and $\Gamma = \overline{\Gamma_{D}}\cup\overline{\Gamma_{R}}$. We suppose that $|\Gamma_{R}|>0$. Throughout the paper, we make the following assumptions:

\begin{assumption}\label{ass:1-1} 
\begin{itemize}
\item[(i)] $A \in (L^{\infty}(\Omega))^{d\times d}$ is pointwise symmetric and there exists $0< a_{\rm min} < a_{\rm max}<\infty$ such that
\begin{equation}\label{eq:1-1-0}
a_{\rm min} |{\bm \xi}|^{2} \leq A({\bm x}){\bm \xi}\cdot{\bm \xi} \leq a_{\rm max}  |{\bm \xi}|^{2},\quad \forall {\bm \xi}\in \mathbb{R}^{d},\quad {\bm x} \in\Omega;
\end{equation}
\item[(ii)] $V \in L^{\infty}(\Omega)$ and there exists $0<V_{\rm min}<V_{\rm max}<\infty$ such that $V_{\rm min}\leq V({\bm x})\leq V_{\rm max}$ for all ${\bm x}\in \Omega$;

\vspace{1ex}
\item[(iii)] $f\in L^{2}(\Omega)$, $g\in L^{2}(\Gamma_{R})$ and $\beta\in L^{\infty}(\Gamma_{R})$. \cm{Moreover, $\beta({\bm x})>0$ for all ${\bm x}\in \Gamma_{R}$}.
\end{itemize}
\end{assumption}
We note that the results in this paper hold for a wider class of $f$ and $g$ in a weaker dual space, but we omit this extension for ease of presentation.

Defining the space
\begin{equation}
H^{1}_{D}(\Omega) = \big\{v\in H^{1}(\Omega)\;:\; v = 0 \;\;{\rm on}\;\,\Gamma_{D}\big\}
\end{equation}
and the sesquilinear form $\mathcal{B}: H^{1}_{D}(\Omega)\times H^{1}_{D}(\Omega) \rightarrow \mathbb{C}$ by
\begin{equation}
\mathcal{B}(u,v) = \int_{\Omega} \big(A\nabla u\cdot {\nabla \overline{v}}-k^{2}V^{2}u\overline{v}\big) \,d{\bm x} - {\rm i}k\int_{\Gamma_{R}} \beta u\overline{v}\,d{\bm s},\quad \forall u,\,v\in H^{1}_{D}(\Omega),
\end{equation}
the weak formulation of the problem \cref{eq:1-1} is to find $u^{\mathdutchbcal{\mathdutchcal{e}}}\in H_{D}^{1}(\Omega)$ such that 
\begin{equation}\label{eq:1-2}
\mathcal{B}(u^{{\mathdutchcal{e}}},v) = F(v):=\int_{\Gamma_{R}} g\overline{v} \,d{\bm s}+ \int_{\Omega}f\overline{v} \,d{\bm x},\quad \forall v\in H_{D}^{1}(\Omega).
\end{equation}
For later use, we introduce some local sesquilinear forms. Let $\omega$ be a subdomain of $\Omega$ and $u$, $v\in H^{1}(\omega)$. We define 
\begin{equation}\label{eq:1-3}
\begin{array}{lll}
{\displaystyle \mathcal{B}_{\omega}(u,v) = \int_{\omega} \big(A\nabla u\cdot {\nabla \overline{v}}-k^{2}V^{2}u\overline{v}\big) \,d{\bm x} - {\rm i}k\int_{\Gamma_{R}\cap \partial \omega} \beta u\overline{v}\,d{\bm s},}\\[3mm]
{\displaystyle  \mathcal{A}_{\omega}(u,v) = \int_{\omega} A\nabla u\cdot {\nabla \overline{v}}\,d{\bm x},\quad \mathcal{A}_{\omega,k}(u,v) = \int_{\omega} \big(A\nabla u\cdot {\nabla \overline{v}}+k^{2}V^{2}u\overline{v}\big)\,d{\bm x},}
\end{array}
\end{equation}
and
\begin{equation}\label{eq:1-4}
\begin{array}{lll}
{\displaystyle \Vert u \Vert_{\mathcal{A},\omega} =  \sqrt{\mathcal{A}_{\omega}(u,u)},\quad \Vert u \Vert_{\mathcal{A},\omega,k} = \sqrt{\mathcal{A}_{\omega,k}(u,u)}.}
\end{array}
\end{equation}
If $\omega=\Omega$, we simply write $\Vert u \Vert_{\mathcal{A}}$ ($\Vert u \Vert_{\mathcal{A},k}$). It can be proved \cite{melenk1995generalized} that there exists $C_{\mathcal{B}}>0$ independent of $k$ such that
\begin{equation}\label{eq:1-5}
|\mathcal{B}(u,v)|\leq C_{\mathcal{B}} \Vert u \Vert_{\mathcal{A},k} \Vert v \Vert_{\mathcal{A},k},\quad \forall u,\,v\in H^{1}(\Omega).
\end{equation}
We assume the well-posedness of the problem \cref{eq:1-2} as follows.
\begin{assumption}\label{ass:1-2}
For any $f\in L^{2}(\Omega)$ and $g\in L^{2}(\Gamma_{R})$, the problem \cref{eq:1-2} has a unique solution $u^{{\mathdutchcal{e}}}\in H_{D}^{1}(\Omega)$, and there exists $C_{\rm stab}(k)>0$ depending polynomially on $k$ such that
\begin{equation}\label{eq:1-6}
 \Vert u^{{\mathdutchcal{e}}}\Vert_{\mathcal{A},k}\leq C_{\rm stab}(k) (\Vert f\Vert_{L^{2}(\Omega)} + \Vert g\Vert_{L^{2}(\Gamma_{R})}).
\end{equation}
\end{assumption}
\begin{rem}
It was proved in \cite{graham2020stability} that under \cref{ass:1-1}, for $d=2$, the problem \cref{eq:1-2} is uniquely solvable in $H_{D}^{1}(\Omega)$, and for $d=3$, some additional assumptions on the coefficient $A$ are required to obtain the unique solvability; see also \cite{chaumont2021scattering}. The condition of a polynomial growth in $k$ on the stability constant is satisfied for a wide class of problems; see \cite{brown2017multiscale, esterhazy2012stability,lafontaine2021most}. Some special instances \cite{chandler2020high} of an exponential growth in $k$, which makes the problem strongly unstable, are ruled out here.
\end{rem}

\subsection{Continuous MS-GFEM}
In this subsection, we present the continuous MS-GFEM for solving the problem \cref{eq:1-2}. For completeness, we first recall the GFEM. Let $\{ \omega_{i} \}_{i=1}^{M}$ be a collection of open subsets of $\Omega$ satisfying $\cup_{i=1}^{M} \omega_{i} = \Omega$ and a pointwise overlap condition:
\begin{equation}\label{eq:2-0}
\exists \,\zeta\in \mathbb{N}\quad\forall {\bm x}\in\Omega \quad {\rm card}\{i\;|\;{\bm x}\in \omega_{i} \}\leq\zeta.
\end{equation}
Let $\{ \chi_{i} \}_{i=1}^{M}$ be a partition of unity subordinate to the open cover satisfying
\begin{equation}\label{eq:2-0-1}
\begin{array}{lll}
{\displaystyle 0\leq \chi_{i}({\bm x})\leq 1,\quad \sum_{i=1}^{M}\chi_{i}({\bm x}) =1, \quad \forall {\bm x}\in \Omega,}\\[4mm]
{\displaystyle \chi_{i}({\bm x})= 0, \quad \forall {\bm x}\in \Omega\setminus\omega_{i}, \quad i=1,\cdots,M,}\\[2mm]
{\displaystyle \chi_{i}\in W^{1,\infty}(\omega_{i}),\;\;\Vert\nabla \chi_{i} \Vert_{L^{\infty}(\omega_i)} \leq \frac{C_{\chi}}{\mathrm{diam}\,(\omega_{i})},\quad i=1,\cdots,M.}
\end{array}
\end{equation}

For each $i=1,\cdots,M$, let $u^{p}_{i} \in H^{1}(\omega_{i})$ be a local particular function and $S_{n_{i}}(\omega_{i})\subset H^{1}(\omega_{i})$ be a local approximation space of dimension $n_{i}$ such that $u^{p}_{i}$ and all functions in $S_{n_{i}}(\omega_{i})$ vanish on $\partial \omega_{i}\cap \Gamma_{D}$. A key feature of the GFEM is to build the global particular function $u^{p}$ and the trial space $S_{n}(\Omega)$ by pasting the local particular functions and the local approximation spaces together using the partition of unity:
\begin{equation}\label{eq:2-0-2}
\begin{array}{lll}
{\displaystyle u^{p} = \sum_{i=1}^{M}\chi_{i}u^{p}_{i},\quad  S_{n}(\Omega) =\Big\{\sum_{i=1}^{M}\chi_{i}\phi_{i}\,:\, \phi_{i}\in S_{n_{i}}(\omega_{i})\Big\}. }
\end{array}
\vspace{-1ex}
\end{equation}
By the assumptions on the partition of unity in \cref{eq:2-0-1}, we see that $u^{p}\in H^{1}_{D}(\Omega)$ and $S_{n}(\Omega)\subset H^{1}_{D}(\Omega)$. The final step of the GFEM is the finite-dimensional Galerkin approximation: Find $u^{G}= u^{p} + u^{s}$ with $u^{s}\in S_{n}(\Omega)$ such that
\begin{equation}\label{eq:2-1}
\mathcal{B}(u^{G}, v) = F(v),\quad \forall v\in S_{n}(\Omega).
\end{equation}
Combining \cref{eq:1-2,eq:2-1} yields the Galerkin orthogonality:
\begin{equation}\label{eq:2-1-0}
\mathcal{B}(u^{\mathdutchcal{e}}-u^{G},\,v) = 0\quad \forall v\in S_{n}(\Omega).
\end{equation}

Obviously, the core of the GFEM is the construction of the local particular functions and of the local approximation spaces. In the MS-GFEM, the local particular functions are defined as solutions of local Helmholtz problems and the local approximation spaces are constructed by eigenvectors of local eigenvalue problems; see \cref{thm:2-2}. We detail the construction in what follows.
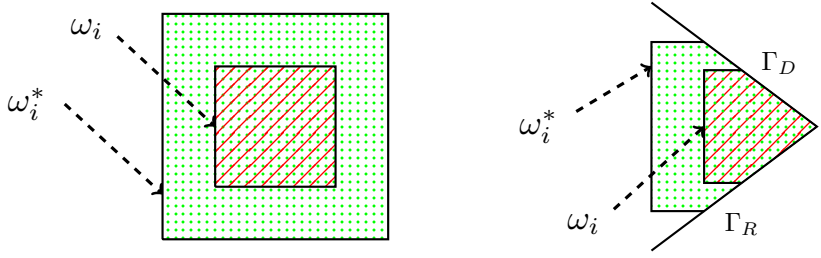
\begin{figure}\label{fig:2-1}
\centering
\begin{tikzpicture}
\draw[pattern=dots, pattern color=green] (0, 0) rectangle (3, 3);
\draw[pattern=north east hatch, hatch distance=2.5mm, hatch thickness=.5pt, pattern color=red] (0.7, 0.7) rectangle (2.3, 2.3); 
\draw[->, dashed, very thick] (-0.6,2.7)--(0.7,1.5);
\draw[thick](0.7,0.7)--(2.3,0.7)--(2.3,2.3)--(0.7,2.3)--(0.7,0.7);
\draw[thick](0,0)--(3,0)--(3,3)--(0,3)--(0,0);
\draw[->, dashed, very thick] (-1.3,1.8)--(0,0.6);
\node[scale=1.3] at (-1.0, 2.8) {$\omega_{i}$};
\node[scale=1.3] at (-1.8, 1.8) {$\omega^{\ast}_{i}$};

 \draw[pattern=north east hatch, hatch distance=2.5mm, hatch thickness=.5pt, pattern color=red] (7.7, 2.25)--(7.2, 2.25)--(7.2,0.75)--(7.7,0.75)--(8.7,1.5)--(7.7,2.25); 
 \draw[pattern=dots, pattern color=green] (7.2, 2.625)--(6.5, 2.625)--(6.5, 0.375)--(7.2,0.375)--(8.7,1.5)--(7.2, 2.625);
\draw[thick] (6.5,-0.15)--(8.7,1.5)--(6.5,3.15);
\draw[thick] (7.7, 2.25)--(7.2, 2.25)--(7.2,0.75)--(7.7,0.75);
\draw[thick] (7.2, 2.625)--(6.5, 2.625)--(6.5, 0.375)--(7.2,0.375);

\draw[->, dashed, very thick] (6.0,0.4)--(7.2,1.5);
\draw[->, dashed, very thick] (5.5,1.7)--(6.5,2.3);
   
\node[scale=1.3] at (5.6,0.2) {$\omega_{i}$};
\node[scale=1.3] at (5.0,1.5) {$\omega^{\ast}_{i}$};
\node at (8.2,2.35) {$\Gamma_{D}$};
\node at (7.7,0.2) {$\Gamma_{R}$};
\end{tikzpicture}
\caption{Illustration of an interior subdomain (left) and a subdomain intersecting the outer boundary (right) and the associated oversampling domains.}
\end{figure}
For a subdomain $\omega_{i}$, we introduce an \textit{oversampling domain} $\omega_{i}^{\ast}$ \cm{with a Lipschitz boundary} such that $\omega_{i} \subset \omega_{i}^{\ast}\subset\Omega$ as illustrated in \cref{fig:2-1}. On each $\omega_{i}^{\ast}$, we define 
\begin{equation}\label{eq:2-2}
\begin{array}{lll}
{\displaystyle H^{1}_{D}(\omega^{\ast}_{i}) = \big\{v\in H^{1}(\omega^{\ast}_{i})\;:\;v = 0\;\;{\rm on}\;\, \partial \omega^{\ast}_{i}\cap \Gamma_{D}\big\}}\\[2mm]
{\displaystyle H^{1}_{DI}(\omega^{\ast}_{i}) = \big\{v\in H^{1}_{D}(\omega^{\ast}_{i})\;:\;v = 0\;\;{\rm on}\;\, \partial \omega^{\ast}_{i}\cap \Omega\big\},}
\end{array}
\end{equation}
and the generalized harmonic space
\begin{equation}\label{eq:2-3}
H_{\mathcal{B}}(\omega^{\ast}_{i}) = \big\{ u\in H^{1}_{D}(\omega^{\ast}_{i})\;:\;  \mathcal{B}_{\omega^{\ast}_{i}}(u, v) = 0 \quad \forall v\in H^{1}_{DI}(\omega^{\ast}_{i})\big\}. 
\end{equation}
$H_{\mathcal{B}}(\omega^{\ast}_{i})$ consists of local solutions of the Helmholtz equation with vanishing source term. A general result on equivalent norms for Sobolev spaces (see, e.g., \cite[Chapter 2]{necas2011direct}) gives that there exists $C>0$, such that for any $u\in H_{\mathcal{B}}(\omega^{\ast}_{i})$,
\begin{equation}\label{eq:2-3-0}
\Vert u\Vert_{L^{2}(\omega_{i}^{\ast})}\leq C\Vert \nabla u \Vert_{L^{2}(\omega_{i}^{\ast})}.
\vspace{-1ex}
\end{equation}
Therefore, $\Vert\cdot\Vert_{\mathcal{A},\omega_{i}^{\ast}}$ is a norm on $H_{\mathcal{B}}(\omega^{\ast}_{i})$. Next we consider the following local Helmholtz problem on $\omega_{i}^{\ast}$:
\begin{equation}\label{eq:2-4}
\left\{
\begin{array}{lll}
{\displaystyle -{\rm div}(A\nabla \psi_{i}) - k^{2}V^{2}\psi_{i}= f,\,\quad {\rm in}\;\, \omega^{\ast}_{i} }\\[2mm]
{\displaystyle \qquad \qquad \,\;\;\;  \qquad \qquad \psi_{i} = 0, \quad\; {\rm on}\;\, \partial \omega^{\ast}_{i}\cap \Gamma_{D}}\\[2mm]
{\displaystyle \qquad \;A\nabla \psi_{i}\cdot {\bm n} - {\rm i}k\beta\psi_{i}=g,  \,\quad \,{\rm on}\;\,\partial \omega^{\ast}_{i} \cap \Gamma_{R}}\\[2mm]
{\displaystyle \qquad \;A\nabla \psi_{i}\cdot {\bm n} - {\rm i}kV\psi_{i}=0,\quad\; {\rm on}\;\, \partial \omega^{\ast}_{i}\cap \Omega.}
\end{array}
\right.
\end{equation}
The weak formulation of the problem \cref{eq:2-4} is given by: Find $\psi_{i}\in H^{1}_{D}(\omega^{\ast}_{i})$, such that for any $v\in H^{1}_{D}(\omega^{\ast}_{i})$,
\vspace{-1ex}
\begin{equation}\label{eq:2-4-0}
\mathcal{B}_{\omega_{i}^{\ast}}(\psi_{i},v) - {\rm i}k\int_{\partial \omega^{\ast}_{i}\cap \Omega}V\psi_{i}\overline{v}\,d{\bm s}= F_{\omega_{i}^{\ast}}(v):= \int_{\Gamma_{R}\cap \partial \omega_{i}^{\ast}} g\overline{v} \,d{\bm s}+ \int_{\omega_{i}^{\ast}}f\overline{v} \,d{\bm x}.
\vspace{-1.5ex}
\end{equation}
\begin{rem}\label{rem:3-1}
In contrast to the Dirichlet boundary conditions for local problems in the positive definite case \cite{ma2021novel}, we impose impedance boundary conditions on the artificial interior boundaries for the local Helmholtz problems to guarantee their unique solvability. \cm{Such boundary conditions (involving $V$) are commonly used as transmission conditions in domain decomposition methods for Helmholtz problems \cite{bootland2021comparison,gong2021domain}. However, any other interior boundary conditions that guarantee well-posedness of the local problems can also be used}.
\end{rem}
We assume the well-posedness of the local Helmholtz problem \cref{eq:2-4-0} as follows.
\begin{assumption}\label{ass:2-1}
For each $i=1,\cdots,M$, the problem \cref{eq:2-4-0} has a unique solution $\psi_{i}\in H^{1}_{D}(\omega^{\ast}_{i})$, and there exists a constant $C^{i}_{\rm stab}(k)$ depending polynomially on $k$ such that
\begin{equation}\label{eq:2-4-1}
\Vert \psi_{i}\Vert_{\mathcal{A},\omega_{i}^{\ast},k}\leq C^{i}_{\rm stab}(k)(\Vert f \Vert_{L^{2}(\omega_{i}^{\ast})} + \Vert g\Vert_{L^{2}(\partial \omega^{\ast}_{i} \cap \Gamma_{R})}).
\end{equation}
\end{assumption}
\cm{\begin{rem}
Following the lines of \cite[Proposition A.3]{chaumont2020multiscale}, it can be shown that the stability estimate \cref{eq:2-4-1} holds with $C^{i}_{\rm stab}(k) = O(1)$ if ${\rm diam}(\omega_i^{\ast}) = O(k^{-1})$.
\end{rem}}

Combining \cref{eq:1-2,eq:2-4-0}, we see that $u^{{\mathdutchcal{e}}}|_{\omega_{i}^{\ast}}-\psi_{i}\in H_{\mathcal{B}}(\omega^{\ast}_{i})$, where $u^{{\mathdutchcal{e}}}$ is the exact solution of the global problem. Therefore, the exact solution is locally decomposed into two parts, one being the solution of the local Helmholtz problem and another belonging to the generalized harmonic space $H_{\mathcal{B}}(\omega^{\ast}_{i})$. To approximate the latter part, we follow the lines of \cite{ma2021novel} to construct a finite-dimensional space that is optimal in an appropriate sense,  using the singular vectors of a compact operator involving the partition of unity function. To this end, we first give a novel identity and the resulting Caccioppoli-type inequality for functions in the generalized harmonic space, which plays a crucial role in the analysis of the continuous MS-GFEM.
\begin{lemma}\label{thm:2-1}
Assume that $\eta \in W^{1,\infty}(\omega_{i}^{\ast})$ satisfies $\eta({\bm x}) = 0$ on $\partial\omega_{i}^{\ast} \cap \Omega$. Then, for any $u, \,v\in H_{\mathcal{B}}(\omega_{i}^{\ast})$,
\begin{equation}\label{eq:2-5}
\mathcal{A}_{\omega_{i}^{\ast}}(\eta u,\eta v) = \int_{\omega_{i}^{\ast}}(A\nabla \eta \cdot \nabla \eta + k^{2}V^{2}\eta^{2}) u\overline{v}\,d{\bm x}.
\vspace{-1.5ex}
\end{equation}
In particular, 
\begin{equation}\label{eq:2-6}
\Vert \eta u \Vert_{\mathcal{A}, \omega_{i}^{\ast}} \leq \Big(a_{\rm max}^{1/2} \Vert \nabla \eta \Vert_{L^{\infty}(\omega_{i}^{\ast})} +kV_{\rm max}\Vert \eta \Vert_{L^{\infty}(\omega_{i}^{\ast})}\Big)\Vert u \Vert_{L^{2}(\omega_{i}^{\ast})},
\end{equation}
where $a_{\rm max}$ and $V_{\rm max}$ are the (spectral) upper bounds of $A$ and $V$ in \cref{ass:1-1}.
\end{lemma}
The proof is given in \cref{sec:A.3}.\pagebreak

Now we are ready to construct the desired optimal space for approximating a generalized harmonic function. We start by introducing the operator 
\begin{equation}\label{eq:2-12}
P_{i}: \big(H_{\mathcal{B}}(\omega_{i}^{\ast}), \,\Vert \cdot \Vert_{\mathcal{A},\omega_{i}^{\ast}}\big) \rightarrow \big(H_{DI}^{1}(\omega_{i}), \,\Vert \cdot \Vert_{\mathcal{A},\omega_{i},k}\big) \;\;{\rm such \;\,that}\;\; P_{i}(v) = \chi_{i}v,
\end{equation}
where $\chi_{i}$ is the partition of unity function supported on $\omega_{i}$. Note that here and after, we equip the spaces $H_{\mathcal{B}}(\omega_{i}^{\ast})$ and $H_{DI}^{1}(\omega_{i})$ with the norms $\Vert \cdot \Vert_{\mathcal{A},\omega_{i}^{\ast}}$ and $\Vert \cdot \Vert_{\mathcal{A},\omega_{i},k}$, respectively. It follows from \cref{thm:2-1} and the compact embedding $H^{1}(\omega_{i}^{\ast})\subset L^{2}(\omega_{i}^{\ast})$ that the operator $P_{i}$ is compact. Next we consider the following Kolmogorov $n$-width of the operator $P_{i}$:
\begin{equation}\label{eq:2-13}
\begin{array}{lll}
{\displaystyle d_{n}(\omega_{i},\omega_{i}^{\ast}) :=\inf_{Q(n)\subset H_{DI}^{1}(\omega_{i})}\sup_{u\in H_{\mathcal{B}}(\omega_{i}^{\ast})} \inf_{v\in Q(n)}\frac {\Vert P_{i}u-v\Vert_{\mathcal{A},\omega_{i},k}}{\Vert u \Vert_{\mathcal{A},\omega_{i}^{\ast}}},}
\end{array}
\end{equation}
where the first infimum is taken over all $n$-dimensional subspaces of $H_{DI}^{1}(\omega_{i})$. Then the optimal approximation space $\hat{Q}(n)$ satisfies
\begin{equation}\label{eq:2-14}
d_{n}(\omega_{i},\omega_{i}^{\ast}) =\sup_{u\in H_{\mathcal{B}}(\omega_{i}^{\ast})} \inf_{v\in \hat{Q}(n)}\frac {\Vert P_{i}u-v\Vert_{\mathcal{A},\omega_{i},k}}{\Vert u \Vert_{\mathcal{A},\omega_{i}^{\ast}}}.
\end{equation}

The following lemma gives a characterization of the $n$-width via the singular values and singular vectors of the compact operator $P_{i}$.
\begin{lemma}\label{lem:2-1}
For each $j\in\mathbb{N}$, let $(\lambda_{j},\phi_{j})$ be the $j$-th eigenpair (arranged in decreasing order) of the problem 
\begin{equation}\label{eq:2-15}
\mathcal{A}_{\omega_{i},k}(\chi_{i} \phi, \chi_{i} v) = \lambda\,\mathcal{A}_{\omega_{i}^{\ast}}(\phi, v),\quad \forall v\in H_{\mathcal{B}}(\omega_{i}^{\ast}).
\end{equation}
Then $d_{n}(\omega_{i},\omega_{i}^{\ast}) =\lambda^{1/2}_{n+1}$ and the associated optimal approximation space is given by 
\begin{equation}\label{eq:2-16}
\hat{Q}(n) = {\rm span}\{\chi_{i} \phi_{1}, \cdots, \chi_{i} \phi_{n}\}.
\end{equation}
\end{lemma}
\begin{proof}
Let $P_{i}^{\ast}:H_{DI}^{1}(\omega_{i})\rightarrow H_{\mathcal{B}}(\omega_{i}^{\ast})$ denote the adjoint of the operator $P_{i}$ \cm{in the $\mathcal{A}_{\omega_i^{\ast}}(\cdot,\cdot)$ inner-product}, and let $\{\phi_{i}\}$ and $\{\lambda_{i}\}$ denote the eigenfunctions and eigenvalues of the problem
\begin{equation}\label{eq:2-17}
P_{i}^{\ast}P_{i} \phi= \lambda\,\phi.
\end{equation} 
A classical result in \cite[Theorem 2.5]{pinkus1985n} gives that $d_{n}(\omega_{i},\omega_{i}^{\ast}) = \lambda^{1/2}_{n+1}$ and that the associated optimal approximation space is given by $\hat{Q}(n) = {\rm span}\{P_{i}\phi_{1}, \cdots, P_{i}\phi_{n}\}$. We complete the proof by noting that \cref{eq:2-15} is the variational formulation of \cref{eq:2-17}.
\end{proof}
\begin{rem}\label{rem:3-2}
By \cref{thm:2-1}, the eigenvalue problem \cref{eq:2-15} can be rewritten as 
\begin{equation}\label{eq:2-18}
(\widetilde{Q}_{k}\phi, \,v)_{L^{2}(\omega_{i}^{\ast})}=\lambda\,\mathcal{A}_{\omega_{i}^{\ast}}(\phi, \,v),\quad \forall v\in H_{\mathcal{B}}(\omega_{i}^{\ast}),
\end{equation}
where $\widetilde{Q}_{k} := A\nabla\chi_{i} \cdot \nabla \chi_{i} + 2k^{2}V^{2}\chi_{i}^{2}$, i.e., an eigenvalue problem with weighted $L^{2}$ norm on $\omega_{i}^{\ast}$.
\end{rem}

From the definition and characterization of the $n$-width, we see that a generalized harmonic function can be well approximated by eigenvectors of the problem \cref{eq:2-15}. This motivates the definition of the local particular function and the local approximation space for the MS-GFEM as follows.
\begin{theorem}\label{thm:2-2}
On subdomain $\omega_i$, let the local particular function and the local approximation space be defined as
\begin{equation}\label{eq:2-19}
u^{p}_{i}:=\psi_{i}|_{\omega_{i}}\;\;\; {\rm and}\;\;\; S_{n_{i}}(\omega_{i}) := {\rm span}\{\phi_{1}|_{\omega_{i}},\cdots, \phi_{n_{i}}|_{\omega_{i}}\},
\end{equation}
where $\psi_{i}$ is the solution of \cref{eq:2-4} and $\phi_{j}$ denotes the $j$-th eigenfunction of the problem \cref{eq:2-15}, and let $u^{{\mathdutchcal{e}}}$ be the exact solution of the problem \cref{eq:1-2}. Then,
\begin{equation}\label{eq:2-20}
\inf_{\varphi\in u^{p}_{i} + S_{n_{i}}(\omega_{i})}\Vert \chi_{i}(u^{{\mathdutchcal{e}}} - \varphi)\Vert_{\mathcal{A},\omega_{i},k}\leq d_{n_{i}}(\omega_{i},\omega_{i}^{\ast})\,\Vert u^{{\mathdutchcal{e}}}- \psi_{i}\Vert_{\mathcal{A},\omega_{i}^{\ast}}.
\end{equation}
\end{theorem}
\begin{proof}
Since $u^{{\mathdutchcal{e}}}|_{\omega_{i}^{\ast}}-\psi_{i}\in H_{\mathcal{B}}(\omega^{\ast}_{i})$, \cref{eq:2-20} follows from the definition and characterization of the $n$-width.
\end{proof}
\begin{rem}\label{rem:3-3}
In \cite{babuvska1997partition}, plane waves were used to construct the local approximation spaces for the homogeneous Helmholtz equation with constant coefficients in $\mathbb{R}^{2}$ as
\begin{equation}
S_{n_i} = {\rm span}\Big\{\exp\Big({\rm i}k\big(x\cos\frac{2\pi q}{n_i}+ y\sin\frac{2\pi q}{n_i}\big)\Big),\;\;q=0,\dots,n_{i}-1 \Big\}.
\end{equation}
Note that in the constant coefficient case, plane wave solutions lie in the generalized harmonic spaces. Therefore, our method can be viewed as an extension of the method in \cite{babuvska1997partition} to heterogeneous-coefficients and to the inhomogeneous ($f\neq 0$) case. A combination of the classical FEM and the plane wave partition of unity method for the Helmholtz equation can be found in \cite{strouboulis2006generalized}.
\end{rem}


\section{Convergence analysis of the continuous MS-GFEM}\label{sec-3}
In this section, we first derive wavenumber explicit upper bounds for the local approximation errors and then establish a quasi-optimal global convergence of the method.
\subsection{Local approximation error estimates}
By \cref{thm:2-2}, the local approximation error in each subdomain is bounded by the $n$-width \cref{eq:2-13} and thus it suffices to derive upper bounds for the $n$-widths. \cm{To avoid technical complications, we assume that $\Omega$ is a Lipschitz polyhedral domain, which is only relevant to the $n$-widths associated with boundary subdomains. The case of a general Lipschitz domain is discussed in Remarks \ref{rem:3-5} and \ref{rem:3-6} below.} For ease of notation, the subscript index $i$ is omitted in this subsection. 

\cm{A key factor that determines the decay rate of the $n$-width $d_{n}(\omega, \omega^{\ast})$ -- both in the case of an interior subdomain or a boundary subdomain $\omega$ --  is the oversampling size, i.e., the distance between $\omega$ and $\partial \omega^{\ast}$ (modulo the boundary of $\Omega$). We denote it by} $\delta^{\ast} := {\rm dist}(\omega,\,\partial \omega^{\ast}\setminus \partial \Omega)$.
\begin{theorem}\label{thm:3-1}
\cm{
Let $\delta^{\ast}>0$, and let $\sigma = k\delta^{\ast}V_{\rm max}/(2a^{1/2}_{\rm max})$. There exist $n_{0}>0$ and $b>0$ independent of $k$, such that the following two results hold.

\vspace{1ex}
\begin{itemize}
\item[(i)] If $\omega$ and $\omega^{\ast}$ are concentric cubes or truncated concentric cubes (in the boundary subdomain case) of side lengths $H$ and $H^{\ast}$, then
\begin{equation}\label{eq:3-1}
d_{n}(\omega, \omega^{\ast})\leq e^{\sigma}e^{-bn^{1/(d+1)}}e^{-\rho(H/H^{\ast})bn^{1/(d+1)}},\quad \forall\, n>n_{0},
\end{equation}
where $\rho(s) = 1+{s\log(s)}/{(1-s)}$.

\vspace{1ex}
\item[(ii)] If $\omega$ and $\omega^{\ast}$ are general domains, then
\begin{equation}\label{eq:3-1-0}
d_{n}(\omega, \omega^{\ast})\leq e^{\sigma}e^{-bn^{1/(d+1)}},\quad \forall\, n>n_{0}.
\end{equation}
\end{itemize}
}
\end{theorem}

A few remarks are in order before proceeding to the proof of the theorem.

\begin{rem}\label{rem:3-3-0}
\cm{Estimate \cref{eq:3-1} holds for other (quasi-concentric) regular domains $\omega$ and $\omega^{\ast}$ that satisfy $|\omega|^{1/d} \,\big (|\omega^{\ast}|^{1/d}\big) \simeq {\rm diam}(\omega) \,\big({\rm resp.}\, {\rm diam}(\omega^{\ast})\big)$, e.g., spheres or tetrahedrons. In this case, $H$ and $H^{\ast}$ in \cref{eq:3-1} are replaced by ${\rm diam}(\omega)$ and ${\rm diam}(\omega^{\ast})$. }
\end{rem}
\vspace{-2ex}
\begin{rem}
The proof of \cref{thm:3-1} provides explicit values for $n_{0}$ and $b$:
\begin{equation}\label{eq:3-2}
n_{0} = 2(4e\Theta)^{d}\quad {\rm and}\quad b = (2e\Theta+1/2)^{-d/(d+1)}, 
\end{equation}
\cm{where $\Theta=C\big(\frac{a_{\rm max}}{a_{\rm min}}\big)^{1/2} \frac{|\omega^{\ast}|^{1/d}}{\delta^{\ast}}$ for general $\omega$ and $\omega^{\ast}$, and $\Theta=C\big(\frac{a_{\rm max}}{a_{\rm min}}\big)^{1/2} \frac{H^{\ast}}{\delta^{\ast}}$ for (truncated) concentric cubes, with $C>0$ depending only on $d$.}
\end{rem}
\begin{rem}\label{rem:3-4}
Let us look carefully at the decay rate in \cref{eq:3-1} where an explicit dependence on all the important parameters is available. If $k=0$, $d_{n}(\omega, \omega^{\ast})$ decays nearly exponentially with respect to $n$ and $H/H^{\ast}$ as shown in \cite{ma2021novel}. If $H\sim H^{\ast}\sim k^{-1}$, then $e^{\sigma}=O(1)$ and the decay rate of $n$-width is similar to that for positive definite elliptic problems. In the general case, since $k$ only appears in the $n$-independent factor $e^{\sigma}$, $d_{n}(\omega, \omega^{\ast})$ still decays nearly exponentially with $n$ and the rate is independent of $k$. The decay of $d_{n}(\omega, \omega^{\ast})$ with respect to $H/H^{\ast}$ (keeping $H$ fixed) in the general case is nontrivial and it depends on the relation between $k$ and $n$. For moderate $k$, $d_{n}(\omega, \omega^{\ast})$ decays nearly exponentially with $H/H^{\ast}$, even for small $n$. However, if $k$ is large, the situation is different and we can distinguish two cases:
\begin{itemize}
\item[(i)] for $n$ sufficiently large, $d_{n}(\omega, \omega^{\ast})$ decays nearly exponentially with $H/H^{\ast}$;

\vspace{0.5ex}
\item[(ii)] for small $n$, $d_{n}(\omega, \omega^{\ast})$ first decreases and then stagnates as $H/H^{\ast}\rightarrow 0$.
\end{itemize}
Thus, there exists a \emph{resonance effect} between $k$ and $n$ that influences the decay of $d_{n}(\omega, \omega^{\ast})$ with respect to the size $H^{\ast}$ of the oversampling domain in the high-frequency regime. Numerical results in \cref{sec-6} confirm the presence of this effect.
\end{rem}

The key to the proof of \cref{thm:3-1} is to explicitly construct an $n$-dimensional space $Q(n)\subset H_{DI}^{1}(\omega)$ with the approximation  error decaying nearly exponentially. As in \cite{babuska2011optimal,ma2021novel}, this can be achieved by an iteration argument performed on a series of nested domains. To do this, we need to first establish an auxiliary approximation result, stating that the generalized harmonic spaces can be approximated locally by $m$-dimensional spaces with an explicit algebraic convergence rate with respect to $m$. \cm{As a first step, we give a general approximation result with the approximation error measured in the $L^{2}$ norm (see \cite{ma2022exponential} for its proof)}.
\begin{lemma}\label{lem:3-1}
\cm{Assume that $\Omega$ is a Lipschitz polyhedral domain. Let $\mathcal{D}\subset\mathcal{D}^{\ast}$ be open connected subsets of $\Omega$ with $\delta = {\rm dist}(\mathcal{D},\,\partial \mathcal{D}^{\ast}\setminus\partial \Omega)>0$, and let $S(\mathcal{D}^{\ast})$ be a closed subspace of $H^{1}(\mathcal{D}^{\ast})$. In addition, we assume that the $H^{1}$-seminorm $\Vert \nabla \cdot\Vert_{L^{2}(\mathcal{D}^{\ast})}$ is a norm on $S(\mathcal{D}^{\ast})$ equivalent to the standard $H^{1}$-norm. Then, there exist positive constants $C_{1}$ and $C_{2}$ depending only on $d$, such that for each integer $m\geq C_{1}|\mathcal{D}^{\ast}|\delta^{-d}$, there exists an $m$-dimensional space $\Psi_{m}(\mathcal{D}^{\ast})\subset S(\mathcal{D}^{\ast})$ satisfying
\begin{equation}\label{eq:3-3}
\inf_{\varphi\in \Psi_{m}(\mathcal{D}^{\ast})} \Vert u-\varphi\Vert_{L^{2}(\mathcal{D})}\leq C_{2}m^{-1/d} |\mathcal{D}^{\ast}|^{1/d} \Vert \nabla u \Vert_{L^{2}(\mathcal{D}^{\ast})}\quad \forall u\in S(\mathcal{D}^{\ast}).
\vspace{-1ex}
\end{equation}}
\end{lemma}
\begin{rem}\label{rem:3-5}
\cm{If $\partial \Omega$ is $C^{1}$ smooth, \cref{lem:3-1} also holds true with the constants $C_{1}$ and $C_{2}$ depending on $\partial \Omega$ (for the domain $\mathcal{D}^{\ast}$ touching $\partial \Omega$). If $\Omega$ has a general Lipschitz boundary, \cref{lem:3-1} can also be proved for Lipschitz domains $\mathcal{D}^{\ast}$. However, in this case, the constants $C_{1}$ and $C_{2}$ may depend on the shape of $\mathcal{D}^{\ast}$. We refer the reader to \cite[Remark 3.11]{ma2022exponential} for more details of these extensions.}
\end{rem}

Combining \cref{lem:3-1} and the Caccioppoli inequality can give the desired auxiliary approximation result with the approximation error measured in the energy norm. \cm{To this end, we need to assume that the domain $\mathcal{D}^{\ast}$ satisfies the cone condition (see, e.g., \cite[pp.~82]{adams2003sobolev}).}
\begin{lemma}\label{lem:3-2}
\cm{Let $\mathcal{D}$ and $\mathcal{D}^{\ast}$ be open connected subsets of $\Omega$ with $\mathcal{D}\subset\mathcal{D}^{\ast}$ and $\delta = {\rm dist}(\mathcal{D},\partial \mathcal{D}^{\ast} \setminus\partial \Omega)>0$, and let the constants $C_{1}$ and $C_{2}$ be as in \cref{lem:3-1}. In addition, we assume that $\mathcal{D}^{\ast}$ satisfies the cone condition. Then, for each integer $m\geq C_{1}|\mathcal{D}^{\ast}|(\delta/2)^{-d}$, there exits an $m$-dimensional space $\Psi_{m}(\mathcal{D}^{\ast})\subset H_{\mathcal{B}}(\mathcal{D}^{\ast})$ such that for any $u\in H_{\mathcal{B}}(\mathcal{D}^{\ast})$, 
\begin{equation}\label{eq:3-4}
\inf_{\varphi\in \Psi_{m}(\mathcal{D}^{\ast})}\Vert u-\varphi\Vert_{\mathcal{A},\mathcal{D}} \leq C_{2}m^{-1/d}|\mathcal{D}^{\ast}|^{1/d}\,\Big(\frac{a_{\rm max}}{a_{\rm min}}\Big)^{1/2} \Big(\frac{2}{\delta} + \frac{kV_{\rm max}}{a^{1/2}_{\rm max}}\Big)\,\Vert u \Vert_{\mathcal{A},\mathcal{D}^{\ast}}.
\end{equation}}
\end{lemma}
\vspace{-2ex}
\begin{proof}
\cm{Let $\mathcal{D}_{\delta/2}$ denote an open connected subset of $\Omega$ satisfying $\mathcal{D}\subset \mathcal{D}_{\delta/2}\subset\mathcal{D}^{\ast}$ and ${\rm dist}(\mathcal{D},\,\partial \mathcal{D}_{\delta/2}\setminus\partial \Omega) = {\rm dist}(\mathcal{D}_{\delta/2},\,\partial \mathcal{D}^{\ast}\setminus\partial \Omega) = \delta/2$. Our proof is divided into two steps. First, we shall apply \cref{lem:3-1} on $\mathcal{D}_{\delta/2}$ and $\mathcal{D}^{\ast}$ with $S(\mathcal{D}^{\ast}) = H_{\mathcal{B}}(\mathcal{D}^{\ast})$. To do this, we need to verify that $H_{\mathcal{B}}(\mathcal{D}^{\ast})$ satisfies the required conditions. It is clear that $H_{\mathcal{B}}(\mathcal{D}^{\ast})$ is a closed subspace of $H^{1}(\mathcal{D}^{\ast})$. Moreover, since the domain $\mathcal{D}^{\ast}$ satisfies the cone condition, the embedding $H^{1}(\mathcal{D}^{\ast})\subset L^{2}(\mathcal{D}^{\ast})$ is compact (see, e.g., \cite[Theorem 6.3]{adams2003sobolev}). Hence, similar to \cref{eq:2-3-0}, it can be shown that there exists $C>0$ such that 
\begin{equation}
\Vert u\Vert_{L^{2}(\mathcal{D}^{\ast})}\leq C\Vert \nabla u\Vert_{L^{2}(\mathcal{D}^{\ast})} \quad  \forall u\in H_{\mathcal{B}}(\mathcal{D}^{\ast}),
\end{equation}
and thus $\Vert\nabla\cdot\Vert_{L^{2}(\mathcal{D}^{\ast})}$ is a norm on $H_{\mathcal{B}}(\mathcal{D}^{\ast})$ equivalent to the standard $H^{1}$-norm. Having verified these conditions, we can apply \cref{lem:3-1} to deduce that for each $m\geq C_{1}|\mathcal{D}^{\ast}|(\delta/2)^{-d}$, there exists an $m$-dimensional space $\Psi_{m}(\mathcal{D}^{\ast})\subset H_{\mathcal{B}}(\mathcal{D}^{\ast})$ satisfying
\begin{equation}\label{eq:3-4-0}
    \inf_{\varphi\in \Psi_{m}(\mathcal{D}^{\ast})} \Vert u-\varphi\Vert_{L^{2}(\mathcal{D}_{\delta/2})}\leq C_{2}m^{-1/d} |\mathcal{D}^{\ast}|^{1/d} a_{\rm min}^{-1/2} \,\Vert u \Vert_{\mathcal{A},\mathcal{D}^{\ast}}\quad \forall u\in H_{\mathcal{B}}(\mathcal{D}^{\ast}).
\end{equation}
Next we choose a cut-off function $\eta\in C^{1}(\mathcal{D}_{\delta/2})$ satisfying
\begin{equation}
 \eta = 0\;\;\,\; {\rm on}\;\, \partial \mathcal{D}_{\delta/2}\setminus \partial\Omega,\quad \eta = 1 \;\;\;\, {\rm on}\;\,\mathcal{D},\quad {\rm and}\;\;\quad |\nabla \eta|\leq 2/\delta.
\end{equation}
Note that for any $u\in H_{\mathcal{B}}(\mathcal{D}^{\ast})$ and $\varphi\in \Psi_{m}(\mathcal{D}^{\ast})$, $u-\varphi\in H_{\mathcal{B}}(\mathcal{D}_{\delta/2})$. Applying the Caccioppoli-type inequality \cref{eq:2-6} to $\eta$ and $u-\varphi$ on $\mathcal{D}_{\delta/2}$ and combining the result with \cref{eq:3-4-0}, we get the desired estimate \cref{eq:3-4}.}
\end{proof}


\Cref{lem:3-2} is the starting point of the iteration argument in which the approximation result \cref{eq:3-4} is applied recursively. \cm{Recall that $\delta^{\ast}= {\rm dist}(\mathcal{\omega},\partial \mathcal{\omega}^{\ast} \setminus\partial \Omega)$. For an integer $N\geq 2$, we denote by $\{\omega^{j}\}_{j=1}^{N+1}$ a family of nested domains satisfying $\omega=\omega^{N+1}\subset \omega^{N}\subset\cdots\subset\omega^{1} = \omega^{\ast}$ and ${\rm dist}(\omega^{j},\partial \omega^{j+1}\setminus \partial \Omega) = \delta^{\ast}/N$. The intermediate domains can be constructed iteratively as follows:
\begin{equation}
    \omega^{j} = \bigcup_{{\bm x}\in \omega^{j+1}} B({\bm x}, \delta^{\ast}/N)\cap \Omega, \qquad j=N,N-1,\ldots,2,
\end{equation}
where $B({\bm x}, \delta^{\ast}/N)$ denotes the ball centered at ${\bm x}$ with radius $\delta^{\ast}/N$. We see that $\{\omega^{j}\}_{j=2}^{N}$ all satisfy the cone condition since they are unions of uniform balls. If $\omega$ and $\omega^{\ast}$ are (truncated) concentric cubes of side lengths $H$ and $H^{\ast}$, $\{\omega^{j}\}_{j=2}^{N}$ can simply be chosen to be (truncated) concentric cubes of side lengths $H^{\ast}-2\delta^{\ast} (j-1)/N$.} Let $n=Nm$ and define
\begin{equation}\label{eq:3-5}
\Psi(n,\omega,\omega^{\ast}) = \Psi_{m}(\omega^{1}) + \cdots+\Psi_{m}(\omega^{N}),
\end{equation}
where $\Psi_{m}(\omega^{j})$ are given by \cref{lem:3-2}. The following lemma shows that $\Psi(n,\omega,\omega^{\ast})$ can deliver a sharper approximation result than \cref{eq:3-4}. \cm{Note that this result holds for both cases where $\omega$ is an interior subdomain or a boundary subdomain.}
\begin{lemma}\label{lem:3-3}
\cm{ Let $\chi$ be the partition of unity function supported on $\omega$, and let the constants $C_{1}$ and $C_{2}$ be as in \cref{lem:3-1}. Moreover, let $m$ and $N$ satisfy $m\geq C_{1} |\omega^{\ast}| (2N/\delta^{\ast})^{d}$.
Then, for any $u\in H_{\mathcal{B}}(\omega^{\ast})$,
\begin{equation}\label{eq:3-6}
\inf_{\varphi\in \Psi(n,\omega,\omega^{\ast})} \frac{\Vert \chi(u - \varphi) \Vert_{\mathcal{A},\omega,k}}{\Vert u \Vert_{\mathcal{A},\omega^{\ast}}} \leq 2e^{\sigma} \xi^{N} \prod_{j=1}^{N} |\omega^{j}|^{1/d},
\vspace{-1ex}
\end{equation}
where $\sigma = k\delta^{\ast}V_{\rm max}/(2a^{1/2}_{\rm max})$ and $\xi$ is given by
\begin{equation}\label{eq:3-7}
\xi = \xi(N,m)=2C_{2}Nm^{-1/d}\Big(\frac{a_{\rm max}}{a_{\rm min}}\Big)^{1/2}\frac{1}{\delta^{\ast}}.
\end{equation}}
\vspace{-2ex}
\end{lemma}
\begin{proof}
\cm{First of all, we note that all the intermediate domains $\{\omega^{j}\}_{j=2}^{N}$ constructed above and the oversampling domain $\omega^{\ast}$ satisfy the cone condition (recall that $\omega^{\ast}$ is a Lipschitz domain).} Using \cref{lem:3-2} with $\mathcal{D}^{\ast} = \omega^{1}=\omega^{\ast}$ and $\mathcal{D}=\omega^{2}$ and noting that ${\rm dist}(\omega^{2}, \partial \omega^{1}\setminus\partial \Omega) = \delta^{\ast}/N$, we see that there exists a $v_{u}^{1}\in \Psi_{m}(\omega^{1})$ such that
\begin{equation}\label{eq:3-8}
\Vert u - v_{u}^{1}\Vert_{\mathcal{A},\omega^{2}} \leq \xi (1+\sigma/N) |\omega^{1}|^{1/d} \Vert u \Vert_{\mathcal{A},\omega^{1}}.
\end{equation}
Since $u - v_{u}^{1}\in H_{\mathcal{B}}(\omega^{2})$, we can apply \cref{lem:3-2} again and combine the result with \cref{eq:3-8} to find a $v_{u}^{2}\in \Psi_{m}(\omega^{2})$ satisfying
\begin{equation}\label{eq:3-9}
\Vert u - v_{u}^{1}-v_{u}^{2}\Vert_{\mathcal{A},\omega^{3}} \leq \xi^{2} (1+\sigma/N)^{2} |\omega^{1}|^{1/d} |\omega^{2}|^{1/d}\,\Vert u\Vert_{\mathcal{A},\omega^{1}}.
\end{equation}
Repeating successively the same argument for $N-3$ times, we see that there exist $v_{u}^{j}\in \Psi_{m}(\omega^{j})$, $j=1,2\cdots,N-1$, such that
\begin{equation}\label{eq:3-10}
\Big \Vert u - \sum_{j=1}^{N-1}v_{u}^{j} \Big\Vert_{\mathcal{A},\omega^{N}} \leq \xi^{N-1} (1+\sigma/N)^{N-1} \Big(\prod_{j=1}^{N-1} |\omega^{j}|^{1/d}\Big) \,\Vert u \Vert_{\mathcal{A},\omega^{1}}.
\end{equation}
Finally, combining \cref{lem:3-1} and the Caccioppoli-type inequality \cref{eq:2-6} with $\eta = \chi$, there exists a $v_{u}^{N}\in \Psi_{m}(\omega^{N})$ such that
\begin{equation}\label{eq:3-11}
\begin{array}{lll}
{\displaystyle  \Big\Vert \chi\big(u - \sum_{j=1}^{N}v_{u}^{j}\big) \Big\Vert_{\mathcal{A},\omega,k} \leq C_{2}m^{-1/d} |\omega^{N}|^{1/d}a^{-1/2}_{\rm min}  }\\[3mm]
{\displaystyle \qquad \quad  \cdot \Big(a^{1/2}_{\rm max}\Vert\nabla\chi\Vert_{L^{\infty}(\omega)} + 2kV_{\rm max}\Big) \Big\Vert u -  \sum_{j=1}^{N-1}v_{u}^{j} \Big\Vert_{\mathcal{A},\omega^{N}}  }\\[3mm]
{\displaystyle \quad \leq 2\xi |\omega^{N}|^{1/d} \Big(\Vert\nabla\chi\Vert_{L^{\infty}(\omega)}\delta^{\ast} /({4N}) +\sigma/N\Big) \Big\Vert u - \sum_{j=1}^{N-1}v_{u}^{j} \Big\Vert_{\mathcal{A},\,\omega^{N}}.}
\end{array}
\end{equation}
Without loss of generality, we assume that $\Vert \nabla\chi\Vert_{L^{\infty}(\omega)} \leq 4N/\delta^{\ast}$. It follows from \cref{eq:3-10}, \cref{eq:3-11}, and the inequality $(1+\sigma/N)^{N}\leq e^{\sigma}$ that
\begin{equation}\label{eq:3-12}
\begin{array}{lll}
{\displaystyle  \Big\Vert \chi\big(u - \sum_{j=1}^{N}v_{u}^{j}\big) \Big\Vert_{\mathcal{A},\omega,k} \leq 2 (1+\sigma/N)^{N} \xi^{N} \Big(\prod_{j=1}^{N}{|\omega^{j}|}^{1/d}\Big) \Vert u \Vert_{\mathcal{A},\omega^{1}} }\\[3mm]
{\displaystyle \qquad \qquad \qquad \qquad \quad \;\;\leq 2e^{\sigma}\xi^{N} \Big(\prod_{j=1}^{N}{|\omega^{j}|}^{1/d}\Big) \Vert u \Vert_{\mathcal{A},\omega^{1}}.}
\end{array}
\end{equation}
Since $\sum_{j=1}^{N}v_{u}^{j}\in \Psi(n,\omega,\omega^{\ast})$ and $\omega^{1} = \omega^{\ast}$, \cref{eq:3-7} follows immediately from \cref{eq:3-12}.
\end{proof}
\begin{rem}\label{rem:3-6}
\cm{By the remark after \cref{lem:3-1}, if $\partial \Omega$ is $C^{1}$ smooth, \cref{lem:3-3} also holds true with the constants $C_{1}$ and $C_{2}$ depending on $\partial \Omega$ (for boundary subdomains). In the case of a general Lipschitz boundary, if all the $\omega^{k}$ are Lipschitz domains, the proof of \cref{lem:3-3} proceeds similarly as above. However, the constants $C_{1}$ and $C_{2}$ in each application of \cref{lem:3-2} above may depend on the shape of the associated $\omega^{k}$ and thus could be different. If the nested domains are shape-regular (regardless of the parts intersecting $\partial \Omega$), e.g., truncated concentric cubes, it is possible to prove that these different constants can be bounded above via a sophisticated analysis (hence, \cref{lem:3-3} holds true). We will address this case in future works.}
\end{rem}

Now we are in a position to prove \cref{thm:3-1}.
\begin{proof}[Proof of \cref{thm:3-1}] 
Let $Q(n) =\big\{\chi v\,:\,v\in \Psi(n,\omega,\omega^{\ast})\big\} \subset H^{1}_{DI}(\omega)$. It follows from \cref{lem:3-3} that 
\begin{equation}\label{eq:3-13}
\sup_{u\in H_{\mathcal{B}}(\omega^{\ast})}\inf_{\varphi \in Q(n)}\frac{\Vert \chi u - \varphi\Vert_{\mathcal{A},\omega,k}}{\Vert u  \Vert_{\mathcal{A},\omega^{\ast}}} \leq 2e^{\sigma} \xi^{N} \prod_{j=1}^{N}{|\omega^{j}|}^{1/d},
\end{equation}
where $\sigma = k\delta^{\ast}V_{\rm max}/(2a^{1/2}_{\rm max})$ and $\xi$ is given in \cref{eq:3-7}. \cm{We first assume that $\omega$ and $\omega^{\ast}$ are general domains (including the boundary subdomain case). Since for $j=1,\ldots,N$, $|\omega^{j}| \leq |\omega^{\ast}|$, we simply have
\begin{equation}\label{eq:3-13-0}
\sup_{u\in H_{\mathcal{B}}(\omega^{\ast})}\inf_{\varphi \in Q(n)}\frac{\Vert \chi u - \varphi\Vert_{\mathcal{A},\omega,k}}{\Vert u  \Vert_{\mathcal{A},\omega^{\ast}}} \leq 2e^{\sigma} \big(\xi |\omega^{\ast}|^{1/d}\big)^{N}.
\end{equation}
Denoting by 
\begin{equation}\label{eq:3-14}
\Theta=C\Big(\frac{a_{\rm max}}{a_{\rm min}}\Big)^{1/2} \frac{|\omega^{\ast}|^{1/d}}{\delta^{\ast}},\;\;\;n_{0} = 2(4e\Theta)^{d},\;\;\; b = (2e\Theta+1/2)^{-d/(d+1)},
\end{equation}
where $C=2\max\{C_{1}, C_{2}\}$, and using a similar argument as in the proof of Theorem 3.6 in \cite{ma2021novel} by choosing $m$ and $N$ such that 
\begin{equation}\label{eq:3-14-0}
\big(e\Theta(N+1)^{2}/N\big)^{d}\leq m < \big(1+e\Theta(N+1)^{2}/N\big)^{d},
\end{equation}}
\cm{we can prove that for any $n>n_{0}$,
\begin{equation}\label{eq:3-14-1}
\big(\xi |\omega^{\ast}|^{1/d}\big)^{N} \leq e^{-1}e^{-bn^{1/(d+1)}}.
\end{equation}}
Inserting \cref{eq:3-14-1} into \cref{eq:3-13-0} and recalling the definition of the $n$-width, we prove \cref{thm:3-1} for the case of general domains $\omega$ and $\omega^{\ast}$.

\cm{If $\omega$ and $\omega^{\ast}$ are concentric cubes or truncated concentric cubes (in the boundary subdomain case) of side lengths $H$ and $H^{\ast}$, we can prove a sharper bound for \cref{eq:3-13}. In this case, as previously noted, we can choose $\{\omega^{j}\}_{j=2}^{N}$ to be (truncated) concentric cubes of side lengths $H^{\ast}-2\delta^{\ast} (j-1)/N$ with $\delta^{\ast}=(H^{\ast}-H)/2$. Hence, we have $|\omega^{j}|^{1/d} = H^{\ast}-2\delta^{\ast} (j-1)/N$ in the interior subdomain case and $|\omega^{j}|^{1/d} \leq H^{\ast}-2\delta^{\ast} (j-1)/N$ in the boundary subdomain case. It follows that
\begin{equation}\label{eq:3-14-2}
\prod_{j=1}^{N}{|\omega^{j}|}^{1/d} \leq \prod_{j=1}^{N} \big(H^{\ast}-2\delta^{\ast} (j-1)/N\big) = (H^{\ast})^{N} \prod_{j=1}^{N} \Big(1-\frac{2(j-1)\delta^{\ast}}{NH^{\ast}}\Big).
\end{equation}
We can easily see that a similar result as \cref{eq:3-14-2} holds if $\omega$ and $\omega^{\ast}$ are spheres or other (quasi-concentric) regular domains that satisfy $|\omega|^{1/d} \,\big (|\omega^{\ast}|^{1/d}\big) \simeq {\rm diam}(\omega)$ $\big({\rm resp.}\, {\rm diam}(\omega^{\ast})\big)$.} 
Combining \cref{eq:3-14-2} and \cref{eq:3-13} leads to
\begin{equation}\label{eq:3-14-3}
\sup_{u\in H_{\mathcal{B}}(\omega^{\ast})}\inf_{\varphi \in Q(n)}\frac{\Vert \chi u - \varphi\Vert_{\mathcal{A},\omega,k}}{\Vert u  \Vert_{\mathcal{A},\omega^{\ast}}} \leq 2e^{\sigma} \big(\xi H^{\ast}\big)^{N} \prod_{j=1}^{N} \Big(1-\frac{2(j-1)\delta^{\ast}}{NH^{\ast}}\Big).
\end{equation}
Defining $\Theta$, $n_{0}$, and $b$ similarly as above (with $|\omega^{\ast}|^{1/d}$ in $\Theta$ replaced by $H^{\ast}$), we can use a similar proof as that of \cref{eq:3-14-1} and Lemma 3.14 of \cite{ma2021novel} to prove that
\begin{equation}\label{eq:3-15}
(\xi H^{\ast})^{N} \prod_{j=1}^{N} \Big(1-\frac{2(j-1)\delta^{\ast}}{NH^{\ast}}\Big)\leq e^{-1}e^{-bn^{1/(d+1)}}e^{-\rho(H/H^{\ast})bn^{{1}/{(d+1)}}},\quad \forall n>n_{0},
\end{equation}
where $\rho(s) = 1+{s\log(s)}/{(1-s)}$. Inserting \cref{eq:3-15} into \cref{eq:3-14-3} yields the desired estimate \cref{eq:3-1}, and the proof of \cref{thm:3-1} is complete.
\end{proof}

\cm{We conclude this subsection by briefly discussing the influence of the partition of unity functions on the local approximation errors. Recall that in the proof of \cref{lem:3-3}, to simplify the presentation, we assumed that $\Vert \nabla\chi\Vert_{L^{\infty}(\omega)} \leq 4N/\delta^{\ast}$ (it typically holds true in practice). In general, $\Vert \nabla\chi\Vert_{L^{\infty}(\omega)}$ only enters the upper bound for the $n$-width as a factor in front of exponentially decaying terms. Thus, the decay rate of the $n$-width with respect to $n$, as well as other important parameters are independent of the choice of the partition of unity function.}

\subsection{Global error estimates} In this subsection, we first derive the global approximation error estimate and then establish a quasi-optimal convergence rate for the method under some assumptions. For convenience, we define
\begin{equation}\label{eq:4-1}
\begin{array}{lll}
{\displaystyle d_{\mathtt{max}} = \max_{i=1,\cdots,M}d_{n_i}(\omega_i,\omega_{i}^{\ast}),\quad C_{\mathtt{max}}(k) = \max_{i=1,\cdots,M} C^{i}_{{\rm stab}}(k),}\\[2mm]
{\displaystyle \qquad \qquad {\rm and}\quad H^{\ast}_{\mathtt{max}} = \max_{i=1,\cdots,M} {\rm diam}\,(\omega_{i}^{\ast}),}
\end{array}
\end{equation}
where $C^{i}_{{\rm stab}}(k)$ are the stability constants defined in \cref{ass:2-1}. Furthermore, we assume that the oversampling domains $\{ \omega_{i}^{\ast}\}_{i=1}^{M}$ satisfy a similar pointwise overlap condition as $\{ \omega_{i}\}_{i=1}^{M}$:
\begin{equation}\label{eq:4-1-0}
\exists \,\zeta^{\ast}\in \mathbb{N}\quad\forall {\bm x}\in\Omega \quad {\rm card}\{i\;|\;{\bm x}\in \omega^{\ast}_{i} \}\leq\zeta^{\ast}.
\end{equation}

\begin{lemma}\label{lem:4-1}
Let $u^{p}$ and $S_{n}(\Omega)$ be the global particular function and the trial space for the continuous MS-GFEM, and let  $u^{\mathdutchcal{e}}$ be the solution of the problem \cref{eq:1-2}. Then, there exists a $\varphi\in u^{p} + S_{n}(\Omega)$ such that
\begin{equation}\label{eq:4-2}
\Vert u^{\mathdutchcal{e}}-\varphi \Vert_{\mathcal{A},k}\leq \sqrt{2\zeta \zeta^{\ast}} d_{\mathtt{max}} \big(C_{\rm stab}(k) + \sqrt{2} C_{\mathtt{max}}(k) \big) (\Vert f\Vert_{L^{2}(\Omega)} + \Vert g\Vert_{L^{2}(\Gamma_{R})}),
\end{equation}
where $C_{\rm stab}(k)$ is the stability constant defined in \cref{ass:1-2}.
\end{lemma}
\begin{proof}
Using \cref{eq:2-20} and \cref{ass:2-1}, we see that on each subdomain $\omega_{i}$, there exists a $\varphi_{i}\in u_{i}^{p} + S_{n_i}(\omega_i)$ such that
\begin{equation}\label{eq:4-3}
\begin{array}{lll}
{\displaystyle \big\Vert \chi_{i}(u^{\mathdutchcal{e}} - \varphi_{i})\big\Vert_{\mathcal{A},\omega_i,k}\leq d_{n_i}(\omega_i,\omega_{i}^{\ast}) \,\Vert u^{{\mathdutchcal{e}}}- \psi_{i}\Vert_{\mathcal{A},\omega_{i}^{\ast}} }\\[2mm]
{\displaystyle \leq d_{n_i}(\omega_i,\omega_{i}^{\ast})\big(\Vert u^{{\mathdutchcal{e}}}\Vert_{\mathcal{A},\omega_{i}^{\ast}} + C_{\rm stab}^{i}(k)(\Vert f \Vert_{L^{2}(\omega_{i}^{\ast})} + \Vert g\Vert_{L^{2}(\partial \omega^{\ast}_{i} \cap \Gamma_{R})})\big).}
\end{array}
\end{equation}
Let $\varphi = \sum_{i=1}^{M}\chi_{i} \varphi_{i} \in u^{p} + S_{n}(\Omega)$. It follows from \cref{eq:2-0} that
\begin{equation}\label{eq:4-4}
\Vert u^{\mathdutchcal{e}}- \varphi \Vert^{2}_{\mathcal{A},k}  = \Big\Vert \sum_{i=1}^{M}\chi_{i}(u^{\mathdutchcal{e}}-\varphi_{i})\Big\Vert^{2}_{\mathcal{A},k} \leq \zeta \sum_{i=1}^{M}\big\Vert \chi_{i}(u^{\mathdutchcal{e}} - \varphi_{i})\big\Vert_{\mathcal{A},\omega_i,k}^{2}.
\end{equation}
Following the same lines as in the proof of Theorem 2.1 in \cite{ma2021novel}, we get \cref{eq:4-2} by inserting \cref{eq:4-3} into \cref{eq:4-4} and using \cref{eq:4-1-0} and the stability estimate \cref{eq:1-6}. 
\end{proof}

To derive a quasi-optimal convergence rate for the method, we introduce the solution operator $\widehat{S}: L^{2}(\Omega)\rightarrow H_{D}^{1}(\Omega)$ for the problem:
\begin{equation}\label{eq:4-4-0}
{\rm Find} \;\;\hat{u}\in H_{D}^{1}(\Omega)\quad {\rm such\;\,that}\quad \mathcal{B}(\hat{u},v) = (f,v)_{L^{2}(\Omega)}\quad \forall v\in H_{D}^{1}(\Omega),
\end{equation}
i.e., $\widehat{S}(f) := \hat{u}$, and define the following quantity
\begin{equation}\label{eq:4-4-1}
\eta (S_{n}(\Omega)):=\sup_{f\in L^{2}(\Omega)}\inf_{\varphi\in S_{n}(\Omega)}\frac{\Vert \widehat{S}(f)-\varphi \Vert_{\mathcal{A},k}}{\Vert f \Vert_{L^{2}(\Omega)}}.
\end{equation}
It follows from \cref{ass:1-2} that the operator $\widehat{S}$ is well defined. Recalling the inequality \cref{eq:1-5} with the $k$-independent constant $C_{\mathcal{B}}$ and using a standard duality argument (see, e.g., \cite{esterhazy2012stability}), we have
\begin{theorem}\label{thm:4-1}
Let $u^{\mathdutchcal{e}}$ be the solution of the problem \cref{eq:1-2}. Assuming that
\begin{equation}\label{eq:4-6}
2C_{\mathcal{B}}kV_{\rm max}\,\eta (S_{n}(\Omega)) \leq 1,
\end{equation}
then the problem \cref{eq:2-1} has a unique solution $u^{G}\in u^{p}+S_{n}(\Omega)$ satisfying
\begin{equation}\label{eq:4-7}
\Vert u^{\mathdutchcal{e}}-u^{G}\Vert_{\mathcal{A},k} \leq 2C_{\mathcal{B}}\inf_{\varphi\in u^{p} + S_{n}(\Omega)}\Vert u^{\mathdutchcal{e}}-\varphi \Vert_{\mathcal{A},k}.
\end{equation}
\end{theorem}
\begin{proof}
Since \cref{eq:2-1} is a finite-dimensional linear system, its unique solvablity is implied by \cref{eq:4-7} and \cref{ass:1-2}. Hence we restrict our attention to the proof of \cref{eq:4-7}. Let $e_{G} = u^{\mathdutchcal{e}}-u^{G}$. We observe that
\begin{equation}\label{eq:4-8}
\Vert e_{G} \Vert^{2}_{\mathcal{A},k} = {\rm Re}\, \mathcal{B}(e_{G}, e_{G}) + 2k^{2}\big(V^{2}e_{G}, e_{G}\big)_{L^{2}(\Omega)}. 
\end{equation}
Using the Galerkin orthogonality \cref{eq:2-1-0} and \cref{eq:1-5} yields that for any $\varphi\in u^{p}+S_{n}(\Omega)$,
\begin{equation}\label{eq:4-9}
\begin{array}{lll}
{\displaystyle \Vert e_{G} \Vert^{2}_{\mathcal{A},k} = {\rm Re}\, \mathcal{B}(e_{G}, u^{\mathdutchcal{e}}-\varphi) + 2k^{2}\big(V^{2}e_{G}, e_{G}\big)_{L^{2}(\Omega)} }\\[2mm]
{\displaystyle \qquad \leq C_{\mathcal{B}}\Vert e_{G} \Vert_{\mathcal{A},k} \Vert u^{\mathdutchcal{e}}-\varphi \Vert_{\mathcal{A},k} + 2k^{2}V^{2}_{\rm max}\Vert e_{G}\Vert_{L^{2}(\Omega)}^{2}.}
\end{array}
\end{equation}
To estimate $\Vert e_{G}\Vert_{L^{2}(\Omega)}$, we consider the adjoint problem:
\begin{equation}\label{eq:4-10}
{\rm Find} \;\;w^{\mathdutchcal{e}}\in H_{D}^{1}(\Omega)\quad {\rm such\;\,that}\quad \mathcal{B}(v, w^\mathdutchcal{e}) = (v, e_{G})_{L^{2}(\Omega)}\quad \forall v\in H_{D}^{1}(\Omega).
\end{equation}
Note that $\overline{w^{\mathdutchcal{e}}}\in  H_{D}^{1}(\Omega)$ satisfies
\begin{equation}\label{eq:4-10-1}
\mathcal{B}(\overline{w^\mathdutchcal{e}},v) = (\overline{e_{G}},v)_{L^{2}(\Omega)}\quad \forall v\in H_{D}^{1}(\Omega),
\end{equation}
i.e., $\overline{w^{\mathdutchcal{e}}} =\widehat{S}(\overline{e_{G}})$. Choosing $v = e_{G}$ in \cref{eq:4-10} and using the Galerkin orthogonality again, we see that for any $\varphi \in S_{n}(\Omega)$, 
\begin{equation}\label{eq:4-11}
\begin{array}{lll}
{\displaystyle \Vert e_{G}\Vert_{L^{2}(\Omega)}^{2} = \mathcal{B}(e_{G}, w^{\mathdutchcal{e}}-\varphi)  \leq C_{\mathcal{B}}\Vert e_{G} \Vert_{\mathcal{A},k} \, \Vert\overline{w^{\mathdutchcal{e}}}-\overline{\varphi} \Vert_{\mathcal{A},k}}\\[2mm]
{\displaystyle \qquad \qquad \;\; \leq  C_{\mathcal{B}}\,\Vert e_{G} \Vert_{\mathcal{A},k}\,\big\Vert \widehat{S}(\overline{e_{G}}) - \overline{\varphi}\big\Vert_{\mathcal{A},k},}
\end{array}
\end{equation}
which, combining with \cref{eq:4-4-1} and the fact that $\overline{\varphi}\in S_{n}(\Omega)$,  gives that
\begin{equation}\label{eq:4-12}
\Vert e_{G}\Vert_{L^{2}(\Omega)}\leq C_{\mathcal{B}}\,\eta (S_{n}(\Omega)) \Vert e_{G} \Vert_{\mathcal{A},k}.
\end{equation}
Inserting \cref{eq:4-12} into \cref{eq:4-9} and using \cref{eq:4-6}, we get
\begin{equation}\label{eq:4-13}
\begin{array}{lll}
{\displaystyle \Vert e_{G} \Vert^{2}_{\mathcal{A},k}\leq C_{\mathcal{B}}\Vert e_{G} \Vert_{\mathcal{A},k} \Vert u^{\mathdutchcal{e}}-\varphi \Vert_{\mathcal{A},k} + 2\big(kV_{\rm max}C_{\mathcal{B}} \eta (S_{n}(\Omega)) \big)^{2}\Vert e_{G} \Vert^{2}_{\mathcal{A},k}}\\[2mm]
{\displaystyle \qquad \qquad \leq C_{\mathcal{B}}\Vert e_{G} \Vert_{\mathcal{A},k} \Vert u^{\mathdutchcal{e}}-\varphi \Vert_{\mathcal{A},k} + \frac{1}{2} \Vert e_{G} \Vert^{2}_{\mathcal{A},k} \quad \forall \varphi\in u^{p}+S_{n}(\Omega),}
\end{array}
\end{equation}
from which the estimate \cref{eq:4-7} follows.
\end{proof}

In the rest of this subsection, we derive an upper bound for $\eta (S_{n}(\Omega))$. To do this, we need the following Poincar\'{e} inequality with an explicit dependence on the diameter of a domain, \cm{which can be found in \cite[Corollary A.15]{toselli2004domain}}.

\begin{lemma}\label{lem:4-2}
Let $\Omega^{\prime}\subset \mathbb{R}^{d}$ be a bounded Lipschitz domain and let $\Gamma^{\prime}\subset \partial \Omega^{\prime}$ with $|\Gamma^{\prime}|>0$. Then there exists a constant $C^{\prime}_{P}$ depending on $\cm{\Gamma^{\prime}}$ and on the shape of $\Omega^{\prime}$, but not on its size, such that for any $u\in H^{1}(\Omega^{\prime})$ with vanishing trace on $\Gamma^{\prime}$,
\begin{equation}\label{eq:4-5}
\Vert u \Vert_{L^{2}(\Omega^{\prime})}\leq C_{P}^{\prime} \,{\rm diam}(\Omega^{\prime})\,\Vert \nabla u\Vert_{L^{2}(\Omega^{\prime})}.
\end{equation}
\end{lemma}
By virtue of \cref{lem:4-2}, we define $C_{P}$ as the uniform Poincar\'{e} constant such that 
\begin{equation}\label{eq:4-5-0}
\Vert u \Vert_{L^{2}(\omega_{i}^{\ast})}\leq C_{P}\,{\rm diam}(\omega_{i}^{\ast})\, \Vert \nabla u\Vert_{L^{2}(\omega_{i}^{\ast})},\quad \forall u\in H^{1}_{DI}(\omega_{i}^{\ast}),\quad \forall i=1,\cdots,M.
\end{equation}

Recalling \cref{ass:1-2} with the stability constant $C_{\rm stab}(k)$, and combining the local approximation error estimates and the Poincar\'{e} inequality, we can prove 
\begin{lemma}\label{lem:4-3}
\cm{Assuming that for $i=1,\ldots,M$, ${\rm diam}(\omega_{i}^{\ast}) \leq 2\,{\rm diam}(\omega_{i})$}, and that
\begin{equation}\label{eq:4-5-1}
kV_{\rm max}C_{P} H^{\ast}_{\mathtt{max}} \leq a^{1/2}_{\rm min}/\sqrt{2},
\vspace{-1.25ex}
\end{equation}
then, 
\begin{equation}\label{eq:4-5-2}
\vspace{-1.25ex}
\eta (S_{n}(\Omega))\leq\sqrt{\zeta \zeta^{\ast}}\Big(\sqrt{2} d_{\mathtt{max}} (C_{\rm stab}(k) + 2\Lambda) + \cm{9C_{\chi}C_{P}\big(a_{\rm max}/a_{\rm min}\big)^{1/2} \Lambda}\Big),
\vspace{-0.5ex}
\end{equation}
where $\Lambda = C_{P}H^{\ast}_{\mathtt{max}}a^{-1/2}_{\rm min}$, and $C_{\chi}$ is given by \cref{eq:2-0-1}.
\end{lemma}
\begin{proof}
By \cref{ass:1-2}, for any $f\in L^{2}(\Omega)$, the problem \cref{eq:4-4-0} has a unique solution $\cm{\widehat{S}(f)}\in H^{1}_{D}(\Omega)$ with the estimate
\begin{equation}\label{eq:4-5-3}
\big \Vert \cm{\widehat{S}(f)} \big\Vert_{\mathcal{A},k}\leq C_{\rm stab}(k) \Vert f\Vert_{L^{2}(\Omega)}.
\end{equation}
On each oversampling domain $\omega_{i}^{\ast}$, we consider the following local Helmholtz problem:
\begin{equation}\label{eq:4-5-4}
{\rm Find} \;\;\hat{\psi}_{i}\in H^{1}_{DI}(\omega_{i}^{\ast}) \quad {\rm such\;\,that} \quad\mathcal{B}_{\omega_{i}^{\ast}}(\hat{\psi}_{i},v) = (f,v)_{L^{2}(\omega^{\ast}_{i})}\quad \forall v\in H^{1}_{DI}(\omega_{i}^{\ast}).
\end{equation}
Note that $\hat{\psi}_{i}$ satisfies the homogeneous Dirichlet boundary conditions on $\partial \omega_{i}^{\ast}\cap \Omega$. Under the assumption \cref{eq:4-5-1}, the local sesquilinear form in \cref{eq:4-5-4} is coercive. In fact, using \cref{eq:4-5-1} and the Poincar\'{e} inequality \cref{eq:4-5-0}, we have that for any $v\in H^{1}_{DI}(\omega_{i}^{\ast})$,
\begin{equation}\label{eq:4-5-5}
\begin{array}{lll}
{\displaystyle {\rm Re}\,\mathcal{B}_{\omega_{i}^{\ast}}(v, v) \geq \Vert v\Vert_{\mathcal{A},\omega_{i}^{\ast}}^{2} - k^{2}V^{2}_{\rm max}C_{P}^{2}\,[{\rm diam}(\omega_{i}^{\ast})]^{2} \Vert \nabla v\Vert^{2}_{L^{2}(\omega_{i}^{\ast})} }\\[2mm]
{\displaystyle \qquad \qquad \quad \;\;\geq \Vert v\Vert_{\mathcal{A},\omega_{i}^{\ast}}^{2} - a_{\rm min}\Vert \nabla v\Vert^{2}_{L^{2}(\omega_{i}^{\ast})}/2 \geq \Vert v\Vert_{\mathcal{A},\omega_{i}^{\ast}}^{2}/2.}
\end{array}
\end{equation}
Using the Poincar\'{e} inequality \cref{eq:4-5-0} again, it follows that
\begin{equation}\label{eq:4-5-6}
\begin{array}{lll}
{\displaystyle \Vert \hat{\psi}_{i} \Vert^{2}_{\mathcal{A},\omega_{i}^{\ast}}\leq 2C_{P}\,{\rm diam}(\omega_{i}^{\ast}) \Vert f\Vert_{L^{2}(\omega_{i}^{\ast})} \Vert \nabla \hat{\psi}_{i}\Vert_{L^{2}(\omega_{i}^{\ast})}}\\[2mm]
{\displaystyle \qquad \qquad \leq 2C_{P}\,{\rm diam}(\omega_{i}^{\ast})\,a^{-1/2}_{\rm min}\Vert f\Vert_{L^{2}(\omega_{i}^{\ast})} \Vert \hat{\psi}_{i} \Vert_{\mathcal{A},\omega_{i}^{\ast}}.}
\end{array}
\end{equation}
Denoting by $\Lambda = C_{P}H^{\ast}_{\mathtt{max}}a^{-1/2}_{\rm min}$, \cref{eq:4-5-6} leads to 
\begin{equation}\label{eq:4-5-7}
\Vert \hat{\psi}_{i} \Vert_{\mathcal{A},\omega_{i}^{\ast}}\leq 2\Lambda \Vert f\Vert_{L^{2}(\omega_{i}^{\ast})}, \quad {\rm and}\quad\; \Vert \hat{\psi}_{i} \Vert_{\mathcal{A},\omega_{i}^{\ast},k}\leq  3\Lambda \Vert f\Vert_{L^{2}(\omega_{i}^{\ast})}.
\end{equation}
Furthermore, combining \cref{eq:4-4-0,eq:4-5-4}, we see that $\widehat{S}(f)|_{\omega_{i}^{\ast}}-\hat{\psi}_{i} \in H_{\mathcal{B}}(\omega_{i}^{\ast})$. Recalling the definition of the $n$-width and using \cref{eq:4-5-7}, it follows that there exists $\varphi_{i}\in S_{n_i}(\omega_i)$ such that
\begin{equation}\label{eq:4-5-8}
\begin{array}{lll}
{\displaystyle \big \Vert \chi_{i}\big(\widehat{S}(f)-\hat{\psi}_{i}-\varphi_{i}\big)\big\Vert_{\mathcal{A},\omega_{i},k}\leq d_{n_i}(\omega_i,\omega_{i}^{\ast})\big\Vert \widehat{S}(f)-\hat{\psi}_{i}\big\Vert_{\mathcal{A},\omega^{\ast}_{i}} }\\[2mm]
{\displaystyle \quad \leq d_{n_i}(\omega_i,\omega_{i}^{\ast})\big(\big\Vert \widehat{S}(f) \big\Vert_{\mathcal{A},\omega^{\ast}_{i}} +  2\Lambda\Vert f\Vert_{L^{2}(\omega_{i}^{\ast})}\big).}
\end{array}
\end{equation}
Define $\hat{u}^{p} = \sum^{M}_{i=1}\chi_{i}\hat{\psi}_{i}$ and $\varphi = \sum^{M}_{i=1}\chi_{i}\varphi_{i}\in S_{n}(\Omega)$. Using \cref{eq:4-5-3} and a similar argument as in the proof of \cref{lem:4-1}, we get
\begin{equation}\label{eq:4-5-9}
\begin{array}{lll}
{\displaystyle \big \Vert \widehat{S}(f)-\hat{u}^{p}-\varphi \big\Vert_{\mathcal{A},k}\leq \sqrt{2\zeta \zeta^{\ast}} d_{\mathtt{max}} \big(\big\Vert \widehat{S}(f) \big\Vert_{\mathcal{A}} +2\Lambda \Vert f\Vert_{L^{2}(\Omega)}\big)} \\[2mm]
{\displaystyle \qquad \qquad \qquad \quad \;\;\leq \sqrt{2\zeta \zeta^{\ast}} d_{\mathtt{max}} \big(C_{\rm stab}(k) + 2\Lambda \big) \Vert f\Vert_{L^{2}(\Omega)},}
\end{array}
\end{equation}
and consequently
\begin{equation}\label{eq:4-5-10}
\big \Vert \widehat{S}(f)-\varphi \big \Vert_{\mathcal{A},k} \leq \Vert\hat{u}^{p} \Vert_{\mathcal{A},k} +\sqrt{2\zeta \zeta^{\ast}} d_{\mathtt{max}} \big(C_{\rm stab}(k) + 2\Lambda \big) \Vert f\Vert_{L^{2}(\Omega)}.
\end{equation}
It remains to estimate $\Vert\hat{u}^{p} \Vert_{\mathcal{A},k}$. By definition, we see that
\begin{equation}\label{eq:4-5-11}
\Vert \hat{u}^{p} \Vert^{2}_{\mathcal{A},k}\leq \zeta\sum_{i=1}^{M}\Vert \chi_{i}\hat{\psi}_{i}\Vert^{2}_{\mathcal{A},\omega_{i},k}\leq \zeta\sum_{i=1}^{M}\big(\Vert \chi_{i}\hat{\psi}_{i}\Vert^{2}_{\mathcal{A},\omega_{i}} + k^{2}\Vert V\hat{\psi}_{i}\Vert^{2}_{L^{2}(\omega_i)}\big).
\end{equation}
A use of the triangle inequality gives that
\begin{equation}\label{eq:4-5-12}
\begin{array}{lll}
{\displaystyle \Vert \chi_{i}\hat{\psi}_{i}\Vert_{\mathcal{A},\omega_{i}} \leq \Vert \hat{\psi}_{i}\Vert_{\mathcal{A},\omega_{i}} + a^{1/2}_{\rm max}\Vert \nabla\chi_{i}\Vert_{L^{\infty}(\omega_i)}\Vert \hat{\psi}_{i}\Vert_{L^{2}(\omega_i)}.}
\end{array}
\end{equation}
\cm{Noting that $\hat{\psi}\in H^{1}_{DI}(\omega_{i}^{\ast})$, we can use the Poincar\'{e} inequality \cref{eq:4-5-0} again to prove
\begin{equation}\label{eq:rev-4-1}
    \Vert \hat{\psi}_{i}\Vert_{L^{2}(\omega_i)} \leq \Vert \hat{\psi}_{i}\Vert_{L^{2}(\omega^{\ast}_i)} \leq C_{P}a^{-1/2}_{\rm min}\,{\rm diam}(\omega_{i}^{\ast}) \Vert \hat{\psi}_{i}\Vert_{\mathcal{A},\omega^{\ast}_{i}}.
\end{equation}
Combining \cref{eq:rev-4-1}, \cref{eq:2-0-1}, and the assumption that ${\rm diam}(\omega_{i}^{\ast}) \leq 2\,{\rm diam}(\omega_{i})$, the second term on the right-hand side of \cref{eq:4-5-12} can be bounded as follows.
\begin{equation}\label{eq:rev-4-2}
 a^{1/2}_{\rm max}\Vert \nabla\chi_{i}\Vert_{L^{\infty}(\omega_i)}\Vert \hat{\psi}_{i}\Vert_{L^{2}(\omega_i)}\leq 2C_{\chi}C_{P}\big(a_{\rm max}/a_{\rm min}\big)^{1/2} \Vert \hat{\psi}_{i}\Vert_{\mathcal{A},\omega^{\ast}_{i}}.   
\end{equation}
Without loss of generality, let us assume that $C_{\chi}C_{P}\big(a_{\rm max}/a_{\rm min}\big)^{1/2}\geq 1$. Inserting \cref{eq:rev-4-2} into \cref{eq:4-5-12} yields that 
\begin{equation}\label{eq:4-5-13}
\Vert \chi_{i}\hat{\psi}_{i}\Vert_{\mathcal{A},\omega_{i}}\leq 3C_{\chi}C_{P}\big(a_{\rm max}/a_{\rm min}\big)^{1/2} \Vert \hat{\psi}_{i}\Vert_{\mathcal{A},\omega^{\ast}_{i}}.
\end{equation}
Combining \cref{eq:4-5-7,eq:4-5-11,eq:4-5-13}, we come to
\begin{equation}\label{eq:4-5-14}
\begin{array}{lll}
{\displaystyle \Vert \hat{u}^{p} \Vert_{\mathcal{A},k}\leq 3C_{\chi}C_{P}\big(a_{\rm max}/a_{\rm min}\big)^{1/2} \Big(\zeta \sum_{i=1}^{M} \Vert\hat{\psi}_{i}\Vert^{2}_{\mathcal{A},\omega^{\ast}_{i},k}\Big)^{1/2} }\\[3mm]
{\displaystyle \qquad \quad \quad \leq 9C_{\chi}C_{P}\big(a_{\rm max}/a_{\rm min}\big)^{1/2} \sqrt{\zeta\zeta^{\ast}} \Lambda\Vert f\Vert_{L^{2}(\Omega)}. }
\end{array}
\end{equation}
}
Substituting \cref{eq:4-5-14} into \cref{eq:4-5-10} yields that
\begin{equation*}\label{eq:4-5-15}
\begin{array}{lll}
{\displaystyle \frac{\big \Vert \widehat{S}(f) -\varphi \big\Vert_{\mathcal{A},k}}{\Vert f\Vert_{L^{2}(\Omega)}} \leq \sqrt{\zeta \zeta^{\ast}}\Big(\sqrt{2} d_{\mathtt{max}} (C_{\rm stab}(k) + 2\Lambda) + 9C_{\chi}C_{P}\big(a_{\rm max}/a_{\rm min}\big)^{1/2} \Lambda\Big),}
\end{array}
\end{equation*}
and the desired estimate \cref{eq:4-5-2} follows.
\end{proof}

Based on \cref{thm:4-1,lem:4-3}, we see that some resolution conditions on $d_{\mathtt{max}}$ and $H^{\ast}_{\mathtt{max}}$ need to be imposed to obtain a quasi-optimal convergence rate for the method. To do this, we define a constant:
\begin{equation}\label{eq:4-5-16}
\Xi = C_{\mathcal{B}}V_{\rm max}\sqrt{\zeta \zeta^{\ast}}.
\end{equation} 

\begin{corollary}\label{cor:4-1}
Let $u^{\mathdutchcal{e}}$ be the solution of problem \cref{eq:1-2} and $u^{G}$ be the continuous MS-GFEM approximation, and \cm{let ${\rm diam}(\omega_{i}^{\ast}) \leq 2\,{\rm diam}(\omega_{i})$ for each $i=1,\ldots,M$}. Supposing that
\begin{equation}\label{eq:4-5-17}
d_{\mathtt{max}} \leq \big(8\sqrt{2}kC_{\rm stab}(k)\,\Xi\big)^{-1},\quad  \cm{H^{\ast}_{\mathtt{max}} \leq \big(36kC_{\chi}C^{2}_{P}a^{1/2}_{\rm max} a^{-1}_{\rm min}\,\Xi\big)^{-1}},
\vspace{-1ex}
\end{equation}
then
\begin{equation}\label{eq:4-5-18}
\vspace{-1ex}
\Vert u^{\mathdutchcal{e}}-u^{G}\Vert_{\mathcal{A},k} \leq 2C_{\mathcal{B}}\inf_{\varphi\in u^{p} + S_{n}(\Omega)}\Vert u^{\mathdutchcal{e}}-\varphi \Vert_{\mathcal{A},k}.
\end{equation}
\end{corollary}
\begin{proof}
The assumptions \cref{eq:4-5-1,eq:4-6} are implied by \cref{eq:4-5-17} and the result follows from \cref{lem:4-3,thm:4-1}.
\end{proof}
\begin{rem}
It follows from \cref{ass:1-2} and \cref{thm:3-1} that the first condition in \cref{eq:4-5-17} is satisfied if the sizes of the local subspaces grow polylogarithmically in $k$ (\cm{provided that the local eigenproblems are solved sufficiently accurately}). The second condition is equivalent to $Hk=O(H^{\ast}_{\mathtt{max}}k) =O(1)$, where $H$ denotes the size of the subdomains. Under these conditions, \cref{lem:4-1}, \cref{thm:3-1}, and \cref{eq:4-5-18} imply a nearly exponential convergence rate of the method. However, the second condition is very stringent in the high frequency regime, and our numerical results in \cref{sec-6} show that it is not necessary in practice to obtain near-exponential convergence.
\end{rem}

\section{Discrete MS-GFEM}\label{sec-4}
In the rest of this paper, we assume that $\Omega$ is a Lipschitz polyhedral domain for simplicity. Let $\tau_{h}=\{K\}$ be a shape-regular triangulation of $\Omega$ consisting of triangles (tetrahedra) or rectangles if $\Omega$ is a rectangular domain. The mesh size 
$h:=\max_{K\in \tau_{h}} {\rm diam}(K)$ is assumed to be sufficiently small to resolve the high frequency features of the wave and the fine-scale details of the coefficients.
Let $U_{h}\subset H^{1}(\Omega)$ be a standard Lagrange finite element space. For simplicity, we take $U_{h}$ to be the space consisting of continuous piecewise linear functions or the $Q1$ bilinear space for a rectangular subdivision. Let $U_{h,D} = U_{h}\cap H_{D}^{1}(\Omega)$. The standard finite element method for the problem \cref{eq:1-2} is: Find $u^{\mathdutchcal{e}}_{h}\in U_{h,D}$ such that
\begin{equation}\label{eq:5-1}
\mathcal{B}(u^{\mathdutchcal{e}}_{h},v_{h}) = F(v_{h})\qquad \forall v_{h}\in U_{h,D}.
\end{equation}

In what follows, we introduce the discrete MS-GFEM for solving the problem \cref{eq:5-1} in parallel with the continuous MS-GFEM in \cref{sec-2}. Let $\{\omega_{i}\}_{i=1}^{M}$ be an overlapping decomposition of $\Omega$ resolved by the mesh. We extend each subdomain $\omega_{i}$ by several layers of fine mesh elements to create a larger oversampling domain $\omega_{i}^{\ast}$, and define
\begin{equation}\label{eq:5-2}
\begin{array}{lll}
{\displaystyle  {U}_{h}(\omega^{\ast}_{i}) = \big\{v_{h}|_{\omega^{\ast}_{i}}\;:\; v_{h}\in U_{h}\big\},}\\[2mm]
{\displaystyle {U}_{h,D}(\omega^{\ast}_{i}) = \big\{v_{h}\in U_{h}(\omega_{i}^{\ast}):\; v_{h} = 0 \;\;{\rm on}\;\, \partial \omega^{\ast}_{i} \cap \Gamma_{D}\big\},  }\\[2mm]
{\displaystyle U_{h,DI}(\omega^{\ast}_i)= \big\{v_{h}\in {U}_{h}(\omega^{\ast}_{i}):\;v_{h} = 0 \;\;{\rm on}\;\, \partial \omega^{\ast}_{i} \cap (\Omega\cup\Gamma_{D})\big\}, }\\[2mm]
{\displaystyle H_{h,\mathcal{B}}(\omega^{\ast}_i)= \big\{u_{h}\in U_{h,D}(\omega^{\ast}_i)\;:\; \mathcal{B}_{\omega^{\ast}_{i}}(u_{h},v_{h}) = 0,\;\, \forall v_{h}\in  U_{h,DI}(\omega^{\ast}_i)\big\},}
\end{array}
\end{equation}
where $H_{h,\mathcal{B}}(\omega^{\ast}_i)$ is referred to as the \emph{discrete generalized harmonic space} and we note that $H_{h,\mathcal{B}}(\omega^{\ast}_i) \nsubseteq H_{\mathcal{B}}(\omega^{\ast}_i)$. Similar to \cref{eq:2-3-0}, there exists $C>0$ independent of $h$, such that for any $u_{h}\in H_{h,\mathcal{B}}(\omega^{\ast}_{i})$,
\begin{equation}\label{eq:5-2-0}
\Vert u_{h}\Vert_{L^{2}(\omega_{i}^{\ast})}\leq C\Vert \nabla u_{h} \Vert_{L^{2}(\omega_{i}^{\ast})}.
\end{equation}
The proof of \cref{eq:5-2-0} is given in \cref{sec:A.4}. Therefore, $\Vert \cdot\Vert_{\mathcal{A},\omega_{i}^{\ast}}$ is also a norm on $H_{h,\mathcal{B}}(\omega^{\ast}_{i})$ equivalent to the standard $H^{1}$ norm. Next we introduce the discrete local Helmholtz problem: Find $\psi_{h,i}\in {U}_{h,D}(\omega^{\ast}_{i})$ such that
\begin{equation}\label{eq:5-3}
\mathcal{B}_{\omega_{i}^{\ast}}(\psi_{h,i},v_{h}) - {\rm i}k\int_{\partial \omega^{\ast}_{i}\cap \Omega}V\psi_{h,i}\overline{v_{h}}\,d{\bm s}= F_{\omega_{i}^{\ast}}(v_{h}),\quad \forall v_{h}\in {U}_{h,D}(\omega^{\ast}_{i}).
\end{equation}

We proceed to construct the optimal spaces for approximating a discrete generalized harmonic function in the same spirit as before. Let $I_{h}:C(\Omega)\rightarrow U_{h}$ be the standard Lagrange interpolation operator. We define the operator 
\begin{equation}\label{eq:5-3-0}
P_{h,i}:  H_{h,\mathcal{B}}(\omega^{\ast}_i)\rightarrow U_{h,DI}(\omega_i)\quad {\rm such \;\,that} \quad P_{h,i}v_{h} = I_{h}(\chi_{i}v_{h}),
\end{equation}
where $\chi_{i}$ is the partition of unity function supported on $\omega_{i}$. For each $n\in \mathbb{N}$, we consider the Kolmogorov $n$-width of $P_{h,i}$ defined by
\begin{equation}\label{eq:5-5}
d_{h,n}(\omega_{i},\omega_{i}^{\ast}) = \inf_{Q(n)\subset U_{h,DI}(\omega_{i})}\sup_{u_{h}\in H_{h,\mathcal{B}}(\omega^{\ast}_i)} \inf_{v_{h}\in Q(n)}\frac {\Vert P_{h,i}u_{h}-v_{h}\Vert_{\mathcal{A},\omega_{i},k}}{\Vert u_{h} \Vert_{\mathcal{A},\omega_{i}^{\ast}}}.
\end{equation}
Similar to \cref{lem:2-1}, we have the following characterization of the $n$-width. The proof is omitted.
\begin{lemma}\label{lem:5-1}
For each $j\in\mathbb{N}$, let $(\lambda_{h,j},\phi_{h,j})$ be the $j$-th eigenpair (arranged in decreasing order) of the problem 
\begin{equation}\label{eq:5-6}
\mathcal{A}_{\omega^{\ast}_{i},k}(I_{h}(\chi_{i} \phi_{h}), I_{h}(\chi_{i} v_{h})) = \lambda_{h}\,\mathcal{A}_{\omega_{i}^{\ast}}(\phi_{h}, v_{h}),\quad \forall v_{h}\in H_{h,\mathcal{B}}(\omega_{i}^{\ast}).
\end{equation}
Then $d_{h,n}(\omega_{i},\omega_{i}^{\ast}) =\lambda^{1/2}_{h,n+1}$ and the optimal approximation space is given by 
\begin{equation}\label{eq:5-7}
\hat{Q}(n) = {\rm span}\big\{I_{h}(\chi_{i} \phi_{h,1}), \cdots, I_{h}(\chi_{i} \phi_{h,n})\big\}.
\end{equation}
\end{lemma}
\begin{rem}
The way that we define the operator $P_{h,i}$ in this paper also works for the positive definite case in \cite{ma2021error} where $P_{h,i}$ was defined in a  slightly different way without involving the interpolation operator.
\end{rem}

Before defining the local particular functions and the local approximation spaces for the discrete MS-GFEM, we make some assumptions on the well-posedness of the discrete problems \cref{eq:5-1,eq:5-3} analogously to the continuous level.
\begin{assumption}\label{ass:5-1}
There exists $h_{0}>0$ such that for any $0<h<h_{0}$, 

\vspace{2mm}
\begin{itemize}
\item[(i)] the problem \cref{eq:5-1} has a unique solution $u^{{\mathdutchcal{e}}}_{h}\in U_{h,D}$, and there exists $\widetilde{C}_{\rm stab}(k)$ depending polynomially on $k$ such that
\begin{equation}\label{eq:5-2-2}
\Vert u^{{\mathdutchcal{e}}}_{h}\Vert_{\mathcal{A},k}\leq \widetilde{C}_{\rm stab}(k) (\Vert f\Vert_{L^{2}(\Omega)} + \Vert g\Vert_{L^{2}(\Gamma_{R})}).
\end{equation}
\item[(ii)] for each $i=1,\cdots,M$, the problem \cref{eq:5-3} is uniquely solvable in ${U}_{h,D}(\omega^{\ast}_{i})$, and there exists $\widetilde{C}_{\rm stab}^{i}(k)$ depending polynomially on $k$ such that
\begin{equation}\label{eq:5-4}
\Vert \psi_{h,i}\Vert_{\mathcal{A},\omega_{i}^{\ast},k}\leq \widetilde{C}^{i}_{\rm stab}(k)(\Vert f \Vert_{L^{2}(\omega_{i}^{\ast})} + \Vert g\Vert_{L^{2}(\partial \omega^{\ast}_{i} \cap \Gamma_{R})}).
\end{equation}
\end{itemize}
\end{assumption}
\begin{rem}\label{rem:5-1}
The unique solvability of the discrete problem \cref{eq:5-1} can be implied by that of the continuous problem \cref{eq:1-2}, as well as that of the continuous adjoint problem with the right-hand term in a weaker dual space; see the remark after \cref{lem:5-4}. Similar results hold for the local problems \cref{eq:5-3}. \cm{Furthermore, it can be proved that for fine meshes, the stability constants $\widetilde{C}_{\rm stab}(k)$ and $\widetilde{C}^{i}_{\rm stab}(k)$ are close to their continuous counterparts, i.e., ${C}_{\rm stab}(k)$ and ${C}^{i}_{\rm stab}(k)$; see, e.g., \cite[Proposition 8.2.7]{melenk1995generalized}, and \cite[Theorem 2]{schatz1996some}}.  
\end{rem}

\begin{theorem}\label{thm:5-1}
Let the local particular function and the local approximation space on $\omega_{i}$ be defined as
\begin{equation}\label{eq:5-8}
u^{p}_{h,i}:=\psi_{h,i}|_{\omega_{i}}\;\;\; {\rm and}\;\;\; S_{h,n_{i}}(\omega_{i}) := {\rm span}\{\phi_{h,1}|_{\omega_{i}},\cdots, \phi_{h,n_{i}}|_{\omega_{i}}\},
\end{equation}
where $\psi_{h,i}$ is the solution of \cref{eq:5-3} and $\phi_{h,j}$ denotes the $j$-th eigenfunction of the problem \cref{eq:5-6}. Then, 
\begin{equation}\label{eq:5-9}
\inf_{\varphi_{h}\in u^{p}_{h,i} + S_{h,n_{i}}(\omega_{i})}\big\Vert I_{h}\big(\chi_{i}(u^{\mathdutchcal{e}}_{h} - \varphi_{h})\big)\big\Vert_{\mathcal{A},\omega_{i},k}\leq d_{h,n_{i}}(\omega_{i},\omega_{i}^{\ast})\,\Vert u^{{\mathdutchcal{e}}}_{h}- \psi_{h,i}\Vert_{\mathcal{A},\omega_{i}^{\ast}}.
\end{equation}
\end{theorem}
\begin{proof}
Noting that $u^{\mathdutchcal{e}}_{h}|_{\omega_{i}^{\ast}}-\psi_{h,i} \in H_{h,\mathcal{B}}(\omega^{\ast}_i)$, \cref{eq:5-9} follows from \cref{lem:5-1} and the definition of the $n$-width.
\end{proof}

\cm{
\begin{rem}
In practice, the local eigenproblems \cref{eq:5-6} are solved using a similar technique as in \cite{ma2021error} where a Lagrange multiplier was introduced to relax the generalized harmonic constraint. Specifically, we solve the eigenproblems in mixed formulation: Find $\lambda_{h}\in\mathbb{R}$, $\phi_{h}\in U_{h,D}(\omega_{i}^{\ast})$, and $p_{h}\in U_{h,DI}(\omega_{i}^{\ast})$ such that
\begin{equation}\label{eq:5-9-1}
\begin{aligned}
\mathcal{A}_{\omega_{i}^{\ast}}(\phi_h,v_{h}) + \mathcal{B}_{\omega_{i}^{\ast}}(v_h,p_h)  =&\, \lambda^{-1}_{h}\mathcal{A}_{\omega^{\ast}_{i},k}\big(I_{h}(\chi_{i} \phi_{h}), I_{h}(\chi_{i} v_h)\big) \;\;\;\, \forall v_h\in U_{h,D}(\omega_{i}^{\ast}),\\
\mathcal{B}_{\omega_{i}^{\ast}}(\phi_{h},\xi_h) =&\;0\qquad \qquad\qquad\qquad\qquad\qquad\quad \;\forall \xi_{h}\in U_{h,DI}(\omega_{i}^{\ast}).
\end{aligned}
\end{equation}
It can be shown that if $h$ is sufficiently small such that the projection from $H_{D}^{1}(\omega_{i}^{\ast})$ onto $U_{h,D}(\omega_{i}^{\ast})$ with respect to the form $\mathcal{B}_{\omega_{i}^{\ast}}(\cdot,\cdot)$ is well defined, then the matrix on the left-hand side of \cref{eq:5-9-1} is invertible; see \cref{lem:5-6} below.
\end{rem}
}

Now we proceed to define the global particular function and the trial space for the discrete MS-GFEM:
\begin{equation}\label{eq:5-10}
 u_{h}^{p} :=\sum_{i=1}^{M}I_{h}(\chi_{i}u^{p}_{h,i})\;\;\;{\rm and}\;\; \; S_{h}(\Omega):= \Big\{ \sum_{i=1}^{M} I_{h}(\chi_{i}v_{h,i}):\; v_{h,i}\in S_{h,n_i}(\omega_{i}) \Big\}. 
\end{equation}
The last step is to solve the discrete problem on $S_{h}(\Omega)$: Find $u^{G}_{h} = u_{h}^{p} + u_{h}^{s}$, where $u_{h}^{s}\in S_{h}(\Omega)$, such that
\begin{equation}\label{eq:5-11}
\mathcal{B}(u_{h}^{G},\,v_{h}) = F(v_{h}) \quad \forall v_{h}\in S_{h}(\Omega).
\end{equation}
By the definition of $S_{h}(\Omega)$, we see that $S_{h}(\Omega)\subset U_{h,D}$, and thus the discrete MS-GFEM delivers a conforming approximation of the problem \cref{eq:5-1}. As in the continuous case, we have the Galerkin orthogonality:
\begin{equation}\label{eq:5-12}
\mathcal{B}(u^{\mathdutchcal{e}}_{h}-u_{h}^{G},\,v_{h}) = 0\quad \forall v_{h}\in S_{h}(\Omega).
\end{equation}

\subsection{Technical tools}
In this subsection, we present some technical tools that will be used for proving the convergence of the discrete MS-GFEM in the next section. We start with the following superapproximation estimates.
\begin{lemma}[\cite{demlow2011local}]\label{lem:5-2}
Assume that $\eta\in C^{\infty}(\Omega)$ satisfying $|\eta|_{W^{j,\infty}(\Omega)}\leq C\delta^{-j}$ for $0\leq j\leq 2$. Then for each $u_{h}\in U_{h}$ and $K\in \tau_{h}$ with $h_{K} :={\rm diam}(K)\leq \delta$, 
\begin{align}
{\displaystyle \Vert \eta^{2}u_{h}- I_{h}(\eta^{2}u_{h})\Vert_{H^{1}(K)}\leq C\big(\frac{h_{K}}{\delta}\Vert \nabla (\eta u_{h})\Vert_{L^{2}(K)}+\frac{h_{K}}{\delta^{2}}\Vert u_{h}\Vert_{L^{2}(K)}\big),\label{eq:5-13}}\\
{\displaystyle \Vert \eta^{2}u_{h}- I_{h}(\eta^{2}u_{h})\Vert_{L^{2}(K)}\leq C\big(\frac{h^{2}_{K}}{\delta}\Vert \nabla(\eta u_{h})\Vert_{L^{2}(K)}+\frac{h^{2}_{K}}{\delta^{2}}\Vert u_{h}\Vert_{L^{2}(K)}\big),\label{eq:5-14}}
\end{align}
where $I_{h}$ is the standard Lagrange interpolation operator.
\end{lemma}

Estimate \cref{eq:5-13} was proved in Theorem 2.1 of \cite{demlow2011local} and \cref{eq:5-14} can be proved using exactly the same argument; see also \cite[Chapter 3]{wahlbin2006superconvergence}. Next we give the multiplicative trace inequality with an explicit dependence on the diameter of the domain. \cm{A proof can be found in \cite[Lemma 3.12]{stephansen2007methodes}; see also \cite[Theorem 4.1]{carstensen2000constants}}.

\begin{lemma}\label{lem:5-3}
Let $u\in H^{1}(\Omega)$. Then for each $K\in\tau_{h}$, there exists $C>0$ depending only on the shape regularity of the mesh such that
\begin{equation}\label{eq:5-15}
\Vert u\Vert^{2}_{L^{2}(\partial K)}\leq C\big(\Vert u\Vert_{L^{2}(K)}\Vert \nabla u\Vert_{L^{2}(K)} + h^{-1}_{K}\Vert u\Vert^{2}_{L^{2}(K)}\big).
\end{equation}
\end{lemma}
\begin{rem}\label{rem:5-2}
A similar result as \cref{eq:5-15} for the domain $\Omega$ is as follows. For any $u\in H^{1}(\Omega)$, there exists a constant $C$ depending only on the shape of $\Omega$, such that
\begin{equation}\label{eq:5-21}
\Vert u\Vert^{2}_{L^{2}(\partial \Omega)}\leq C\big(\Vert u\Vert_{L^{2}(\Omega)}\Vert \nabla u\Vert_{L^{2}(\Omega)} + {\rm diam}(\Omega)^{-1}\Vert u\Vert^{2}_{L^{2}(\Omega)}\big).
\end{equation}
\end{rem}

The following lemma gives a uniform approximation result for compact subsets of $H^{1}_{DI}(\omega)$ in $U_{h,DI}(\omega)$, which will be used in the proof of the convergence of eigenvalues.
\begin{lemma}[\cite{schatz1996some}]\label{lem:5-7}
Let $\mathcal{S}$ be a fixed compact subset of $H^{1}_{DI}(\omega)$. For any $\varepsilon>0$, there exists $h_{0} =h_{0}(\mathcal{S},\varepsilon)$ such that if $0< h\leq h_{0}$, for each $v\in \mathcal{S}$, there exists a $v_{h} \in U_{h,DI}(\omega)$ satisfying
\begin{equation}\label{eq:5-40}
\Vert v-v_{h}\Vert_{H^{1}(\omega)}\leq \varepsilon.
\end{equation}
\end{lemma}

To prove the convergence of eigenvalues, we also need some local projections. For each $i=1,\cdots,M$, let us define $\Pi^{i}_{h}: H_{D}^{1}(\omega_{i}^{\ast})\rightarrow U_{h,D}(\omega_{i}^{\ast})$ by
\begin{equation}\label{eq:5-22}
\mathcal{B}_{\omega_{i}^{\ast}}(\Pi^{i}_{h}u, v_{h}) = \mathcal{B}_{\omega_{i}^{\ast}}(u, v_{h}) \quad \forall v_{h}\in U_{h,D}(\omega_{i}^{\ast}).
\end{equation}
To ensure the stability and well-definedness of $\Pi^{i}_{h}$, an assumption on the continuous adjoint problem is needed.
\begin{assumption}\label{ass:5-3}
For any $\mathcal{G}\in \big(H^{1}_{D}(\omega_{i}^{\ast})\big)^{\prime}$, the adjoint problem
\begin{equation}\label{eq:5-23}
B_{\omega_{i}^{\ast}}(v, u) = \mathcal{G}(v)\quad \forall v\in H_{D}^{1}(\omega_{i}^{\ast})
\end{equation}
has a unique solution $u\in H_{D}^{1}(\omega_{i}^{\ast})$, and there exists $\cm{C(k)}>0$ such that
\begin{equation}\label{eq:5-24}
\Vert u\Vert_{\mathcal{A},\omega_{i}^{\ast},k}\leq \cm{C(k)}\Vert \mathcal{G}\Vert_{(H^{1}_{D}(\omega_{i}^{\ast}))^{\prime}}.
\end{equation}
\end{assumption}

\begin{lemma}\label{lem:5-4}
Under \cref{ass:5-3}, there exists an $h_{0}(k)>0$, such that for any $0<h<h_{0}(k)$, $\Pi^{i}_{h}$ is well-defined and satisfies 
\begin{equation}\label{eq:5-25}
\Vert \Pi^{i}_{h}u \Vert_{\mathcal{A},\omega_{i}^{\ast},k}\leq C \Vert u\Vert_{\mathcal{A},\omega_{i}^{\ast},k},\;\;\, \Vert u - \Pi^{i}_{h}u\Vert_{\mathcal{A},\omega_{i}^{\ast},k} \leq C\inf_{v_{h}\in U_{h,D}(\omega_{i}^{\ast})}  \Vert u - v_{h}\Vert_{\mathcal{A},\omega_{i}^{\ast},k}
\vspace{-1ex}
\end{equation}
for any $u\in H_{D}^{1}(\omega_{i}^{\ast})$, where $C>0$ is independent of $h$ and $k$.
\end{lemma}
\begin{proof}
Since \cref{eq:5-22} is a finite-dimensional linear system, the estimate \cref{eq:5-25} yields the unique solvability of \cref{eq:5-22}. Therefore, it suffices to prove \cref{eq:5-25}. Under \cref{ass:5-3}, it follows from \cite[Theorem 2]{schatz1996some} that for any $\varepsilon>0$, there exists an $h_{0}(\varepsilon,k)>0$, such that for any $0<h<h_{0}(\varepsilon,k)$, 
\begin{align}
{\displaystyle \Vert u - \Pi^{i}_{h}u\Vert_{L^{2}(\omega_{i}^{\ast})}\leq \varepsilon \Vert u - \Pi^{i}_{h}u\Vert_{\mathcal{A},\omega_{i}^{\ast},k},\label{eq:5-26}}\\
{\displaystyle \Vert u - \Pi^{i}_{h}u\Vert_{\mathcal{A},\omega_{i}^{\ast},k} \leq C\inf_{v_{h}\in U_{h,D}(\omega_{i}^{\ast})}  \Vert u - v_{h}\Vert_{\mathcal{A},\omega_{i}^{\ast},k}.\label{eq:5-26-0}}
\end{align}
Therefore, the second part of \cref{eq:5-25} is proved. Next by choosing $v_{h} = \Pi^{i}_{h}u$ in \cref{eq:5-22} and taking the real part of the equation,  we see that
\begin{equation}\label{eq:5-27}
\Vert\Pi^{i}_{h}u \Vert_{\mathcal{A},\omega_{i}^{\ast},k}^{2} \leq C\Vert\Pi^{i}_{h}u \Vert_{\mathcal{A},\omega_{i}^{\ast},k} \Vert u \Vert_{\mathcal{A},\omega_{i}^{\ast},k} + 2k^{2}V^{2}_{\rm max}\Vert\Pi^{i}_{h}u\Vert^{2}_{L^{2}(\omega_{i}^{\ast})}.
\end{equation}
To bound the second term on the right-hand side of \cref{eq:5-27}, we use \cref{eq:5-26} to deduce
\begin{equation}\label{eq:5-28}
\Vert\Pi^{i}_{h}u\Vert_{L^{2}(\omega_{i}^{\ast})}\leq \Vert u\Vert_{L^{2}(\omega_{i}^{\ast})} + \varepsilon \big(\Vert\Pi^{i}_{h}u \Vert_{\mathcal{A},\omega_{i}^{\ast},k} +\Vert u \Vert_{\mathcal{A},\omega_{i}^{\ast},k}\big).
\end{equation}
Inserting \cref{eq:5-28} into \cref{eq:5-27} and taking $\varepsilon$ sufficiently small such that $\varepsilon< (2kV_{\rm max})^{-1}$, we get the first part of \cref{eq:5-25}.
\end{proof}
\begin{rem}\label{rem:5-3}
Using the same argument and a condition similar to  \cref{ass:5-3} on the global adjoint problem, the discrete problem \cref{eq:5-1} is uniquely solvable  for $h$ sufficiently small provided the continuous problem \cref{eq:1-2} is uniquely solvable.
\end{rem}

\begin{rem}
\cm{A $k$-explicit resolution condition on the FE discretization that ensures quasi-optimality of the projection \cref{eq:5-22}, is an important issue and related studies date back at least to \cite{aziz1988two}. A classical technique to derive this condition is the so-called Schatz argument \cite{schatz1974observation}. In essence, it is the approach used in \cref{thm:4-1} with the key argument being the derivation of the estimate in \cref{eq:5-26}. In the case of smooth coefficients and regular geometries, this estimate can be proved by combining standard duality arguments, $k$-explicit regularity results for an adjoint problem (see \cref{eq:4-10}) and standard FE error estimates. In the linear FEM case, it is well known that $h_{0}(k)= O(k^{-2})$ \cite{melenk1995generalized}. Higher-order FEMs result in less restrictive resolution conditions in terms of the dependence on $k$ \cite{lafontaine2022wavenumber,melenk2011wavenumber}. However, in the case of general heterogeneous coefficients, the above proof fails due to the very low regularity (typically $H^{1}$) of the solution to the adjoint problem. This difficulty was circumvented in \cite{schatz1996some} by combining the fact that the solution space of the adjoint problem is a compact set in $H^{1}$ and the uniform approximability of such compact sets in FE spaces with sufficiently small $h$ (\cref{lem:5-7}). However, an explicit dependence of $h_{0}(k)$ on $k$ is not available that way. Furthermore, in view of the very low regularity of the solution, it is unclear whether or not using higher-order FEMs can relax the resolution condition in general.}  
\end{rem}

We end this section by discussing the well-posedness of a saddle point problem arising from an elliptic BVP defined on the generalized harmonic space and its FE approximation.
\begin{lemma}\label{lem:5-5}
Given $\mathcal{F} \in (H^{1}_{D}(\omega^{\ast}_{i}))^{\prime}$, consider the problem of finding $u\in H^{1}_{D}(\omega_{i}^{\ast})$ and $p\in H^{1}_{DI}(\omega_{i}^{\ast})$ such that
\begin{equation}\label{eq:5-29}
\begin{aligned}
\mathcal{A}_{\omega_{i}^{\ast}}(u,v) + \mathcal{B}_{\omega_{i}^{\ast}}(v,p) =&\, \mathcal{F}(v)\quad \forall v\in H^{1}_{D}(\omega_{i}^{\ast}),\\
\mathcal{B}_{\omega_{i}^{\ast}}(u,\xi) =&\;0\qquad\;\; \forall \xi\in H^{1}_{DI}(\omega_{i}^{\ast}).
\end{aligned}
\end{equation}
Under \cref{ass:5-3}, there exists a unique solution $(u,p)$ to \cref{eq:5-29} and 
\begin{equation}\label{eq:5-30}
\Vert u\Vert_{\mathcal{A},\omega_{i}^{\ast},k} + \Vert p\Vert_{\mathcal{A},\omega_{i}^{\ast},k}\leq C\Vert \mathcal{F}\Vert_{(H^{1}_{D}(\omega_{i}^{\ast}))^{\prime}}.
\end{equation}
\end{lemma}
\begin{proof}
By \cref{eq:2-3-0}, we see that the sesquilinear form $\mathcal{A}_{\omega_{i}^{\ast}}(\cdot,\cdot)$ is coercive on $H_{\mathcal{B}}(\omega_{i}^{\ast})$. \cm{Moreover, a combination of \cref{ass:5-3} and the fact that $H^{1}_{DI}(\omega^{\ast}_{i})\subset H^{1}_{D}(\omega^{\ast}_{i})$ yields the following inf-sup condition:
\begin{equation*}
\inf_{p\in H^{1}_{DI}(\omega^{\ast}_{i})} \sup_{v\in H^{1}_{D}(\omega^{\ast}_{i})} \frac{|\mathcal{B}_{\omega_{i}^{\ast}}(v,p)|}{\Vert v\Vert_{\mathcal{A},\omega_{i}^{\ast},k} \Vert p\Vert_{\mathcal{A},\omega_{i}^{\ast},k} }\geq \inf_{p\in H^{1}_{D}(\omega^{\ast}_{i})} \sup_{v\in H^{1}_{D}(\omega^{\ast}_{i})} \frac{|\mathcal{B}_{\omega_{i}^{\ast}}(v,p)|}{\Vert v\Vert_{\mathcal{A},\omega_{i}^{\ast},k} \Vert p\Vert_{\mathcal{A},\omega_{i}^{\ast},k} } \geq C.
\end{equation*}
}
Now the result follows immediately from \cite[Theorem 4.2.3]{boffi2013mixed}.
\end{proof}


\begin{lemma}\label{lem:5-6}
Given $\mathcal{F} \in (H^{1}_{D}(\omega^{\ast}_{i}))^{\prime}$, consider the discrete problem of finding $u_{h}\in U_{h,D}(\omega_{i}^{\ast})$ and $p_{h}\in U_{h,DI}(\omega_{i}^{\ast})$ such that
\begin{equation}\label{eq:5-34}
\begin{aligned}
\mathcal{A}_{\omega_{i}^{\ast}}(u_{h},v_{h}) + \mathcal{B}_{\omega_{i}^{\ast}}(v_{h},p_{h}) =&\, \mathcal{F}(v_{h})\quad \forall v_{h}\in U_{h,D}(\omega_{i}^{\ast}),\\
\mathcal{B}_{\omega_{i}^{\ast}}(u_{h},\xi_{h}) =&\;0\qquad\;\,\;\;\, \forall \xi_{h}\in U_{h,DI}(\omega_{i}^{\ast}).
\end{aligned}
\end{equation}
Under \cref{ass:5-3}, there exists $h_{0}>0$, such that for any $0<h<h_{0}$, the problem \cref{eq:5-34} has a unique solution $(u_{h},p_{h})$ with
\begin{equation}\label{eq:5-35}
\begin{array}{lll}
{\displaystyle \qquad \Vert u-u_h\Vert_{\mathcal{A},\omega_{i}^{\ast},k} + \Vert p-p_h\Vert_{\mathcal{A},\omega_{i}^{\ast},k}}\\[3mm]
{\displaystyle \leq C\big(\inf_{v_h\in U_{h,D}(\omega_{i}^{\ast})}\Vert u-v_h\Vert_{\mathcal{A},\omega_{i}^{\ast},k} + \inf_{q_h\in U_{h,DI}(\omega_{i}^{\ast})}\Vert p-q_h\Vert_{\mathcal{A},\omega_{i}^{\ast},k}\big),}
\end{array} 
\end{equation}
where $(u,p)$ denotes the solution of \cref{eq:5-29}.
\end{lemma}
\begin{proof}
The coerciveness of $\mathcal{A}_{\omega_{i}^{\ast}}(\cdot,\cdot)$ on $H_{h,\mathcal{B}}(\omega_{i}^{\ast})$ is implied by \cref{eq:5-2-0} and thus it suffices to prove the discrete inf-sup condition. For any $p_{h}\in U_{h,DI}(\omega_{i}^{\ast})$, we consider the discrete problem of finding $w_{h}\in U_{h,D}(\omega_{i}^{\ast})$ such that
\begin{equation}\label{eq:5-36}
\mathcal{B}_{\omega_{i}^{\ast}}(w_{h}, v_{h}) = \mathcal{A}_{\omega_{i}^{\ast},k}(p_{h},v_{h})\quad \forall v_{h}\in U_{h,D}(\omega_{i}^{\ast}),
\end{equation}
\cm{which is uniquely solvable for $h$ sufficiently small due to  \cref{lem:5-4} and}
\begin{equation}\label{eq:5-39}
\cm{\Vert w_{h} \Vert_{\mathcal{A},\omega_{i}^{\ast},k}\leq C\Vert p_{h}\Vert_{\mathcal{A},\omega_{i}^{\ast},k} \, .}
\end{equation}
Combining \cref{eq:5-36,eq:5-39} gives the discrete inf-sup condition, and the result is an immediate consequence of \cite[Theorem 5.2.2]{boffi2013mixed}.
\end{proof}

\section{Convergence analysis of the discrete MS-GFEM}\label{sec-5}
In this section, we first prove error estimates for the discrete MS-GFEM and then show that the eigenvalues of the discrete eigenproblems converge towards those of the continuous eigenproblems as $h\rightarrow 0$.
\subsection{Local and global error estimates}
The key to deriving a nearly exponential convergence rate for the local approximations of the discrete MS-GFEM is a discrete Caccioppoli inequality which is proved in detail below.
\begin{lemma}[Discrete Caccioppoli inequality]\label{lem:6-1}
Let $\omega\subset\omega^{\ast}$ be subdomains of $\Omega$ with $\delta:={\rm dist}\,\big(\omega, \, \partial \omega^{\ast}\setminus\partial \Omega\big)>0$. In addition, let $\max_{K\cap \omega^{\ast}\neq \emptyset}h_{K}\leq \min\{\frac{1}{2}\delta, \,k^{-1}\}$. Then, for each $u_{h}\in H_{h,\mathcal{B}}(\omega^{\ast})$,
\begin{equation}\label{eq:6-1}
\Vert u_{h}\Vert_{\mathcal{A},\omega}\leq C\delta^{-1}\Vert u_{h} \Vert_{L^{2}(\omega^{\ast})} + \sqrt{2}kV_{\rm max}\Vert u_{h} \Vert_{L^{2}(\omega^{\ast})},
\end{equation}
where $C$ depends only on $d$, $\cm{\Vert \beta\Vert_{L^{\infty}(\Gamma_R)}}$, the (spectral) bounds of the coefficients $A$ and $V$, and the shape regularity of the mesh.
\end{lemma}
\begin{proof}
Let $\eta\in C^{\infty}(\omega^{\ast})$ be a cut-off function satisfying $0\leq\eta\leq 1$ and 
\begin{equation}\label{eq:6-2}
\eta =1\;\;\; {\rm in} \;\;\,\omega,\quad \eta = 0\;\; \;{\rm on}\;\;\,\partial \omega^{\ast}\setminus\partial \Omega,\quad |\eta|_{W^{j,\infty}(\omega^{\ast})}\leq C\delta^{-j},\;\;\;j=1,\,2.
\end{equation}
Using \cref{eq:2-9} with $u=v=u_{h}$ gives that
\begin{equation}\label{eq:6-3}
\begin{array}{lll}
{\displaystyle \Vert \eta u_{h}\Vert^{2}_{\mathcal{A},\omega^{\ast}}=\int_{\omega^{\ast}}(A\nabla \eta \cdot \nabla \eta) |u_{h}|^{2}\,d{\bm x} + {\rm Re}\big[\mathcal{A}_{\omega^{\ast}}(u_{h}, \eta^{2}u_{h})\big] }\\[3mm]
{\displaystyle \qquad  = \int_{\omega^{\ast}}(A\nabla \eta \cdot \nabla \eta + k^{2}V^{2}\eta^{2}) |u_{h}|^{2}\,d{\bm x} + {\rm Re}\big[\mathcal{B}_{\omega^{\ast}}(u_{h}, \eta^{2}u_{h})\big].}
\end{array}
\end{equation}
Since $u_{h}\in H_{h,\mathcal{B}}(\omega^{\ast})$, we see that $\mathcal{B}_{\omega^{\ast}}(u_{h}, I_{h}(\eta^{2}u_{h})) = 0$ and thus
\begin{equation}\label{eq:6-4}
\begin{array}{lll}
{\displaystyle  \mathcal{B}_{\omega^{\ast}}(u_{h}, \eta^{2}u_{h}) = \mathcal{B}_{\omega^{\ast}}(u_{h}, \eta^{2}u_{h}-I_{h}(\eta^{2}u_{h})) }\\[3mm]
{\displaystyle = \,\mathcal{A}_{\omega^{\ast}}(u_{h}, \eta^{2}u_{h}-I_{h}(\eta^{2}u_{h})) + k^{2}\int_{\omega^{\ast}}V^{2}u_{h}\big(\eta^{2}u_{h}-I_{h}(\eta^{2}u_{h})\big)\,d{\bm x} }\\[3mm]
{\displaystyle \quad -\,{\rm i}k\int_{\partial \omega^{\ast}\cap \Gamma_{R}}\beta u_{h}\big(\eta^{2}u_{h}-I_{h}(\eta^{2}u_{h})\big)\,d{\bm s}.}
\end{array}
\end{equation}
In what follows, we bound the right-hand side of \cref{eq:6-4} term by term. Using \cref{eq:5-13} and an inverse estimate (local to each $K$), we obtain
\begin{equation}\label{eq:6-5}
\begin{array}{lll}
{\displaystyle \mathcal{A}_{\omega^{\ast}}(u_{h}, \eta^{2}u_{h}-I_{h}(\eta^{2}u_{h}))}\\[2mm]
{\displaystyle \leq C\sum_{K\cap \omega^{\ast} \neq \emptyset}h_{K}\Vert \nabla u_{h}\Vert_{L^{2}(K)} \big(\delta^{-1} \Vert\nabla (\eta u_{h})\Vert_{L^{2}(K)} + \delta^{-2} \Vert u_{h}\Vert_{L^{2}(K)}\big)}\\[4mm]
{\displaystyle \leq C\delta^{-2}\sum_{K\cap \omega^{\ast} \neq \emptyset}  \Vert u_{h}\Vert^{2}_{L^{2}(K)} + \frac{a_{\rm min}}{6} \sum_{K\cap \omega^{\ast} \neq \emptyset} \Vert\nabla (\eta u_{h})\Vert^{2}_{L^{2}(K)}}\\[3mm]
{\displaystyle \leq C\delta^{-2}\Vert u_{h}\Vert^{2}_{L^{2}(\omega^{\ast})} + \frac{1}{6}\Vert \eta u_{h}\Vert^{2}_{\mathcal{A},\omega^{\ast}}.}
\end{array}
\end{equation}
Similarly, we can use \cref{eq:5-14} and the assumption that $kh_{K}\leq 1$ to get
\begin{equation}\label{eq:6-6}
\begin{array}{lll}
{\displaystyle k^{2}\int_{\omega^{\ast}}V^{2}u_{h}\big(\eta^{2}u_{h}-I_{h}(\eta^{2}u_{h})\big)\,d{\bm x}}\\[3mm]
{\displaystyle \leq Ck^{2}\sum_{K\cap \omega^{\ast} \neq \emptyset}h^{2}_{K}\Vert u_{h}\Vert_{L^{2}(K)} \big(\delta^{-1} \Vert\nabla (\eta u_{h})\Vert_{L^{2}(K)} + \delta^{-2} \Vert u_{h}\Vert_{L^{2}(K)}\big)}\\[3mm]
{\displaystyle \leq C\delta^{-2}\Vert u_{h}\Vert^{2}_{L^{2}(\omega^{\ast})} + \frac{1}{6}\Vert \eta u_{h}\Vert^{2}_{\mathcal{A},\omega^{\ast}}.}
\end{array}
\end{equation}
It remains to estimate the last term. Observe that
\begin{equation}\label{eq:6-7}
\begin{array}{lll}
{\displaystyle \Big|{\rm i}k\int_{\partial \omega^{\ast}\cap \Gamma_{R}}\beta u_{h}\big(\eta^{2}u_{h}-I_{h}(\eta^{2}u_{h})\big)\,d{\bm s}\Big| }\\[3mm]
{\displaystyle \; \leq k \,\Vert \beta\Vert_{L^{\infty}(\Gamma_R)} \sum_{K\cap \omega^{\ast} \neq \emptyset}\int_{\partial K} \big|u_{h}\big(\eta^{2}u_{h}-I_{h}(\eta^{2}u_{h}\big)\big|\,d{\bm s}.}
\end{array}
\end{equation}
Using the multiplicative trace inequality \cref{eq:5-15} yields that
\begin{equation}\label{eq:6-8}
\begin{array}{lll}
{\displaystyle \int_{\partial K} \big|u_{h}\big(\eta^{2}u_{h}-I_{h}(\eta^{2}u_{h})\big)\big|\,d{\bm s} \leq \Vert u_{h}\Vert_{L^{2}(\partial K)} \,\Vert \eta^{2}u_{h}-I_{h}(\eta^{2}u_{h})\Vert_{L^{2}(\partial K)}}\\[3mm]
{\displaystyle \quad \leq C\big( h_{K}^{-1/2}\Vert u_{h}\Vert_{L^{2}(K)}+\Vert u_{h}\Vert^{1/2}_{L^{2}(K)}\Vert \nabla u_{h}\Vert^{1/2}_{L^{2}(K)}\big)\big(h^{-1/2}_{K}\Vert \eta^{2}u_{h}-I_{h}(\eta^{2}u_{h})\Vert_{L^{2}(K)} }\\[2mm]
{\displaystyle \qquad +\, \Vert \eta^{2}u_{h}-I_{h}(\eta^{2}u_{h})\Vert^{1/2}_{L^{2}(K)}\,\Vert \nabla(\eta^{2}u_{h}-I_{h}(\eta^{2}u_{h}))\Vert^{1/2}_{L^{2}(K)} \big).}
\end{array}
\end{equation}
Applying \cref{lem:5-2} to \cref{eq:6-8} and using a similar argument as in \cref{eq:6-5}, it can be proved that
\begin{equation}\label{eq:6-9}
\Vert \beta\Vert_{L^{\infty}(\Gamma_R)}\int_{\partial K} \big|u_{h}\big(\eta^{2}u_{h}-I_{h}(\eta^{2}u_{h})\big)\big|\,d{\bm s} \leq h_{K}\big(C\delta^{-2}\Vert u_{h}\Vert^{2}_{L^{2}(K)} + \frac{1}{6}\Vert \eta u_{h}\Vert^{2}_{\mathcal{A},K}\big).
\end{equation}
Inserting \cref{eq:6-9} into \cref{eq:6-7} and noting that $kh_{K}\leq 1$, we get
\begin{equation}\label{eq:6-10}
\Big|{\rm i}k\int_{\partial \omega^{\ast}\cap \Gamma_{R}}Vu_{h}\big(\eta^{2}u_{h}-I_{h}(\eta^{2}u_{h})\big)\,d{\bm s}\Big|\leq C\delta^{-2}\Vert u_{h}\Vert^{2}_{L^{2}(\omega^{\ast})} + \frac{1}{6}\Vert \eta u_{h}\Vert^{2}_{\mathcal{A},\omega^{\ast}}.
\end{equation}
Collecting the estimates \cref{eq:6-5,eq:6-6,eq:6-10} and recalling \cref{eq:6-4}, we arrive at
\begin{equation}\label{eq:6-11}
\big| \mathcal{B}_{\omega^{\ast}}(u_{h}, \eta^{2}u_{h}) \big|\leq  C\delta^{-2}\Vert u_{h}\Vert^{2}_{L^{2}(\omega^{\ast})} + \frac{1}{2}\Vert \eta u_{h}\Vert^{2}_{\mathcal{A},\omega^{\ast}},
\end{equation}
which, combining with \cref{eq:6-2,eq:6-3}, gives \cref{eq:6-1}.
\end{proof}
\begin{corollary}\label{cor:6-1}
Let $\omega$ and $\omega^{\ast}$ satisfy the same assumptions as in \cref{lem:6-1} and let $h\leq \min\{\frac{1}{2}\delta, k^{-1}\}$. Assume that $\eta \in W^{1,\infty}(\omega^{\ast})$ satisfying $\Vert\eta\Vert_{L^{\infty}(\omega^{\ast})}\leq 1$ and $supp\,(\eta)\subset \overline{\omega}$. Then, for each $u_{h}\in H_{h,\mathcal{B}}(\omega^{\ast})$,
\begin{equation}\label{eq:6-12}
\Vert I_{h}(\eta u_{h})\Vert_{\mathcal{A},\omega^{\ast},k}\leq  C\big(\delta^{-1} + kV_{\rm max} + \Vert \nabla \eta\Vert_{L^{\infty}(\omega^{\ast})} \big)\Vert u_{h} \Vert_{L^{2}(\omega^{\ast})},
\end{equation}
where $C$ depends on the bounds of the coefficients and the shape regularity of the mesh.
\end{corollary}
\begin{proof}
Using the stability of the interpolation operator and the assumption that $\Vert\eta\Vert_{L^{\infty}(\omega^{\ast})}\leq 1$, we have 
\begin{equation}\label{eq:6-13}
\Vert I_{h}(\eta u_{h})\Vert_{\mathcal{A},\omega^{\ast},k} \leq C \Vert \eta u_{h}\Vert_{\mathcal{A},\omega^{\ast},k}\leq C\big( \Vert \eta u_{h}\Vert_{\mathcal{A},\omega^{\ast}} + kV_{\rm max}\Vert u_{h}\Vert_{L^{2}(\omega^{\ast})}\big).
\end{equation}
Next we use the triangle inequality, the assumptions on $\eta$, and \cref{lem:6-1} to obtain
\begin{equation}\label{eq:6-14}
\begin{array}{lll}
{\displaystyle \Vert \eta u_{h}\Vert_{\mathcal{A},\omega^{\ast}}\leq \Vert A^{1/2} u_{h}\nabla \eta \Vert_{L^{2}(\omega^{\ast})} + \Vert \eta\Vert_{L^{\infty}(\omega^{\ast})}\Vert u_{h}\Vert_{\mathcal{A},\omega} }\\[2mm]
{\displaystyle \qquad \qquad \quad \leq \big(a_{\rm max}^{1/2}\Vert \nabla \eta\Vert_{L^{\infty}(\omega^{\ast})} + C\delta^{-1} + \sqrt{2}kV_{\rm max}\big)\Vert u_{h}\Vert_{L^{2}(\omega^{\ast})}.}
\end{array}
\end{equation}
Inserting \cref{eq:6-14} into \cref{eq:6-13} completes the proof of \cref{eq:6-12}.
\end{proof}

Now we can give upper bounds for the local approximation errors. The result is very similar to that in \cref{thm:3-1} and hence, for the sake of brevity, we only state it for the case where $\omega$ and $\omega^{\ast}$ are general domains. 
\begin{theorem}\label{thm:6-1}
\cm{Let $\delta^{\ast} = {\rm dist}(\omega,\,\partial \omega^{\ast}\setminus \partial \Omega)>0$, and let $\sigma = k\delta^{\ast}V_{\rm max}/(2a^{1/2}_{\rm max})$. There exist $n_{0}>0$ and $b>0$ independent of $h$ and $k$, such that for any $n>n_{0}$, if $h\leq \min\{k^{-1},\,\delta^{\ast}/(4bn^{1/(d+1)})\}$, then
\begin{equation}\label{eq:6-15}
d_{h,n}(\omega, \omega^{\ast})\leq e^{\sqrt{2}\sigma}e^{-bn^{1/(d+1)}}.
\end{equation}}
\vspace{-3ex}
\end{theorem}

The proof of \cref{thm:6-1} follows the same lines as that of \cref{thm:3-1}, by first combining the discrete Caccioppoli inequality and \cref{lem:3-1} to establish a similar approximation result as \cref{lem:3-2}, and then applying this approximation result recursively on a family of nested domains. The details are omitted here.

Before giving the global error estimates for the method, let us recall the stability constants $\widetilde{C}_{\rm stab}(k)$ and $\widetilde{C}^{i}_{{\rm stab}}(k)$ defined in \cref{ass:5-1}. As in \cref{eq:4-1}, we define
\begin{equation}\label{eq:6-16}
\begin{array}{lll}
{\displaystyle {d}_{h,\mathtt{max}} = \max_{i=1,\cdots,M}d_{h,n_i}(\omega_i,\omega_{i}^{\ast}),\quad \widetilde{C}_{\mathtt{max}}(k) = \max_{i=1,\cdots,M} \widetilde{C}^{i}_{{\rm stab}}(k).}
\end{array}
\end{equation}
The global approximation error of the discrete MS-GFEM is given in the following lemma, which can be proved using exactly the same technique as in \cref{lem:4-1}.
\begin{lemma}\label{lem:6-3}
Let $u^{\mathdutchcal{e}}_{h}$ be the solution of the discrete problem \cref{eq:5-1} and let $u_{h}^{p}$ and $S_{h}(\Omega)$ be the global particular function and the trial space of the discrete MS-GFEM. Then, there exists a $\varphi_{h}\in u^{p}_{h} + S_{h}(\Omega)$ such that
\begin{equation}\label{eq:6-17}
\Vert u_{h}^{\mathdutchcal{e}}-\varphi_{h} \Vert_{\mathcal{A},k}\leq \sqrt{2\zeta \zeta^{\ast}} \,{d}_{h,\mathtt{max}} \big(\widetilde{C}_{\rm stab}(k) + \sqrt{2} \widetilde{C}_{\mathtt{max}}(k)\big) (\Vert f\Vert_{L^{2}(\Omega)} + \Vert g\Vert_{L^{2}(\Gamma_{R})}).
\end{equation}
\end{lemma}
Arguing as in the proof of \cref{thm:4-1,lem:4-3}, we can get a similar quasi optimality for the discrete method as \cref{cor:4-1}. Before stating the result, we recall the constants $\Xi$ defined in \cref{eq:4-5-16} and $H^{\ast}_{\mathtt{max}}$ defined in \cref{eq:4-1}.
\begin{theorem}\label{thm:6-2}
Let $u_{h}^{\mathdutchcal{e}}$ be the solution of the discrete problem \cref{eq:5-1} and $u_{h}^{G}$ be the discrete MS-GFEM approximation, and \cm{let ${\rm diam}(\omega_{i}^{\ast}) \leq 2\,{\rm diam}(\omega_{i})$ for each $i=1,\ldots,M$}. Suppose that
\begin{equation}\label{eq:6-18}
d_{h,\mathtt{max}} \leq \big(8\sqrt{2}kC_{\rm stab}(k)\,\Xi\big)^{-1},\quad  \cm{H^{\ast}_{\mathtt{max}} \leq \big(36kC_{\chi}C_{I}C^{2}_{P}a^{1/2}_{\rm max} a^{-1}_{\rm min}\,\Xi\big)^{-1}},
\end{equation}
where $C_{I}>0$ depends on the stability of the interpolation operator $I_{h}$. Then,
\begin{equation}\label{eq:6-19}
\Vert u_{h}^{\mathdutchcal{e}}-u_{h}^{G}\Vert_{\mathcal{A},k} \leq 2C_{\mathcal{B}}\inf_{\varphi_{h}\in u_{h}^{p} + S_{h}(\Omega)}\Vert u_{h}^{\mathdutchcal{e}}-\varphi_{h} \Vert_{\mathcal{A},k}.
\end{equation}
\end{theorem}

\subsection{Convergence of the eigenvalues} As we have seen in the preceding sections, the $n$-widths at the continuous and discrete levels are given by the square roots of the $(n+1)$-th eigenvalue of the continuous and discrete eigenproblems \cref{eq:2-15,eq:5-6}, respectively. In this subsection, we will prove the convergence of the discrete $n$-widths to the continuous ones as $h\rightarrow 0$ by proving the convergence of the eigenvalues. To simplify notation, we omit the subscript $i$ in the presentation. To start with, we recall the operators $P$ and $P_{h}$ defined in \cref{eq:2-12,eq:5-3-0}, and define $T=P^{\ast}P:H_{\mathcal{B}}(\omega^{\ast})\rightarrow H_{\mathcal{B}}(\omega^{\ast})$ such that for each $u\in H_{\mathcal{B}}(\omega^{\ast})$, $Tu\in H_{\mathcal{B}}(\omega^{\ast})$ satisfies
\begin{equation}\label{eq:7-1}
\mathcal{A}_{\omega^{\ast}}(Tu, v) = \mathcal{A}_{\omega,k}(\chi u, \chi v),\quad \forall v\in H_{\mathcal{B}}(\omega^{\ast}),
\end{equation}
and $T_{h}=P_{h}^{\ast}P_{h}: H_{h,\mathcal{B}}(\omega^{\ast})\rightarrow H_{h,\mathcal{B}}(\omega^{\ast})$ such that for each $u_{h}\in H_{h,\mathcal{B}}(\omega^{\ast})$, $T_{h}u_{h}\in H_{h,\mathcal{B}}(\omega^{\ast})$ satisfies
\begin{equation}\label{eq:7-2}
\mathcal{A}_{\omega^{\ast}}(T_{h}u_{h}, \cm{v_{h}}) = \mathcal{A}_{\omega,k}(I_{h}(\chi u_{h}), I_{h}(\chi v_{h})),\quad \forall v_{h}\in H_{h,\mathcal{B}}(\omega^{\ast}).
\end{equation}
$T$ and $T_{h}$ ($0< h\leq 1$) are self-adjoint, positive, and compact operators. The continuous and discrete eigenproblems \cref{eq:2-15,eq:5-6} can be formulated as the following spectral problems for the operators $T$ and $T_{h}$:
\begin{equation}\label{eq:7-3}
\begin{array}{lll}
{\displaystyle u_{j}\in H_{\mathcal{B}}(\omega^{\ast}), \quad Tu_{j} = \lambda_{j} u_{j},\quad j=1,2,\cdots,}\\[2mm]
{\displaystyle \lambda_{1}\geq \lambda_{2} \geq\cdots\lambda_{j}\geq\cdots,\quad \lambda_{j}>0,}
\end{array}
\end{equation}
and 
\begin{equation}\label{eq:7-4}
\begin{array}{lll}
{\displaystyle u_{h,j}\in H_{h,\mathcal{B}}(\omega^{\ast}), \quad T_{h}u_{h,j} = \lambda_{h,j} u_{h,j},\quad j=1,2,\cdots,}\\[2mm]
{\displaystyle \lambda_{h,1}\geq \lambda_{h,2} \geq\cdots\lambda_{h,j}\geq\cdots,\quad \lambda_{h,j}>0,}
\end{array}
\end{equation}
where the eigenvalues are repeated according to their multiplicities. In order to prove convergence of the discrete eigenvalues, we use an abstract theoretical framework developed in \cite{jikov2012homogenization} which requires some assumptions on the associated operators and function spaces.
\begin{assumption}\label{ass:7-1}
The operators $T$, $T_{h}$ and the spaces $H_{\mathcal{B}}(\omega^{\ast})$, $H_{h,\mathcal{B}}(\omega^{\ast})$ satisfy the following conditions:
\begin{itemize}
\item[$\mathbf{A1.}$] \emph{There exist continuous linear operators $R_{h}:H_{\mathcal{B}}(\omega^{\ast})\rightarrow H_{h,\mathcal{B}}(\omega^{\ast})$ satisfying
\begin{equation}\label{eq:7-5}
\Vert R_{h}u\Vert_{\mathcal{A},\omega^{\ast}}\leq c_{0} \Vert u\Vert_{\mathcal{A},\omega^{\ast}}, \quad \forall u\in H_{\mathcal{B}}(\omega^{\ast}),
\end{equation}
where the constant $c_{0}$ is independent of $h$; moreover, for any $u,\, v\in H_{\mathcal{B}}(\omega^{\ast})$ and $u_{h},\, v_{h}\in H_{h,\mathcal{B}}(\omega^{\ast})$ satisfying 
\begin{equation}\label{eq:7-6}
\lim_{h\rightarrow 0}\Vert u_{h} - R_{h}u\Vert_{\mathcal{A},\omega^{\ast}} = 0,\quad \lim_{h\rightarrow 0}\Vert v_{h} - R_{h}v\Vert_{\mathcal{A},\omega^{\ast}} = 0,
\vspace{-1ex}
\end{equation}
it holds that}
\vspace{-1ex}
\begin{equation}\label{eq:7-7}
\lim_{h\rightarrow0}\mathcal{A}_{\omega^{\ast}}(u_{h}, v_{h}) = \mathcal{A}_{\omega^{\ast}}(u, v).
\end{equation}

\item[$\mathbf{A2.}$] \emph{The operators $T_{h}$ and $T$ are self-adjoint, positive, and compact, and the norms $\Vert T_{h}\Vert=\Vert T_{h}\Vert_{\mathcal{L}(H_{h,\mathcal{B}}(\omega^{\ast}))}$ are uniformly bounded with respect to $h$.}

\vspace{2ex}
\item[$\mathbf{A3.}$] \emph{If $\psi_{h}\in H_{h,\mathcal{B}}(\omega^{\ast})$, $\psi\in H_{\mathcal{B}}(\omega^{\ast})$ and 
\begin{equation}\label{eq:7-8}
\lim_{h\rightarrow 0}\Vert \psi_{h} - R_{h}\psi \Vert_{\mathcal{A},\omega^{\ast}} = 0,
\vspace{-1ex}
\end{equation}
then}
\vspace{-1ex}
\begin{equation}\label{eq:7-9}
\lim_{h\rightarrow 0}\Vert T_{h}\psi_{h} - R_{h}T\psi \Vert_{\mathcal{A},\omega^{\ast}} = 0.
\end{equation}

\item[$\mathbf{A4.}$] \emph{For any sequence $\psi_{h}\in H_{h,\mathcal{B}}(\omega^{\ast})$ with $\sup_{h\in (0,1]}\Vert \psi_{h} \Vert_{\mathcal{A},\omega^{\ast}}<\infty$, there exists a subsequence $\psi_{h^{\prime}}$ and a function $u\in H_{\mathcal{B}}(\omega^{\ast})$ such that 
\begin{equation}\label{eq:7-10}
\Vert T_{h^{\prime}}\psi_{h^{\prime}} - R_{h^{\prime}}u \Vert_{\mathcal{A},\omega^{\ast}}\rightarrow 0\quad as \quad h^{\prime}\rightarrow 0.
\end{equation}
}
\end{itemize}
\end{assumption}

By \cite[Lemma 11.3 \,\&\, Theorem 11.4]{jikov2012homogenization}, the following theorem holds true.
\begin{theorem}\label{thm:7-1}
Let $\{\lambda_{j}\}$ and $\{\lambda_{h,j}\}$ be the eigenvalues of problems \cref{eq:7-3,eq:7-4}, respectively. Assume that \cref{ass:7-1} holds true. Then, for each $j=1,2,\cdots$, $\lambda_{h,j}\rightarrow\lambda_{j}$ as $h\rightarrow 0$. Moreover, for sufficiently small $h$,
\begin{equation}\label{eq:7-12}
|\lambda_{h,j} - \lambda_{j}|\leq 2\sup_{u\in N(\lambda_{j}, T)}\Vert T_{h}R_{h} u - R_{h}Tu\Vert_{\mathcal{A},\omega^{\ast}},\quad j=1,2,\cdots,
\vspace{-1ex}
\end{equation}
where $N(\lambda_{j}, T)$ is the normalized eigenspace of $T$ corresponding to the eigenvalue $\lambda_{j}$:
\begin{equation}\label{eq:7-13}
N(\lambda_{j}, T) = \big\{u\in H_{\mathcal{B}}(\omega^{\ast})\,:\, \Vert u\Vert_{\mathcal{A},\omega^{\ast}}=1,\;\; Tu = \lambda_{j} u \big\}.
\end{equation}
\end{theorem}
\begin{rem}
\cm{We also have the following approximation result for the eigenvectors (see \cite[Theorem 11.5]{jikov2012homogenization}), which provides an error estimate for the approximation of the local spaces. Let $\{\lambda_{j}\}$ be the eigenvalues of problems \cref{eq:7-3}, and $N(\lambda_{j}, T)$ be as in \cref{eq:7-13}. Assume that $j\geq 1,\,s\geq 1$ are integers, and that
\begin{equation}
 \lambda_{j-1}>\lambda_{j}=\cdots=\lambda_{j+s-1}>\lambda_{j+s},   
\end{equation}
i.e., the multiplicity of the eigenvalue $\lambda_{j}$ is $s$. Then for any $u\in N(\lambda_{j}, T)$, there exists a linear combination $u_{h}$ of eigenvectors $u_{h,j},\ldots,u_{h,j+s-1}$ of problem \cref{eq:7-4} such that
\begin{equation}\label{eq:7-13-0}
\Vert u_{h} - R_{h}u  \Vert_{\mathcal{A},\omega^{\ast}} \leq C_{j}\Vert T_{h}R_{h} u - R_{h}Tu\Vert_{\mathcal{A},\omega^{\ast}},
\end{equation}
where $C_{j}$ is independent of $h$. The error in approximating the local spaces is roughly bounded by the sum of the right-hand term of \cref{eq:7-13-0} and the "projection" error.
}
\end{rem}

By \cref{thm:7-1}, in order to prove the convergence of the eigenvalues, it suffices to prove that conditions ${\rm A1}\,$--$\,{\rm A4}$ are satisfied. In fact, we have
\begin{theorem}\label{thm:7-2}
Under \cref{ass:5-3}, conditions ${\rm A1}\,$--$\,{\rm A4}$ are satisfied.
\end{theorem}
\begin{proof}
The verification of conditions ${\rm A1}\,$--$\,{\rm A4}$ shares some similarites with that for positive definite problems in \cite{ma2021error} and thus we omit some details of the proof which can be found in \cite{ma2021error}. We start with the verification of condition ${\rm A1}$. Let $h$ be sufficiently small. We define $R_{h} = \Pi_{h}|_{H_{\mathcal{B}}(\omega^{\ast})}$, where $\Pi_{h}$ is the projection defined in \cref{eq:5-22}. It follows from the definition of $\Pi_{h}$ that $R_{h}u\in H_{h,\mathcal{B}}(\omega^{\ast})$ for any $u\in H_{\mathcal{B}}(\omega^{\ast})$. By \cref{lem:5-4} and \cref{eq:2-3-0}, we see that \cref{eq:7-5} holds and that
\begin{equation}\label{eq:7-14}
\Vert R_{h}u - u\Vert_{\mathcal{A},\omega^{\ast}}\rightarrow 0 \quad {\rm as}\quad  h\rightarrow 0.
\end{equation}
Therefore, the first part of condition ${\rm A1}$ is proved and the other part can be proved by using \cref{eq:7-14} and the same technique used in the proof of Theorem 4.2 in \cite{ma2021error}. Condition ${\rm A2}$ follows directly from the definition of operators $T$ and $T_{h}$. To verify condition ${\rm A3}$, we first extend the definition of $T_{h}$ to all of $H_{D}^{1}(\omega^{\ast})$ and observe that
\begin{equation}\label{eq:7-15}
\begin{array}{lll}
{\displaystyle \Vert T_{h}\psi_{h} - R_{h}T\psi \Vert_{\mathcal{A},\omega^{\ast}} \leq \Vert T_{h}(\psi_{h} - R_{h}\psi) \Vert_{\mathcal{A},\omega^{\ast}} + \Vert T_{h}(R_{h}\psi - \psi) \Vert_{\mathcal{A},\omega^{\ast}} }\\[3mm]
{\displaystyle \qquad +\,\Vert R_{h}T\psi - T\psi \Vert_{\mathcal{A},\omega^{\ast}} +\Vert T_h\psi - T\psi \Vert_{\mathcal{A},\omega^{\ast}}.}
\end{array}
\end{equation}
The uniform boundedness of $\Vert T_{h} \Vert$, \cref{eq:7-8}, and \cref{eq:7-14} yield that
\begin{equation}\label{eq:7-16}
\lim_{h\rightarrow0}\big(\Vert T_{h}(\psi_{h} - R_{h}\psi) \Vert_{\mathcal{A},\omega^{\ast}} + \Vert T_{h}(R_{h}\psi - \psi) \Vert_{\mathcal{A},\omega^{\ast}} +\Vert R_{h}T\psi - T\psi \Vert_{\mathcal{A},\omega^{\ast}} \big) = 0.
\end{equation}
It remains to show that $\Vert T_h\psi - T\psi \Vert_{\mathcal{A},\omega^{\ast}}\rightarrow 0$ as $h\rightarrow 0$. To this end, we consider an auxiliary problem of finding $\widetilde{T}_{h}\psi\in H_{h,\mathcal{B}}(\omega^{\ast})$ such that
\begin{equation}\label{eq:7-17}
\mathcal{A}_{\omega^{\ast}}(\widetilde{T}_{h}\psi, v_{h}) = \mathcal{A}_{\omega,k}(\chi \psi, \chi v_{h}),\quad \forall v_{h}\in H_{h,\mathcal{B}}(\omega^{\ast}).
\end{equation}
Rewriting problems \cref{eq:7-1,eq:7-17} as saddle point problems as in \cite{ma2021error} and using \cref{lem:5-5,lem:5-6}, we see that
\begin{equation}\label{eq:7-18}
\Vert \tilde{T}_h\psi - T\psi \Vert_{\mathcal{A},\omega^{\ast}}\rightarrow 0\quad {\rm as}\quad  h\rightarrow 0.
\end{equation}
Next we will prove that $\Vert \tilde{T}_h\psi - T_{h}\psi \Vert_{\mathcal{A},\omega^{\ast}}\rightarrow 0$. Let $e_{h} = \tilde{T}_h\psi - T_{h}\psi$. Subtracting \cref{eq:7-2} from \cref{eq:7-17} and using \cref{eq:2-3-0}, we see that $e_{h}$ satisfies
\begin{equation}\label{eq:7-19}
\Vert e_{h}\Vert^{2}_{\mathcal{A},\omega^{\ast}}\leq C\Vert \chi\psi - I_{h}(\chi\psi)\Vert^{2}_{\mathcal{A},\omega} + C\Vert \chi e_{h} -I_{h}(\chi e_{h})\Vert_{\mathcal{A},\omega}.
\end{equation}
Moreover, there exists $C_{0}>0$ such that $\Vert e_{h}\Vert_{H^{1}(\omega^{\ast})}\leq C_{0}$ for all $h$. Given $C_{1}>0$, we define a subset of $H^{1}_{DI}(\omega^{\ast})$:
\begin{equation}\label{eq:7-20}
\mathcal{S} = \big\{\chi u\,:\,u\in H^{1}_{D}(\omega^{\ast}),\;\; \Vert u\Vert_{H^{1}(\omega^{\ast})}\leq C_{0},\;\; \Vert u\Vert_{\mathcal{A},\omega}\leq C_{1}\Vert u\Vert_{L^{2}(\omega^{\ast})} \big\}.
\end{equation}
Since $e_{h}\in H_{h,\mathcal{B}}(\omega^{\ast})$, the discrete Caccioppoli inequality \cref{eq:6-1} implies that there exists $C_{1}>0$ such that $\chi e_{h}\in \mathcal{S}$ for all $h$. Furthermore, it follows from \cref{eq:6-14} and Rellich's theorem that $\mathcal{S}$ is a compact subset in $H^{1}_{DI}(\omega)$. Hence, we can use \cref{lem:5-7} to assert that for any $\varepsilon>0$, there exists $h_{0}>0$, such that if $0<h<h_{0}$, there exists $v_{h}\in U_{h,DI}(\omega^{\ast})$ satisfying $\Vert \chi e_{h} - v_{h}\Vert_{H^{1}(\omega)}\leq C\varepsilon$, and thus
\begin{equation}\label{eq:7-21}
\Vert \chi e_{h} -I_{h}(\chi e_{h})\Vert_{\mathcal{A},\omega} = \Vert (\chi e_{h} -v_{h}) -I_{h}(\chi e_{h}-v_{h})\Vert_{\mathcal{A},\omega}\leq C\varepsilon,
\end{equation}
where we have used that fact that $I_{h}v_{h} = v_{h}$. Combining \cref{eq:7-19,eq:7-21} shows that $\Vert \tilde{T}_h\psi - T_{h}\psi \Vert_{\mathcal{A},\omega^{\ast}}\rightarrow 0$, which, together with \cref{eq:7-15,eq:7-16,eq:7-18}, give \cref{eq:7-9}.  Therefore, condition ${\rm A3}$ is verified.

Finally, we check the validity of condition ${\rm A4}$. Let $\{\psi_{h}\}$ be a sequence satisfying $\psi_{h}\in H_{h,\mathcal{B}}(\omega^{\ast})$ for each $h\in (0,1]$ and $\sup_{h\in (0,1]}\Vert \psi_{h} \Vert_{\mathcal{A},\omega^{\ast}}<\infty$. By \cref{eq:5-2-0}, we see that $\sup_{h\in (0,1]}\Vert \psi_{h} \Vert_{\mathcal{A},\omega^{\ast},k}<\infty$. Therefore, we can extract a subsequence $\{\psi_{h^{\prime}}\}$ such that $\{\psi_{h^{\prime}}\}$ converges weakly (in $ H^{1}_{D}(\omega^{\ast})$) to some $\psi\in H^{1}_{D}(\omega^{\ast})$. Applying a similar argument as in the proof of \cref{eq:5-2-0} yields that $\psi\in H_{\mathcal{B}}(\omega^{\ast})$. 
Define $u=T\psi\in H_{\mathcal{B}}(\omega^{\ast})$. It can be proved that $u$ satisfies \cref{eq:7-10} by the same argument as in the proof of Theorem 4.2 in \cite{ma2021error} using the discrete Caccioppoli inequality. Therefore, conditions ${\rm A1}\,$--$\,{\rm A4}$ are verified.
\end{proof}
\section{Numerical experiments}\label{sec-6}
In this section, we provide some numerical results to support the theoretical analysis and to demonstrate the effectiveness of the method.

\subsection{Classical Helmholtz example}\label{sec:6-1}
First we consider \cref{eq:1-1} on the unit square $\Omega=(0,1)^2$ with $\Gamma_{R} = \partial \Omega$, $A({\bm x}) =I$, $V({\bm x}) = 1$, $\beta({\bm x}) =1$, and $f({\bm x})=0$. The boundary data $g$ is chosen such that the problem admits the plane-wave solution $u^{\mathdutchcal{e}} = \exp\big({\rm i}\vec{k} \cdot {\bm x}\big)$ with $\vec{k} = k(0.6,\,0.8)$.

The underlying fine FE mesh with mesh-size $h$ on $\Omega$ 
is based on a uniform Cartesian grid. To implement the MS-GFEM, we first split the domain into $M=m^{2}$ non-overlapping subdomains resolved by the mesh, and then extend each subdomain by 2 layers of fine mesh elements to create an overlapping decomposition $\{ \omega_{i}\}_{i=1}^{M}$ of $\Omega$. Each overlapping subdomain $\omega_{i}$ is further extended by $\ell$ layers of fine mesh elements to create an oversampling domain $\omega_{i}^{\ast}$ on which the local problems are solved, i.e., $\ell h$ denotes the oversampling size. We denote by $n_{loc}$ the number of eigenvectors selected in each subdomain for building the local approximation space. 

Let $u^{\mathdutchcal{e}}_{h}$ and $u_{h}^{G}$ be the standard FE approximation and the (discrete) MS-GFEM approximation of the problem, respectively. Denote by $\mathbf{error}^{\Large \mathdutchcal{e}}$ ($\mathbf{error}$) the relative error between $u_{h}^{G}$ and the exact solution $u^{\mathdutchcal{e}}$ (resp. $u_{h}^{e}$), i.e.,
\begin{equation}
\mathbf{error}^{\Large \mathdutchcal{e}} := \frac{\Vert u^{\mathdutchcal{e}} - u_{h}^{G}\Vert_{\mathcal{A},k}}{\Vert u^{\mathdutchcal{e}}\Vert_{\mathcal{A},k}},\quad \mathbf{error} := \frac{\Vert u^{\mathdutchcal{e}}_{h} - u_{h}^{G}\Vert_{\mathcal{A},k}}{\Vert u^{\mathdutchcal{e}}_{h}\Vert_{\mathcal{A},k}}.
\end{equation}

\begin{figure}\label{fig:6-1}
\begin{center}
\includegraphics[scale=0.44] {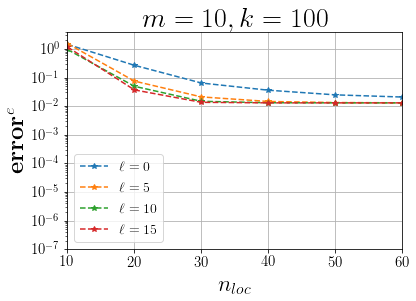}~\hspace{-1mm}
\includegraphics[scale=0.44] {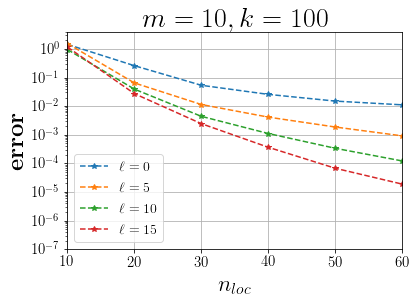}
\vspace{-4ex}
\caption{Classical Helmholtz example (subsection 6.1): Plots of $\mathbf{error}^{\mathdutchcal{e}}$ (left) and $\mathbf{error}$ (right) against $n_{loc}$ with $m=10$ and $k=100$. $\ell h$ denotes the oversampling size.}
\end{center}
\end{figure}

\begin{figure}\label{fig:6-2}
\begin{center}
\includegraphics[scale=0.44] {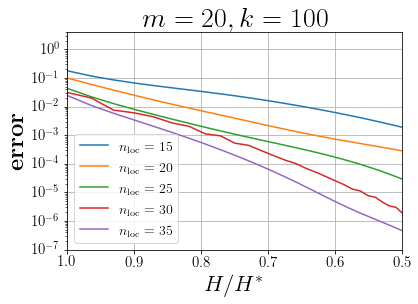}~\hspace{-1mm}
\includegraphics[scale=0.44] {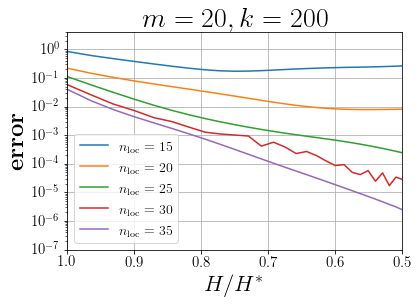}\vspace{-4ex}
\caption{Classical Helmholtz example (subsection 6.1): Plots of $\mathbf{error}$ against $H/H^{\ast}$ with $m=20$ for $k=100$ (left) and $k=200$ (right).}
\end{center}
\end{figure}

First we let $h=10^{-3}$ fixed and test our method for two wavenumbers, $k=100$ and $k=200$. In \cref{fig:6-1}, we show the decay of the errors with respect to the dimension of the local spaces for different oversampling sizes with $k=100$. \Cref{fig:6-1} (left) displays the errors between the discrete MS-GFEM approximations and the exact solution. The horizontal asymptote arises when $n_{loc}$ is sufficiently large such that the errors are dominated by the fine-scale FE approximation error. The errors between the discrete MS-GFEM approximations and the standard FE approximation are shown in \cref{fig:6-1} (right) and we can clearly see that they decay nearly exponentially with respect to $n_{loc}$, agreeing well with our theoretical analysis. In addition, it is visible that the method works even without oversampling, i.e., $\ell=0$. In \cref{fig:6-2}, we illustrate the decay of the errors with respect to $H/H^{\ast}$ (by varying the parameter $\ell$) for different dimensions of local spaces and different wavenumbers, where $H$ and $H^{\ast}$ represent the sizes of the subdomains and the oversampling domains, respectively. We can see that for $k=100$, the errors decay nearly exponentially with respect to $H/H^{\ast}$ for all different dimensions of local spaces and that for $k=200$ and a large $n_{loc}$, the errors decay similarly as in the $k=100$ case. However, for $k=200$ and a small $n_{loc}$ (15 or 20), the errors first decrease and then stagnate with increasing $H^{\ast}$. This verifies the presence of the resonance effect described in Remark~\ref{rem:3-4}.

\cm{Next we carry out a "frequency scaling" test by keeping $k^{3}h^2=1$, $H/h=0.054$, and $H/H^{\ast}=0.8$ fixed while increasing the wavenumber $k$. \Cref{fig:6-2-0} shows the errors of the method for different wavenumbers $k$,
which clearly confirm the robustness of our method with respect to a growing wavenumber.}

\begin{figure}\label{fig:6-2-0}
\begin{center}
\includegraphics[scale=0.41] {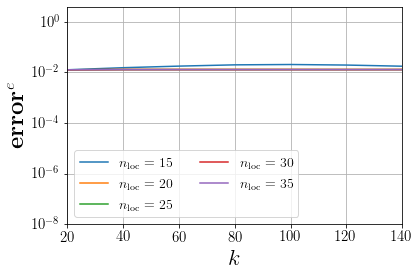}~\hspace{-1mm}
\includegraphics[scale=0.41] {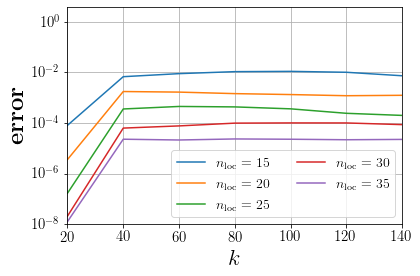}
\vspace{-2ex}
\caption{Classical Helmholtz example (subsection 6.1): Plots of $\mathbf{error}^{\mathdutchcal{e}}$ (left) and $\mathbf{error}$ (right) against $k$ with $k^{3}h^{2}=1$, $H/h=0.054$, and $H/H^{\ast}=0.8$.}
\end{center}
\end{figure}


\subsection{A scattering problem}\label{sec:6-2}
We consider the following heterogeneous scattering problem on the unit square $\Omega=(0,1)^2$ with $\Gamma_{R} = \partial \Omega$: $V({\bm x}) = 1$, $\beta({\bm x})=1$, $f({\bm x})=0$ and the coefficient $A({\bm x})$  as illustrated in \cref{fig:6-3} (left). The incident wave is taken as $u^{\rm inc} = \exp({\rm i}\vec{k}\cdot{\bm x})$ with $\vec{k} = k\big(\frac{1}{\sqrt{2}},\,-\frac{1}{\sqrt{2}}\big)$ and on $\Gamma_{R}$ we use first-order absorbing boundary conditions: $g = {\bm n}\cdot \nabla u^{\rm inc} -{\rm i}ku^{\rm inc}$.

\begin{figure}[t]
\label{fig:6-3}
\centering
\includegraphics[scale=0.19]{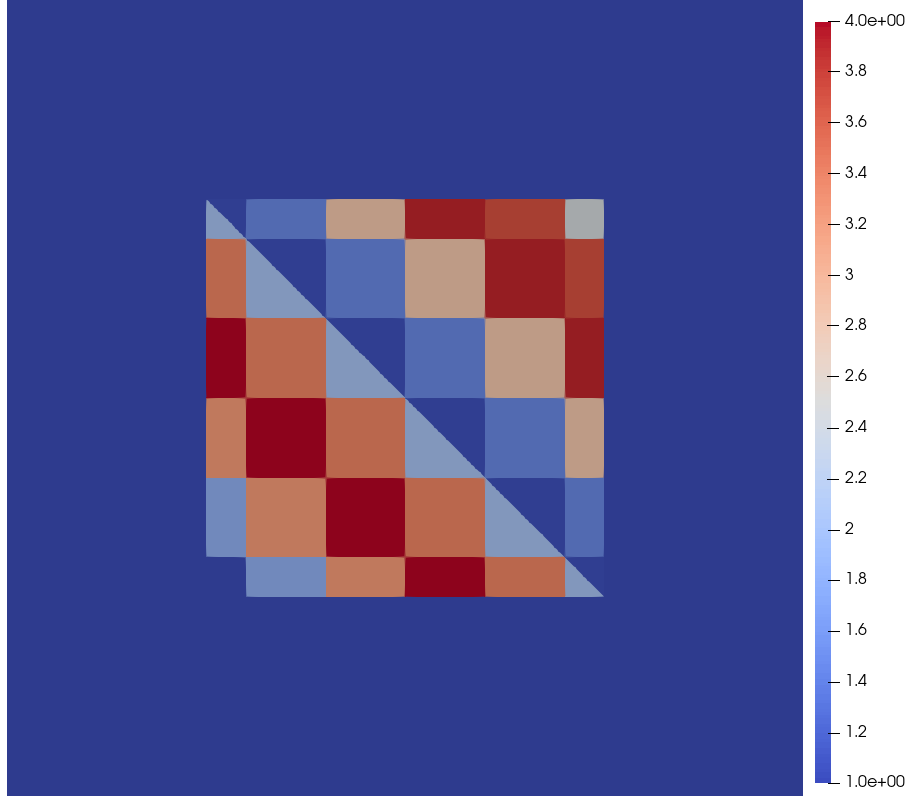}~\hspace{2ex}
\includegraphics[scale=0.19]{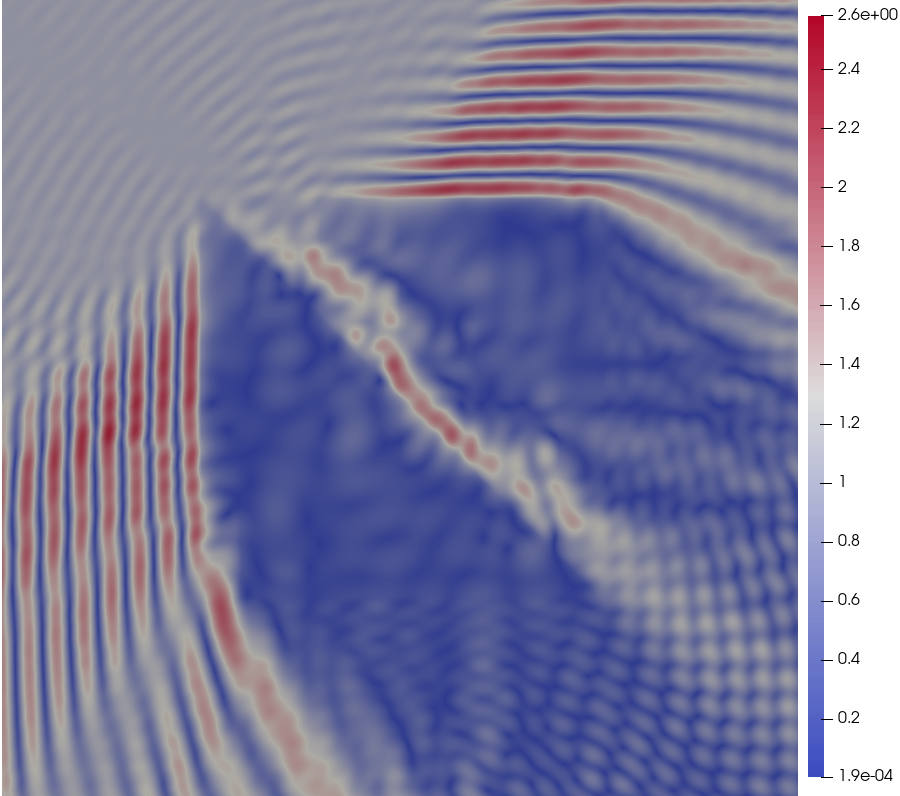}
\caption{Scattering problem (subsection 6.2): The coefficient $A({\bm x}) = a({\bm x})I$ (left) and the modulus of the standard FE solution $|u^{\mathdutchcal{e}}_{h}|$ with $ k=130$ (right).}
\end{figure}

The computational settings are similar to those in \cref{sec:6-1} except that the fine-scale FE mesh size is $h=1/1400$. The modulus of the fine-scale FE solution $u^{\mathdutchcal{e}}_{h}$ with $k=130$ is displayed in \cref{fig:6-3} (right). Since the exact solution of the problem is not available, $u^{\mathdutchcal{e}}_{h}$ is used as the reference solution for computing the errors. First we compare the decay rates of the errors with respect to $n_{loc}$ for $k=130$ and $k=260$ in \cref{fig:6-4}. It can be observed that with a fixed oversampling size, the decay rates with respect to $n_{loc}$ are roughly the same for two different wavenumbers, which confirms the theoretical analysis. In addition, as shown in \cref{fig:6-5} (left), we see that for $k=130$, the errors decay nearly exponentially with respect to $H/H^{\ast}$ and no stagnation phenomenon is observed. Finally, we illustrate the decay of the errors with respect to $m$ (the number of subdomains in one direction) for a fixed oversampling size in \cref{fig:6-5} (right). One can observe that the errors generally decay dramatically with increasing $m$. We note that in this case, the quantity $H/H^{\ast}$ decreases with increasing $m$ since the parameter $\ell$ is fixed.

\begin{figure}[t]
\label{fig:6-4}
\begin{center}
\includegraphics[scale=0.44] {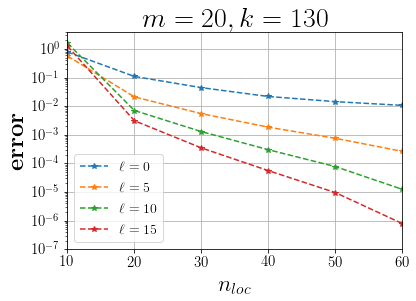}~\hspace{-1mm}
\includegraphics[scale=0.44] {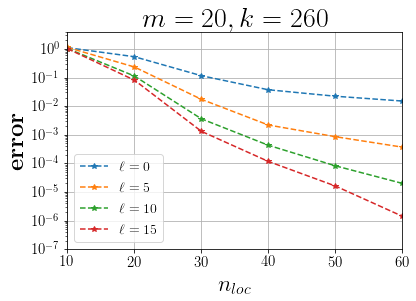}\vspace{-4ex}
\caption{Scattering problem (subsection 6.2): Plots of $\mathbf{error}$ against $n_{loc}$ with $m=20$ for $k=130$ (left) and $k=260$ (right).}
\end{center}
\end{figure}

\begin{figure}[!htbp]\label{fig:6-5}
\begin{center}
\includegraphics[scale=0.44] {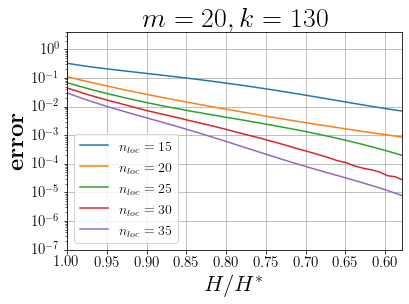}~\hspace{-1mm}
\includegraphics[scale=0.44] {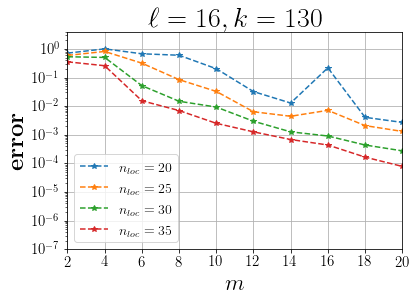}\vspace{-2ex}
\caption{Scattering problem (subsection 6.2): Plots of $\mathbf{error}$ against $H/H^{\ast}$ (left) and against $m$ (right) with $k=130$.}
\end{center}
\end{figure}

\subsection{Marmousi problem}\label{sec:6-3}
In the last example, we consider the benchmark Marmousi model \cite{versteeg1994marmousi}. The problem is posed on the domain $(0, \,9\,{\rm km})\times (-3\,{\rm km}, 0)$ and a point source is placed near the top boundary. Homogeneous Dirichlet and impedance boundary conditions are prescribed on the top surface and on the remaining three sides, respectively. Moreover, $A({\bm x})=I$, $\beta({\bm x})=kV({\bm x})$, and the velocity field $1/V({\bm x})$ is depicted in \cref{fig:6-6} (top).

\begin{figure}\label{fig:6-6}
\centering
\includegraphics[scale=0.28]{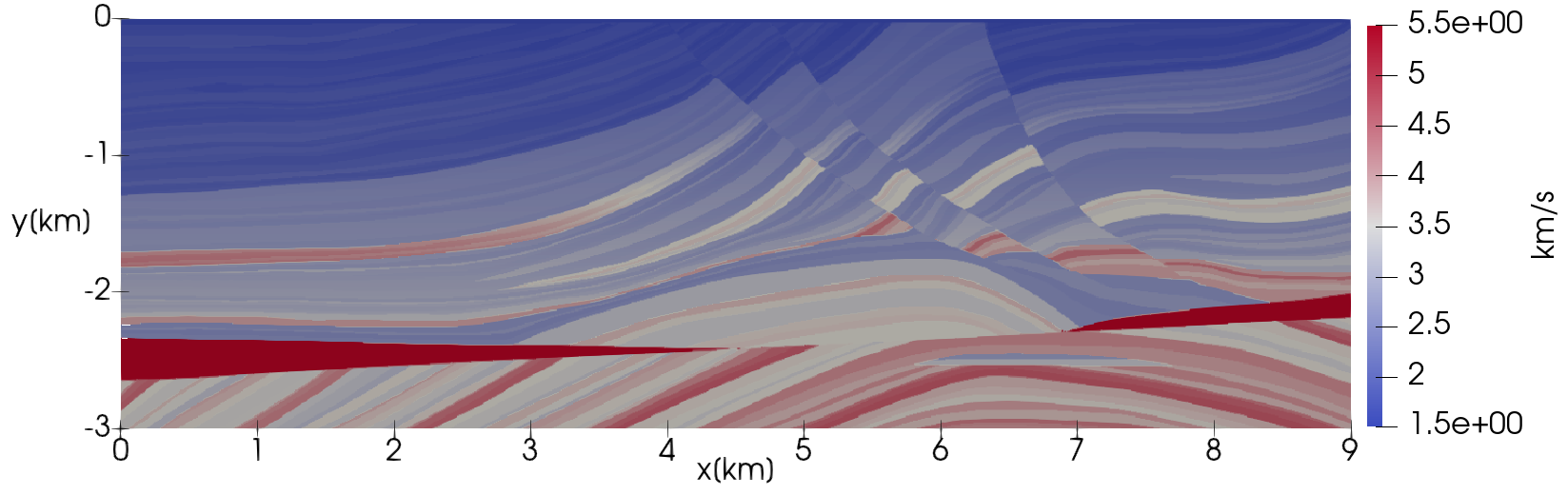}

\includegraphics[scale=0.28]{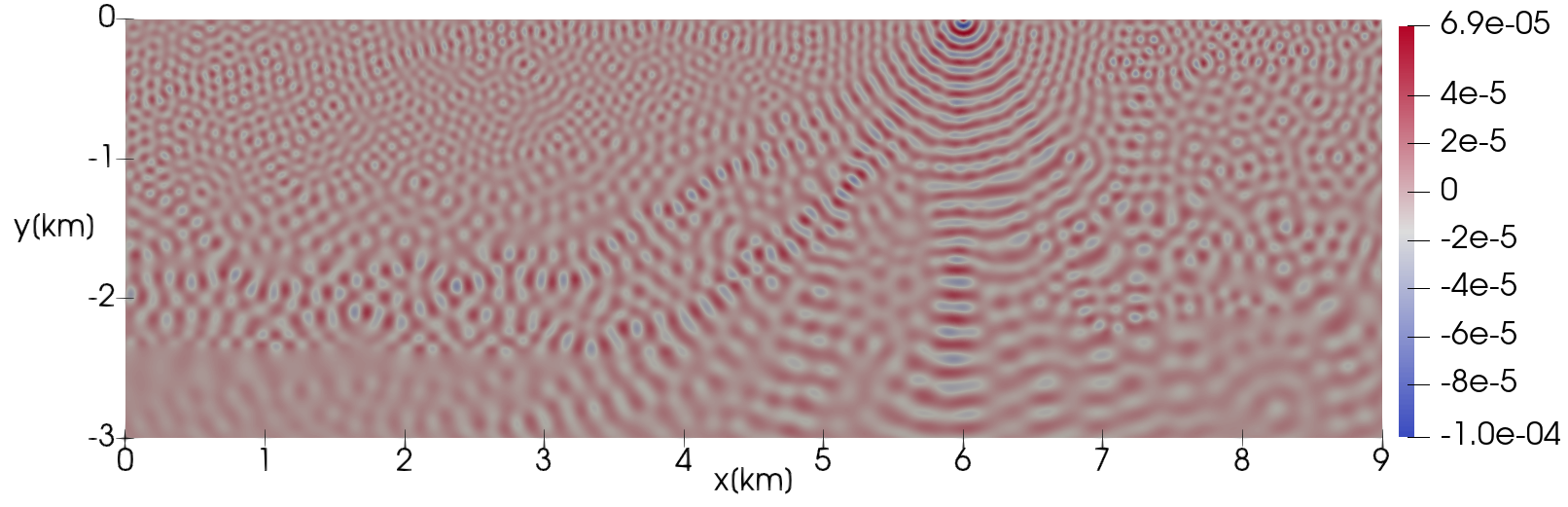}
\caption{Marmousi problem (subsection 6.3): The velocity field of the Marmousi model (top) and the real part of $u^{G}_{h}$ (bottom) computed with $m=16$, $\ell=9$, and $n_{loc}=30$.}
\end{figure}

The frequency for the test is taken as 20 Hz ($k=40\pi$) and we use 10 points per (minimal) wavelength for the FE discretization, corresponding to $h=7.5$m. In view of the geometry of the computational domain, we choose $m$ and $3m$ subdomains in the $y$ and $x$ directions, respectively. Otherwise, the computational setting is the same as in the previous subsections. The real part of the discrete MS-GFEM approximation computed using about 23000 local basis functions ($m=16$, $\ell=9$, $n_{loc}=30$) with a relative error less than $10^{-3}$ is plotted in \cref{fig:6-6} (bottom).

First we display the decay of the eigenvalues in three subdomains in a semi-logarithmic scale in \cref{fig:6-7} (left): one in the interior, one near the Dirichlet boundary and one near the impedance boundary. As predicted by the theoretical analysis, the eigenvalues in all the three subdomains decay nearly exponentially. Next we study the decay of the errors with respect to $n_{loc}$ and the results are shown in \cref{fig:6-7} (right). We see that for this realistic model with a heterogeneous velocity field, the errors of our method decay nearly exponentially with increasing $n_{loc}$ just as we have observed in the previous examples. Finally, we show the influence of the size of the subdomains by plotting the decay of the errors with respect to $H/H^{\ast}$ in \cref{fig:6-8}. It can be clearly seen that decreasing the size of the subdomains (increasing $m$) can significantly alleviate the stagnation of the errors with increasing oversampling size for a small $n_{loc}$. In fact, as noted in Remark~\ref{rem:3-4}, if $H\sim H^{\ast}\sim O(k^{-1})$, the method for Helmholtz problems behaves similarly to that for positive definite problems and the resonance effect disappears.

\begin{figure}\label{fig:6-7}
\begin{center}
\includegraphics[scale=0.44] {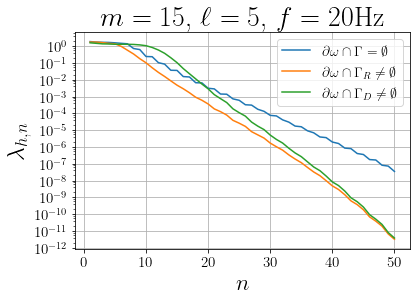}~\hspace{-1mm}
\includegraphics[scale=0.44] {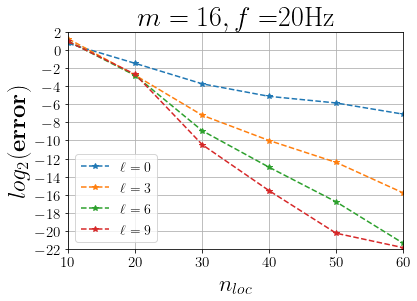}\vspace{-4ex}
\caption{Marmousi problem (subsection 6.3): Plot of $\lambda_{h,n}$ against $n$ (left); plot of $\mathbf{error}$ against $n_{loc}$ (right).}
\end{center}
\end{figure}

\begin{figure}\label{fig:6-8}
\begin{center}
\includegraphics[scale=0.44] {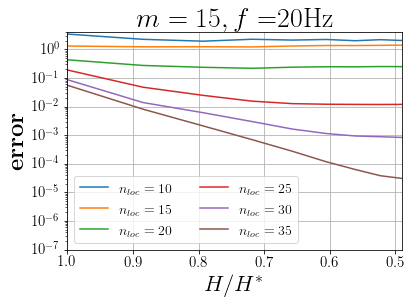}~\hspace{-1mm}
\includegraphics[scale=0.44] {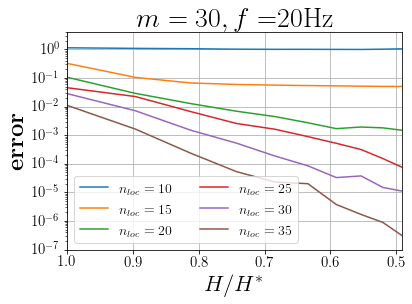}\vspace{-2ex}
\caption{Marmousi problem (subsection 6.3): Plots of $\mathbf{error}$ against $H/H^{\ast}$ for $m=15$ (left) and $m=30$ (right).}
\end{center}
\end{figure}

\section{Summary and future work}
We have performed a systematic investigation of a multiscale spectral GFEM with novel local approximation spaces for heterogeneous Helmholtz problems at the continuous and discrete level, providing a comprehensive analysis. Wavenumber explicit error estimates for the local and global approximations are obtained, and the influence of the wavenumber on the error is theoretically and numerically investigated. This goes well beyond most previous studies. Furthermore, our method provides a unified mathematical framework within which Trefftz-type discretization schemes for heterogeneous Helmholtz problems and efficient solvers for discrete Helmholtz problems can be developed and analysed.

There are two important issues of the method that remain to be addressed. The first is adaptivity. The new formalism readily facilitates an adaptive choice of the mesh size in the discretization of the local problems and of the number of eigenvectors to be included in each of the local approximation spaces. The second issue we aim to address is a discontinuous formulation. In the MS-GFEM, the local basis functions are pasted together by a partition of unity to form the continuous trial and test functions. An alternative is to use a discontinuous formulation in which the continuity of the numerical solution across neighbouring subdomains is maintained in a weak sense, just as in the ultraweak variational formulation \cite{cessenat1998application} or in the discontinuous enrichment method \cite{farhat2003discontinuous}. In this setting, we then expect the resulting linear system of the coarse problem to be better conditioned. Moreover, by combining similar local approximation spaces as in this paper with a least-squares method \cite{monk1999least}, we may expect to obtain a coercive formulation and thus get rid of the resolution conditions required for quasi optimality of the method.

\appendix
\section{Proof of \cref{thm:2-1}}
\label{sec:A.3}
For any $u,v\in H^{1}_{D}(\omega_{i}^{\ast})$, a direct calculation shows that
\begin{equation}\label{eq:2-7}
\begin{array}{lll}
{\displaystyle \int_{\omega_{i}^{\ast}}A\nabla (\eta u) \cdot \nabla (\eta\overline{v})\,d{\bm x}  = \int_{\omega_{i}^{\ast}}(A\nabla \eta \cdot \nabla \eta) u\overline{v}\,d{\bm x} - \int_{\omega_{i}^{\ast}}(A\nabla u \cdot \nabla \eta) \eta \overline{v}\,d{\bm x}}\\[4mm]
{\displaystyle \qquad \qquad +\,\int_{\omega_{i}^{\ast}}(A\nabla \eta \cdot \nabla \overline{v}) \eta u\,d{\bm x} + \int_{\omega_{i}^{\ast}}A\nabla u \cdot \nabla (\eta^{2}\overline{v})\,d{\bm x}.}
\end{array}
\end{equation}
Exchanging $u$ and $\overline{v}$ in \cref{eq:2-7}, it follows that
\begin{equation}\label{eq:2-8}
\begin{array}{lll}
{\displaystyle \int_{\omega_{i}^{\ast}}A\nabla (\eta \overline{v}) \cdot \nabla (\eta u)\,d{\bm x}  = \int_{\omega_{i}^{\ast}}(A\nabla \eta \cdot \nabla \eta) u\overline{v}\,d{\bm x} - \int_{\omega_{i}^{\ast}}(A\nabla \overline{v} \cdot \nabla \eta) \eta u\,d{\bm x}}\\[4mm]
{\displaystyle \qquad \qquad +\,\int_{\omega_{i}^{\ast}}(A\nabla \eta \cdot \nabla u) \eta \overline{v}\,d{\bm x} + \int_{\omega_{i}^{\ast}}A\nabla \overline{v} \cdot \nabla (\eta^{2}u)\,d{\bm x}.}
\end{array}
\end{equation}
Adding \cref{eq:2-7,eq:2-8} together and using the symmetry of $A$, we get
\begin{equation}\label{eq:2-9}
\begin{array}{lll}
{\displaystyle \int_{\omega_{i}^{\ast}}A\nabla (\eta u) \cdot \nabla (\eta\overline{v})\,d{\bm x}  = \int_{\omega_{i}^{\ast}}(A\nabla \eta \cdot \nabla \eta) u\overline{v}\,d{\bm x} }\\[4mm]
{\displaystyle  +\,\frac{1}{2}\,\Big(\int_{\omega_{i}^{\ast}}A\nabla u \cdot \nabla (\eta^{2}\overline{v})\,d{\bm x} +\int_{\omega_{i}^{\ast}}A\nabla \overline{v} \cdot \nabla (\eta^{2}u)\,d{\bm x}\Big).}
\end{array}
\end{equation}
By the assumptions on $\eta$, we see that for any $u,v\in H^{1}_{D}(\omega_{i}^{\ast})$, $\eta^{2}u, \eta^{2}v\in H_{DI}^{1}(\omega_{i}^{\ast})$. If, in addition, $u,v\in H_{\mathcal{B}}(\omega_{i}^{\ast})$, then we have 
\begin{equation}\label{eq:2-10}
\mathcal{B}_{\omega_{i}^{\ast}}(u,\eta^{2} v) = 0,\quad \mathcal{B}_{\omega_{i}^{\ast}}(v,\eta^{2}u) = 0.
\end{equation}
Therefore, 
\begin{equation}\label{eq:2-11}
\begin{array}{lll}
{\displaystyle \int_{\omega_{i}^{\ast}}A\nabla u \cdot \nabla (\eta^{2}\overline{v})\,d{\bm x} = k^{2}\int_{\omega_{i}^{\ast}}\eta^{2}V^{2}u\overline{v}\,d{\bm x} + {\rm i}k\int_{\partial \omega^{\ast}_{i}\cap \Gamma_{R}} \eta^{2}\beta u\overline{v}\,d{\bm s}, }\\[4mm]
{\displaystyle \int_{\omega_{i}^{\ast}}A\nabla \overline{v} \cdot \nabla (\eta^{2}u)\,d{\bm x} =  k^{2}\int_{\omega_{i}^{\ast}}\eta^{2}V^{2}u\overline{v}\,d{\bm x} - {\rm i}k\int_{\partial \omega^{\ast}_{i}\cap \Gamma_{R}} \eta^{2}\beta u\overline{v}\,d{\bm s}.}
\end{array}
\end{equation}
Inserting \cref{eq:2-11} into \cref{eq:2-9} yields \cref{eq:2-5} and the inequality \cref{eq:2-6} follows.

\section{Proof of \cref{eq:5-2-0}}
\label{sec:A.4}
It is easy to see that for any $h>0$, the $H^{1}$ seminorm is a norm on $H_{h,\mathcal{B}}(\omega^{\ast}_{i})$. Since all norms are equivalent in a finite-dimensional vector space, there exist $C_{h}>0$ depending on $h$ such that
\begin{equation}\label{eq:B.1}
\Vert u_{h}\Vert_{L^{2}(\omega_{i}^{\ast})}\leq C_{h}\Vert \nabla u_{h}\Vert_{L^{2}(\omega_{i}^{\ast})},\quad \forall u_{h}\in H_{h,\mathcal{B}}(\omega^{\ast}_{i}).
\end{equation} 
It suffices to show that $\limsup_{h\rightarrow 0} C_{h} < + \infty$. If it doesn't hold, then there exists a sequence $\{u_{h_n}\}_{n=1}^{\infty}$ with $h_{n}\rightarrow 0$ as $n\rightarrow \infty$, such that $\Vert u_{h_{n}}\Vert_{L^{2}(\omega_{i}^{\ast})}=1$ and $\Vert \nabla u_{h_n}\Vert_{L^{2}(\omega_{i}^{\ast})}\leq 1/n$ hold for any $n\geq 1$. We can find a subsequence, still denoted by $\{u_{h_n}\}_{n=1}^{\infty}$, and a function $u_{0}\in H^{1}_{D}(\omega_{i}^{\ast})$, such that $u_{h_n}\rightharpoonup u_{0}$ weakly in $H^{1}_{D}(\omega_{i}^{\ast})$. Since $H^{1}(\omega_{i}^{\ast})$ is compactly embedded into $L^{2}(\omega_{i}^{\ast})$, we see that $u_{h_n}\rightarrow u_{0}$ strongly in $L^{2}(\omega_{i}^{\ast})$. It follows that $\Vert u_{0} \Vert_{L^{2}(\omega_{i}^{\ast})} =1$ and $\nabla u_{0} = 0$. Therefore, $u_{0}$ is a constant. Next, we will show that $u_{0}\in H_{\mathcal{B}}(\omega^{\ast}_{i})$. For any fixed $v\in H^{1}_{DI}(\omega_{i}^{\ast})$, 
\begin{equation}\label{eq:B.2}
\mathcal{B}_{\omega_{i}^{\ast}}(u_{0},v) = \mathcal{B}_{\omega_{i}^{\ast}}(u_{0}-u_{h_{n}},v) + \mathcal{B}_{\omega_{i}^{\ast}}(u_{h_{n}}, I_{h_{n}}v) + \mathcal{B}_{\omega_{i}^{\ast}}(u_{h_n}, v-I_{h_n}v),
\end{equation}
where $I_{h}$ is the standard Lagrange interpolation operator. Since $\{u_{h_{n}}\}$ converges weakly to $u_{0}$ in $ H^{1}_{D}(\omega_{i}^{\ast})$, we have 
\begin{equation}\label{eq:B.3}
\mathcal{A}_{\omega_{i}^{\ast},k}(u_{0}-u_{h_{n}},v)\rightarrow 0 \quad {\rm as} \quad n\rightarrow \infty.
\end{equation}
Using the compact embedding of $H^{1}(\omega_{i}^{\ast})$ into $L^{2}(\omega_{i}^{\ast})$ and $L^{2}(\partial\omega_{i}^{\ast})$, we see that
\begin{equation}\label{eq:B.4}
\Vert u_{0}-u_{h_{n}}\Vert_{L^{2}(\omega_{i}^{\ast})}\rightarrow 0,\quad \Vert u_{0}-u_{h_{n}}\Vert_{L^{2}(\partial \omega_{i}^{\ast})}\rightarrow 0 \quad {\rm as} \quad n\rightarrow \infty.
\end{equation}
Combining \cref{eq:B.3,eq:B.4} implies that $\mathcal{B}_{\omega_{i}^{\ast}}(u_{0}-u_{h_{n}},v)\rightarrow 0$ as $n\rightarrow \infty$. Moreover, since $u_{h_n}\in H_{h_{n},\mathcal{B}}(\omega_{i}^{\ast})$ and $I_{h_{n}}v\in U_{h_{n},DI}(\omega_{i}^{\ast})$, the second term $\mathcal{B}_{\omega_{i}^{\ast}}(u_{h_{n}}, I_{h_{n}}v)$ vanishes. Finally, in view of the boundedness of the sequence $\{u_{h_{n}}\}$, we conclude
\begin{equation}
|\mathcal{B}_{\omega_{i}^{\ast}}(u_{h_{n}}, v-I_{h_{n}}v)|\leq C\Vert v-I_{h_n}v\Vert_{\mathcal{A},\omega_{i}^{\ast},k}\rightarrow 0 \quad {\rm as} \quad n\rightarrow \infty.
\end{equation}
Making $n\rightarrow \infty$ in \cref{eq:B.2} yields that $\mathcal{B}_{\omega_{i}^{\ast}}(u_{0},v)=0$ for any $v\in H^{1}_{DI}(\omega_{i}^{\ast})$. Therefore, $u_{0}\in H_{\mathcal{B}}(\omega_{i}^{\ast})$. Since $u_{0}$ is a constant, we see that $u_{0}\equiv 0$, which contradicts with the fact that $\Vert u_{0} \Vert_{L^{2}(\omega_{i}^{\ast})} =1$.  

\bibliographystyle{siamplain}
\bibliography{references}
\end{document}